\documentclass[11pt]{amsart}
\usepackage{pdflscape}
\usepackage{graphicx}
\usepackage{amscd}
\usepackage{lscape}
\usepackage{euscript}
\usepackage{amsfonts}
\usepackage{amsmath, amscd, amssymb}
\usepackage{latexsym}
\def \R{\mathbb{R}}
\def \im{\mathop{\mathrm{im}}}
\def \Z{\mathbb{Z}}
\def \Q{\mathbb{Q}}
\def \D{\mathrm{Diff}}

\def \P{\mathcal{P}}
\def \SP{\mathcal{SP}}
\def \tp{\mathop\mathrm{tp}}
\def \Tp{\mathop\mathrm{Tp}}
\def \AP{\mathcal{AP}}
\def \dim{\mathrm{dim}\,}
\def \id{\mathrm{id}}
\def \k{\kappa}
\def \s{\smallsmile\!}

\def \colim{\mathrm{\colim}}
\def \O{\mathop{\mathrm{O}}}
\def \Hom{\mathop{\mathrm{Hom}}}
\def \BSO{\mathop{\mathrm{BSO}}}
\def \BO{\mathop\mathrm{BO}}
\def \SO{\mathop{\mathrm{SO}}}
\def \Th{\mathop{\mathrm{Th}}}
\def \A{\mathcal{A}}
\def \rk{\mathop{\mathrm{rk}}}

\input xy
\xyoption{all}

\newtheorem{theorem}{Theorem}[section]

\newtheorem{lemma}[theorem]{Lemma}

\newtheorem{corollary}[theorem]{Corollary}

\theoremstyle{remark}
\newtheorem{remark}[theorem]{Remark}

\theoremstyle{definition}
\newtheorem{definition}[theorem]{Definition}

\newtheorem{example}[theorem]{Example}

\title{Invariants of singular sets of smooth maps}
\author{R. Sadykov}
\thanks{The author was supported by the Postdoctoral Fellowship for Foreign Researchers, Japan Society for the Promotion of Science; Postdoctoral Fellowship of Max Planck Institute, Germany; Research Instructorship at Dartmouth College, USA; and Postdoctoral Fellowship of Toronto University, Canada.}
\begin{document}
\date{\today}
\begin{abstract}
A singular point of a smooth map $f\colon M\to N$ of manifolds of dimensions $m$ and $n$ respectively is a point in $M$ at which the rank of the differential $df$ is less than the minimum of $m$ and $n$. The classical invariant of the set $S$ of singular points of $f$ of a given type is defined by taking the fundamental class $[\bar{S}]\in H_*(M)$ of the closure of $S$. We introduce and study new invariants of singular sets for which the classical invariants may not be defined, i.e., for which $\bar{S}$ may not possess the fundamental class. The simplest new invariant is defined by carefully choosing the fundamental class of the intersection of $\bar{S}$ and its slight perturbation in $M$. Surprisingly, for certain singularity types such an invariant is well-define (and not trivial) despite the fact that $\bar{S}$ does not possess the fundamental class. 

We determine new invariants for maps with Morin singularities---i.e., singularities of types $\mathcal{A}_k$ for $k\ge 1$ in the $ADE$-classification of simple singularities by Dynkin diagrams---and, as an application, show that these invariants together with generalized Miller-Morita-Mumford classes form a commutative graded algebra of characteristic classes that completely determine the cobordism groups of maps with at most $\A_k$-singularities for each $k\ge 1$. 
\end{abstract}
\maketitle

\setcounter{tocdepth}{2}
\tableofcontents

\addcontentsline{toc}{part}{Introduction}

\section{Introduction}

Let $f\colon M\to N$ be a general position map of complex analytic manifolds of complex dimensions $m$ and $n$ respectively. A point $x$ in $M$ is \emph{singular} if the differential of $f$ at $x$ is of rank less than the minimum of $m$ and $n$. Let $S\subset M$ denote the set of singular points of $f$ of a given type. Its closure $\bar{S}$ in $M$ is a submanifold of $M$ with singularities of codimension at least $2$. In particular $\bar{S}$ possesses a fundamental class whose image $[\bar{S}]$ in $H_*(M)$ does not change under homotopy of $f$. Rene Thom observed~\cite{Th}, \cite{HK} that the cohomology class dual to $[\bar{S}]$ is a polynomial in terms of Chern classes of the tangent vector bundle of $M$ and the pullback $f^*TN$ of the tangent bundle of $N$. Explicit expressions for the \emph{Thom polynomials} for particular singularity types were computed by Porteous~\cite{Po}, Ronga~\cite{Ro}, Gaffney~\cite{Ga} and T.~Ohmoto~\cite{Oh}; and recently a number of Thom polynomials were explicitly computed by Rim\'{a}nyi, Feher and Kazarian by means of a very efficient Rim\'{a}nyi restriction method~\cite{Ri1}, \cite{Kaz4}, \cite{FR}. 

In the case of a general position map $f\colon M\to N$ of smooth manifolds of dimensions $m$ and $n$ respectively, Chern classes are replaced with Stiefel-Whitney classes and the definition of Thom polynomials carries through provided that the closure of the set of singular points $\bar{S}$ of a given type possesses a fundamental class for every general position map $f$. If the \emph{dimension} $m-n$ of $f$ is positive, however, this is not necessarily the case since the boundary $\partial\bar{S}$ may  be of dimension $\dim \bar{S}-1$.   

\begin{figure}[ht]
	\centering
			\includegraphics[draft=false, width=80mm]{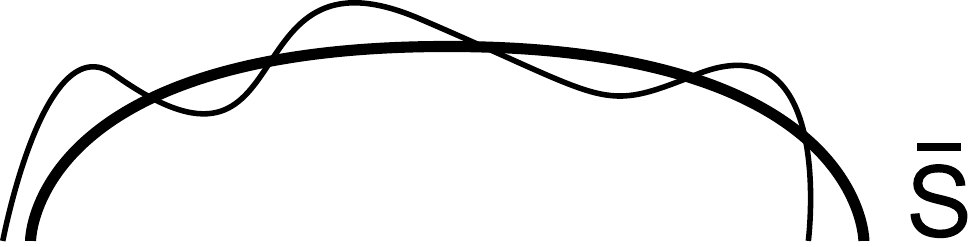}
\caption{For a smooth map $f$ of positive dimension, the set $\bar{S}$ (bold) of singular points of a given type may have a non-empty boundary of dimension $\dim\bar{S}-1$, and therefore may not possess a fundamental class. Nevertheless the intersection of $\bar{S}$ with a carefully chosen perturbation of $\bar{S}$ in $M$ may still possess a fundamental class, which is a non-trivial well-defined invariant.}
\label{fig:7}
\end{figure}

We consider smooth maps $f$ of positive dimension and introduce new invariants of singular sets. The simplest new invariant of the set $S$ of singular points of $f$ of a given type is defined by taking the fundamental class of the intersection of $\bar{S}$ and its carefully chosen perturbation in $M$. We show that for certain singularity types these invariants are well-defined even though the classical Thom polynomial is not well-defined, i.e., even though $\bar{S}$ does not possess a fundamental class. 

\begin{remark} If $S$ is a closed submanifold of $M$, then the homology class $[S'\cap S]$ does not depend on the choice of  the perturbation $S'$ of $S$ in $M$. On the other hand, if $S$ is not closed, then different choices of the perturbation $\bar{S}'$ of $\bar{S}$ lead to different homology classes $[\bar{S}'\cap \bar{S}]$. Our computation indirectly shows that for each $f$ there is a distinguished choice of the perturbation $\bar{S}'$ such that the homology class $[\bar{S}'\cap \bar{S}]$ is an invariant of the set $S=S(f)$ of singular points of $f$ of the prescribed type. 
\end{remark}

We determine a number of non-trivial invariants for Morin singularities, which are singularities of types $\A_r$ for $r\ge 1$ in the $ADE$ classification of simple singularities by Dynkin diagrams~\cite{Ar1}, \cite{Ar2}. As an application we show that these invariants together with the generalized Miller-Morita-Mumford classes completely determine the rational cobordism groups of \emph{$\A_r$-maps}, i.e., maps with singularities only of types $\A_i$ for $i\le r$. 

A comprehensive study of invariants of $\mathcal{A}$, $\mathcal{D}$ and $\mathcal{E}$ type singularities of Lagrange and Legendre  maps is carried out by Vassilyev~\cite{Va} by means of the Vassilyev complex. Instead of the Vassilyev complex, we use the Kazarian spectral sequence which is a far-reaching development of the Vassilyev complex, and which is defined in the Appendix by Kazarian in the addition \cite{Va1} of \cite{Va} (see also \cite{Kaz1}--\cite{Kaz4}). It follows that the peculiar invariants of singular sets of smooth $\A_r$-maps have no counterparts in the list of invariants of $\A_r$ singular sets of Lagrange and Legendre maps in the book \cite{Va}.    

In the remaining introductory sections we recall the definition of generalized Miller-Morita-Mumford characteristic classes of surface bundles and extend these invariants to invariants of smooth maps (\S\ref{s:mmm}); recall the Landweber-Novikov version of Thom polynomials (\S\ref{s:surp}); define new invariants of maps with prescribed singularities (\S\ref{s:inv}); and give a geometric interpretation of invariants of $\A_r$-maps (\S\ref{s:interpretation}). In sections~\ref{s:3}--\ref{s:rsg} we define Morin singularity types $\A_r$ and compute their relative symmetry groups.
In section~\ref{s:8} we construct the classifying loop space $\Omega^{\infty}\mathbf{A_r}$ for maps with $\A_r$ singularities; the cohomology classes in $H^*(\Omega^{\infty}\mathbf{A_r})$ are invariants of $\A_r$-maps. In subsequent sections we compute the rational cohomology groups of $\Omega^{\infty}\mathbf{A_r}$ by means of the Kazarian spectral sequence. Computation of differentials in the Kazarian spectral sequence is the most technical part of the paper; the dimensional argument in our case is not sufficient (see the discussion in the beginning of section~\ref{s:diff}). In the final sections \S\ref{s:rational}--\S\ref{s:final} we turn to applications of new invariants. After a brief review of necessary theorems from rational homotopy theory (\S\ref{s:rational}), we recall the definition of cobordism groups of maps with prescribed singularities (\S\ref{s:def}), and compute the rational cobordism groups of $\A_r$-maps for each $r\ge 1$ (\S\ref{s:final}).

\section{The generalized Miller-Morita-Mumford classes}\label{s:mmm}

Let $f\colon M\to N$ be a smooth map of closed manifolds of dimensions $m$ and $n$ respectively with positive difference $d=m-n$. We may choose a positive integer $t>0$ and embed $M$ into $\R^t\times N$ so that $f$ is the projection of $M$ onto the second factor. Let $j$ denote the inclusion into $\R^t\times N$ of an open tubular neighborhood of $M$ in $\R^t\times N$. We recall that the one point compactification construction is contravariant with respect to inclusions of open subsets. In particular, the inclusion $j$ leads to a map 
\[
    j_+\colon S^t\wedge N_+ \longrightarrow \Th \nu
\]
of one point compactifications, where $N_+$ stands for the disjoint union of $N$ and a point, while $\Th\nu$ is the one point compactification of the open tubular neighborhood of $M$. If the map $f$ is \emph{oriented}, i.e., there is a chosen orientation of the stable vector bundle $f^*TN\ominus TM$ over $M$, then there is a homomorphism, called the \emph{Umkehr} (or Gysin) map~\cite{Co}, 
\[
     f_!\colon H^*(M)\xrightarrow{\Th} \tilde{H}^{*+t-d}(\Th\nu)\xrightarrow{j_+^*} \tilde{H}^{*+t-d}(S^t\wedge N_+)\xleftarrow{\approx} H^{*-d}(N) 
\]
where the first and last maps are Thom isomorphisms, and all cohomology groups are taken with coefficients in $\Q$.

Generalized Miller-Morita-Mumford classes of a map $f$ are defined under the assumption that $f$ has no singular points, or, equivalently, that the map $f$ is the projection map of an oriented smooth fiber bundle. In this case $f$ determines a \emph{vertical tangent bundle} over $M$ which is the subbundle of the tangent bundle of $M$ that consists of vectors tangent to the fibers of $f$. We note that the vertical tangent bundle represents the stable vector bundle $TM\ominus f^*TN$. Given a characteristic class $p$ of vector bundles, the \emph{generalized Miller-Morita-Mumford class} $\mathop\mathrm{k}_p(f)$ associated with $p$ is defined to be the image of the class $p$ of the vertical tangent bundle under the Umkehr map 
\[
   f_!\colon H^*(M)\longrightarrow H^{*-d}(N).
\]  
These classes generalize the standard Miller-Morita-Mumford classes~\cite{Ma}, \cite{Mi}, \cite{Mo}, which are characteristic classes of surface bundles. If $f$ is not a fiber bundle, i.e., if $f$ has singular points, then the vertical tangent bundle does not exist, but it can be replaced by the stable vector bundle $TM\ominus f^*TN$. Thus for every characteristic class $p$ of oriented vector bundles of dimension $d$ that is a polynomial in terms of Pontrjagin classes, we define the (extended) generalized Miller-Morita-Mumford class $\mathop\mathrm{k}_p(f)$ of a smooth map $f$ to be the class in $H^*(N)$ given by the image of the $p$-th characteristic class of $TM\ominus f^*TN$ under the Umkehr map $f_!$. 

\begin{theorem}\label{mmm} The generalized Miller-Morita-Mumford class $\mathop\mathrm{k}_p$ associated with an arbitrary rational Pontrjagin class $p\in \Q[p_1,p_2,...]$ is a \underline{non-trivial} characteristic class of $\A_r$-maps for each $r>0$. 
\end{theorem}

\begin{remark} The extended generalized Miller-Morita-Mumford classes associated with different Pontrjagin classes are algebraically independent. In particular,  the class $\mathop\mathrm{k}_{p_1p_2}$ can not be expressed in terms of  $\mathop\mathrm{k}_{p_1}$ and $\mathop\mathrm{k}_{p_2}$. The generalized Miller-Morita-Mumford classes are closely related to Landweber-Novikov operations (see \S\ref{s:surp}). 
\end{remark} 
 
\begin{remark} In the case of maps of even positive dimension $d$, Ebert showed that the generalized Miller-Morita-Mumford classes are algebraically independent already for manifold bundles~\cite{Eb}. On the other hand, in contrast to Theorem~\ref{mmm}, in the case of maps of odd positive dimension $d$, there are trivial generalized Miller-Morita-Mumford classes on manifold bundles; Ebert computed that these are precisely the characteristic classes associated with  classes in the ideal generated by the components $\mathcal{L}_{4n}\in H^{4n}(\BSO_d)$ of the Hirzebruch $\mathcal{L}$-class for $4n>d$, see \cite{Eb}.  
\end{remark}

In the next section we define characteristic classes of Morin maps associated with higher singularities.

\section{The Landweber-Novikov version of Thom polynomials} \label{s:surp}

To begin with we fix a prescribed set of singularity types (or just one singularity type) so that for any choice of a general position map $f\colon M\to N$ of complex analytic manifolds of dimensions $m$ and $n$ respectively, the set $S$ of points at which $f$ has a prescribed singularity is a submanifold of $M$. We do not require that the set $S$ is closed in $M$ and let $T$ denote the set $\bar{S}\setminus S$. 
The Thom polynomial $\Tp(f)$ can be defined to be the image of the Thom class under the homomorphism 
\begin{equation}\label{eque:1}
    H^*(M\setminus T, M\setminus \bar{S})\longrightarrow H^*(M\setminus T),
\end{equation}
induced by the inclusion of pairs of topological spaces. It is known that the class $\Tp(f)$ lifts to a cohomology class on $M$, and therefore we may regard $\Tp(f)$ to be a cohomology class in $H^*(M)$ defined up to the image of the group $H^*(M, M\setminus T)$. In other words, the so-defined Thom polynomial $\Tp(f)$ is defined only up to cohomology classes that come from ``higher order" singularities. For an example, we refer the reader to the paper of Rimanyi~\cite{Ri2}. 

Motivated by Landweber-Novikov operations in the complex cobordism theory, Kazarian defined the \emph{Landweber-Novikov version} of Thom polynomials of holomorphic maps \cite{Kaz3}. For the Thom polynomial $\Tp(f)\in H^*(M)$ of a holomorphic map $f\colon M\to N$ of complex analytic manifolds, the \emph{Landweber-Novikov} version of $\Tp(f)$ is defined to be the class $f_!\Tp(f)\in H^*(N)$. 
 
\begin{remark} The operations $\mathbf{MU}^*(\mathbf{MU})$ in complex cobordisms are determined by Landweber~\cite{La} and Novikov~\cite{No} (see also \cite{Ad}). The group of cobordism operations is isomorphic to the completed tensor product $\mathbf{MU}^*(\cdot)\hat\otimes A^*$, where $A^*$ is the group of \emph{basic} cobordism operations. Basic cobordism operations are linear combinations of \emph{Landweber-Novikov} operations $s_w$; there is one operation $s_w$ for each set of integers $w=(w_1,...,w_k)$. The latter operations correspond to cohomology operations in $H^*(\mathbf{MU})$, e.g., see the proof of \cite[Theorem 4.2]{La}. The Landweber-Novikov version of the Thom polynomial $f_!\Tp(f)$ is an invariant of the cobordism class of $f$\cite{Kaz3}; in particular it determines an operation in $H^*(\mathbf{MU})$ which associates the cohomology class $f_!\Tp(f)$ to each cobordism class $f$ of appropriate dimension. Thus, each singularity type determines~\cite{Kaz3} a linear combination of cohomology Landweber-Novikov operations. 
\end{remark} 
 
Higher Thom polynomials $\Tp(f, p)$ of the singular set $\bar{S}$ are defined to be the classes $i_!p(\nu)$, where $p$ is any characteristic class of vector bundles and $\nu$ is the normal bundle of the proper inclusion $i$ of $S$ into $M\setminus T$. In particular, the Thom polynomial $\Tp(f)$ coincides with $\Tp(f, 1)$. By the Kazarian spectral sequence, it immediately follows that each class $\Tp(f,p)$ is a polynomial in terms of characteristic classes of the tangent bundle of $M$ and the pull back $f^*TN$ of the tangent bundle of $N$. Furthermore this polynomial expression for $\Tp(f,p)$ is unique up to the ideal generated by polynomial invariants of higher order singularities. 

Next, let us turn to a generalization of the above definitions to the case of  maps $f\colon M\to N$ of manifolds. We do not assume that $M$ and $N$ are complex analytic, and we do not assume that $f$ is holomorphic. Instead, we assume that $M$ and $N$ are smooth, and $f$ is a solution of a fixed coordinate free differential relation $\mathcal{R}$ (e.g., $f$ is a smooth map, or $f$ is a Morin map). To begin with we give a straightforward generalization of the above definitions. Then we explain why such an approach does not work and modify the definition. 

As above we fix a prescribed set of singularity types so that for any choice of a general position map $f\colon M\to N$ of manifolds, the set $S$ of points at which $f$ has a prescribed singularity is a submanifold of $M$. Let $T$ denote the set $\bar{S}\setminus S$. In contrast to the case of holomorphic maps, the cohomology class $\tp(f)\in H^*(M\setminus T)$ defined by means of the homomorphism (\ref{eque:1}) may not admit a lift to a class in $H^*(M)$. To simplify the exposition, let us assume that it admits a lift to a class in $H^*(M)$. Then we may identify $\tp(f)$ with a cohomology class in $H^*(M)$ defined up to the image of $H^*(M, M\setminus T)$. 

Next example shows that the class $\tp(f)$ is not necessarily an invariant of $\mathcal{R}$-maps.

\begin{example}\label{exa:1} The simplest singularity types for maps of positive dimension are \emph{fold} (see Definition~\ref{d:fold}), \emph{cusp} and \emph{swallowtail} singularity types, which are $\A_1$, $\A_2$ and $\A_3$-singularity types in the $ADE$ classification. There is a map $f$ for which the class $\tp(f)$ associated with the set $S$ of fold singular points of a given index is not trivial. On the other hand, if the dimension of the target of $f$ is at least $2$, then the map $f$ is homotopic in the class of maps with swallowtail singular points to a map $g$ with no closed components of $S$. Since the class $\tp(g)$ is trivial, we have $\tp(f)\ne \tp(g)$.      
\end{example}

If the normal bundle $\nu$ of the proper embedding $i$ of $S$ into $M\setminus T$ has a canonical orientation, then for each characteristic class $p$ of vector bundles, we define the class $\tp(f, p)$ by $i_!p(\nu)$. If the normal bundle is not orientable, then the class $\tp(f, p)$ can still be defined by means of the Umkehr map~\cite{OSS} with twisted coefficients; the class $\tp(f,p)$ may still be a class with non-twisted coefficients. In contrast to higher Thom polynomials in the case of complex analytic maps, in the case of smooth $\mathcal{R}$-maps a priory the class $\tp(f,p)$ is not necessarily an invariant in the sense that if $f$ and $g$ are two $\mathcal{R}$-maps which are homotopic through $\mathcal{R}$-maps, then $\tp(f,p)$ maybe different from $\tp(g,p)$.  

\begin{example}\label{exa:2} Let $f\colon M\to N$ be a fold map from an oriented closed path-connected manifold $M$ of dimension $8s+4$ to an oriented manifold of dimension $4s+3$ for some integer $s\ge 0$. Then the set $S\subset M$ of fold singular points of $f$ of a fixed even index $a\le 2s$ is an oriented closed submanifold of $M$ of dimension $4s+2$ that carries a fundamental class. Let $e$ be the Euler class of vector bundles of dimension $4s+2$. Then the class $\tp(f, e)$ associated with the set $S$ equals $n$ times the generator of $H^{8s+4}(M)$, where  $n$ is the algebraic number of points in the intersection of $S$ and its slight perturbation in $M$. We will see that this class may be non-trivial. On the other hand, as in Example~\ref{exa:1}, the map $f$ is homotopic through $\A_3$-maps to a map $g$ with no closed components of $S$. In particular, the Euler class $e$ evaluated on the normal bundle of $S$ in $M\setminus T$ has no non-trivial components, and therefore $\tp(g, e)=0$. 
\end{example}

In the next section we will modify the definition of the class $\tp(g,e)$ so that it agrees with the class $\tp(f,e)$.

\section{New invariants of smooth maps with prescribed singularities}\label{s:inv}

Examples~\ref{exa:1} and \ref{exa:2} show that in general the classes $\tp(f, p)$ may not be invariant under homotopy of $f$ through $\mathcal{R}$-maps, and therefore, the straightforward generalizations of the Thom and higher Thom polynomials to the case of smooth maps of positive dimension are not appropriate. The reason of the failure of invariance of $\tp(f,p)$ is in indeterminacy of $\tp(f, p)$, which is not invariant under homotopy of $\mathcal{R}$-maps. To rectify the issue we define the classes $\tp(f, p)$ by means of classifying spaces. To motivate the definition, suppose that there is a smooth universal $\mathcal{R}$-map $F\colon \mathbf{E}\to \mathbf{B}$ for which there is a (not necessarily injective or surjective) correspondence between homotopy classes of $\mathcal{R}$-maps $f\colon M\to N$ and homotopy classes of maps $\tilde f\colon N\to \mathbf{B}$ such that pairs $(f, \tilde{f})$ of corresponding maps fit pullback diagrams
\[
  \begin{CD} 
    M @>>> \mathbf{E} \\
    @Vf VV @VF VV \\
    N @>\tilde{f}>> \mathbf{B}.
  \end{CD}  
\] 
Then we may compute the Landweber-Novikov version $F_!\tp(F, p)$ of classes $\tp(F, p)$, which we denote by $\Tp(F, p)\in H^*(B)$, and define $\Tp(f,p)$ to be the pullback classes $\tilde{f}^*\Tp(F, p)$. The so-defined classes $\Tp(f,p)$ are invariant under homotopy of $\mathcal{R}$-maps since the homotopy class of $\tilde{f}$ does not depend on the choice of an $\mathcal{R}$-map representing the homotopy class $[f]$ of $\mathcal{R}$-maps. In general a smooth universal $\mathcal{R}$-map of finite dimensional manifolds may not exist, but a classifying space which plays the role of $\mathbf{B}$ always exists and we will use one. 

\begin{remark} A few constructions of the universal map $u$ are known for different choices of $\mathcal{R}$. For smooth maps of negative dimension with a prescribed set of simple singularities, the universal map has been constructed by Rimanyi and Szucs~\cite{RS}. For maps of an arbitrary fixed dimension with an arbitrary set of singularities, the universal map is constructed by Kazarian~\cite{Kaz4}. In \cite{Sad5} the author used a different construction adopted from the construction of Madsen and Weiss~\cite{MW}, which can be sketched as follows. Consider the space of pairs $(f, \Delta)$, where $\Delta$ is a standard simplex and $f$ is an $\mathcal{R}$-map into $\Delta$. The $i$-th face of $(f, \Delta)$ is defined to be the pair $(\partial_if, \Delta_i)$, where $\Delta_i$ is the $i$-th face of $\Delta$ and $\partial_if=f|f^{-1}(\Delta_i)$. The space $\mathbf{B}$ is defined to be the quotient of the space of points $x\in \Delta$ in all pairs $(f, \Delta)$ under the identifications generated by identifications of pairs $(f', \Delta')$ and those faces $(\partial_if,\Delta_i)$ that are identical to $(f', \Delta')$. The union of maps $f$ determines a map $u\colon \mathbf{E}\to\mathbf{B}$. Now, let $f\colon M\to N$ be an arbitrary $\mathcal{R}$-map. Choose a triangulation of $N$ by simplices $\Delta$. Then each pair $(f|f^{-1}(\Delta), \Delta)$ has a counterpart in the space $\mathbf{B}$, and therefore there is a map $\tilde{f}\colon N\to \mathbf{B}$. 
\end{remark}

Surprisingly, the classes $\Tp(f,p)$ may be non-trivial invariants even if the classes $\Tp(f)$ are not invariants.  

\begin{example}\label{exa:3}
Let $f\colon M\to N$ be a fold map of dimension $4s+1$ of oriented closed manifolds for some non-negative integer $s$. Then the set $S\subset M$ of fold singular points of $f$ of a fixed index $a\le 2s$ is an oriented closed submanifold that carries a fundamental class. However, a homotopy of $f$ through Morin maps does not preserve the homology class $[\bar{S}]$. On the other hand, the invariant $\Tp(f,e)=I'_{1,a}$ associated with fold singular points of index $a$ and the Euler class $e$ does not change under homotopy through Morin maps.  Furthermore, according to Theorem~\ref{th:main}, this class is non-trivial (despite the fact that the class $\Tp(f)$ is not an invariant).  
\end{example} 
 
\begin{remark} If $S$ stands for the set of all fold singular points of a fold map $f$ of dimension $4s+1$, then $[\bar{S}]$ does not change under homotopy of $f$ through Morin maps; the Thom polynomial of fold singular points of all indices is well-defined. The invariant $\sum I'_{1,a}$, where $a$ ranges over all indices $0,...,2s+1$, was studied by Ohmoto, Saeki and Sakuma in the paper \cite{OSS}. The authors proved that this invariant equals $p_{2s+1}(TM\ominus f^*TN)\in H^{8s+4}(M)$ modulo $2$-torsion if considered with integer coefficients.  
\end{remark} 
 
\begin{remark} A priori the classes $\Tp(f,p)$ are not polynomials in terms of characteristic classes of the tangent bundle of $M$ and the pullback $f^*TN$ of the tangent bundle of $N$. 
\end{remark}

\section{Well-defined invariants of $\A_r$-maps}\label{s:interpretation}

The list of invariants of $\A_r$ maps depends on the dimension $d$. In this section we list well-defined invariants of Morin maps $f\colon M\to N$ of dimension $d$, where $d=4s+3$ for some $s\ge 0$, see Theorem~\ref{th:20.6} (the other characteristic classes can be described similarly). Namely, we describe the classes
%In this case the rational oriented cobordism class of a Morin map $f: M\to N$ %of smooth manifolds is determined by classes $f_*(I_{Q,a})$, $f_*(I_Q)$, %$f_*(\tau_i)$ and $f_*(\sigma_Q)$ in $H^*(N)$ which we interpret now. 
\begin{equation}\label{eq:0.1}
     I_{Q,a}(f), I_Q(f), \tau_i(f) \quad \mathrm{and} \quad \sigma_Q(f) 
\end{equation}
in $H^*(N)$ corresponding to some of the generators of cohomology algebras $H^*(\Omega^{\infty}\mathbf{A_r})$ of the classifying infinite loop space of $\A_r$-maps. To this end we will describe the classes  
\begin{equation}\label{eq:0.2}
I'_{Q,a}(f), I'_Q(f), \tau'_i(f) \quad \mathrm{and}\quad \sigma'_Q(f)
\end{equation}
in the cohomology group $H^*(M)$ such that each class in (\ref{eq:0.1}) is a pushforward of the corresponding class in (\ref{eq:0.2}) by means of $f_!$. 
%The other characteristic classes in $H^*(N)$ are the pushforwards of products %of classes in (\ref{eq:0.2}). 
%For example, a priori, in the case of a map of $f$ codimension $-2$ into a %manifold $N$, the set of independent characteristic classes in $H^*(N)$ %consists of $\tau_1^{i}(f)$, with $i\ge 1$. 

Since in general the definition of invariants is not straightforward (see Examples~\ref{exa:2} and \ref{exa:3}), to simplify the exposition we will assume that the given map $f: M\to N$ is a {\it fold map}. Our computation shows that each of these classes extends to an invariant of $\A_r$-maps for each $r$. 

\begin{definition}\label{d:fold} A map $f\colon M\to N$ is a fold map if for each point $x\in M$ there are coordinate neighborhoods $U\subset M$ of $x$ and $V$ of $f(x)$ such that either $f|U$ is regular or $f|U$ is of the form
\begin{equation}\label{eq:0.3}
     (x_1,  ..., x_{m}) \mapsto (- x_1^2-\cdots - x_a^2 + x_{a+1}^2 + \cdots + x_{a+b}^2, x_{a+b+1}, ..., x_m)  
\end{equation}
for some $a$ and $b$, with $a\le b$, where $(x_1,...,x_m)$ are coordinates in $U$. The number $a$ in (\ref{eq:0.3}) is called the {\it index} of a fold singular point $x$ of $f$. From the local form (\ref{eq:0.3}), it follows that if $f$ is a fold map, then the set $S_a(f)$ of all fold singular points of $f$ of index $a$ is a submanifold of $M$.  
\end{definition}

\subsection{Invariants $I_{Q,a}$ and $I_Q$.} We will see that the normal bundle $\nu_a$, with $a=0, ..., 2s+1$, of the submanifold $i_a\colon S_a(f)\subset M$  of fold singular points of $f$ of index $a$ has a canonical orientation. Let $U_a$ denote the Thom class in $H^{4s+4}(M, M\setminus S_a(f))$ orienting $\nu_a$, and let $e_a\in H^*(S_a(f))$ be the Euler class of $\nu_a$ given by the restriction of $U_a$ to $S_a(f)$. 

A choice of an orientation of the cokernel bundle of $f$ over $S_a(f)$ leads to a splitting $\nu_a=\nu^+_a\oplus \nu^-_a$ of the vector bundle $\nu_a$ into the Whitney sum of its subbundles where the quadratic form of $f$ is positive definite and negative definite respectively. We choose the orientation of the cokernel bundle of $f$ over $S_a(f)$ so that the dimensions of $\nu^+_a$ and $\nu^-_a$ are $a$ and $b=4s+4-a$ respectively.

Let $Q$ denote a polynomial in $\left\lfloor a/2\right\rfloor + \left\lfloor b/2 \right\rfloor$  variables $p_1, ..., p_{\left\lfloor a/2\right\rfloor}$ and
$p'_1, ..., p'_{\left\lfloor b/2 \right\rfloor}$. Then the invariant $I'_{Q,a}(f)$ is defined to be the class in $H^*(M)$ given by the image under the Umkehr map $(i_a)_!$ of the class 
\begin{equation}\label{eq:1.1}
     e_a\, Q(p_1(\nu^+_a), \dots , p_{\left\lfloor a/2\right\rfloor}(\nu^+_a),  p'_1(\nu^-_a), \dots, p'_{\left\lfloor b/2 \right\rfloor}(\nu^-_a)),
\end{equation}
where the classes $p_i(\nu^+_a)$ and $p'_j(\nu^-_a)$ are the corresponding Pontrjagin classes of $\nu^+_a$ and $\nu^-_a$ respectively. This invariant is trivial if $a$ is odd, since the rational Euler class of an odd dimension class is trivial. 

\begin{remark} We will see that for each polynomial $Q$ there is a map $f$ with nontrivial invariant $I'_{Q,a}(f)$. This class determines an invariant of Morin maps, despite the fact that the fundamental class of $S_a$ is not invariant under homotopy through Morin maps. We will also see that, for example, the classes $(i_a)_!Q$ do not determine invariants of Morin maps.   
\end{remark}  
 
The invariant $I'_{Q,2s+2}(f)=I'_Q(f)$ is well defined only under an additional assumption as the normal bundle of $S_a(f)$ in $M$ for $a=2s+2$ may not split. Here we need to assume that the polynomial $Q$ is chosen so that 
\[
    Q(p_1, ..., p_{2s+2}, p'_1, ..., p'_{2s+2})\ = \     Q(p'_1, ..., p'_{2s+2}, p_1, ..., p_{2s+2}).
\]

\subsection{Invariants $\sigma_Q$ and $\tau_j$.} The class $\sigma_Q$ is defined somewhat similarly to $I'_Q$. To begin with we choose a polynomial $Q$ in variables $p_1,..., p_{2s+2}, p'_1,..., p'_{2s+2}$ so that 
\[
    Q(p_1, ..., p_{2s+2}, p'_1, ..., p'_{2s+2})\ = \ -     Q(p'_1, ..., p'_{2s+2}, p_1, ..., p_{2s+2}).
\]
Next we define cohomology classes $\sigma_{Q,a}\in H^*(M)$ as above replacing (\ref{eq:1.1}) by $(i_a)_! Q$. The classes $\sigma_{Q,a}$ are not invariant under Morin homotopies of $f$, but it turned out (see Theorem~\ref{th:20.6}) that the class 
\begin{equation}\label{eq:1.2}
\sigma_{Q}=\sum_{a=0}^{2s+2}\, (-1)^a\sigma_{Q,a}
\end{equation}
is a well-defined invariant. 

\begin{remark} The invariant $\sigma_Q$ is not defined, for example, for $Q=1$, because the normal bundle of the set $S$ of singular points of index $2s+2$ is non-orientable, and therefore there is no rational Thom class $U_a$. On the other hand, if we consider, for example, only oriented functions, then the invariant $\sigma_1$ can be defined. In fact, for a function $f\colon M\to \R$ on a closed oriented manifold of dimension $4s+4$, the invariant $\sigma_1$ is Poincare dual  to the homology class given by the Euler characteristic of $M$ times the fundamental class of $M$.   
\end{remark}

Finally, the class $\tau_{j}$, with $j=s+1, ..., 2s+2$, is defined similarly to $\sigma_Q$. Here we replace the polynomial $Q$ by $p'_jQ$ and take the sum (\ref{eq:1.2}).

\section{Cobordism groups of Morin maps}

A generalized cohomology theory associates to each topological space a sequence of abelian groups, called generalized cohomology groups, satisfying all Eilenberg-Steenrod axioms except the dimension axiom. Cobordism theories constitute an important class of generalized cohom\-ology theories~\cite{St}, \cite{Ru}; each cobordism theory $Cob_{\tau}^*$ corresponds to  a structure $\tau$ on continuous maps so that the cobordism groups $Cob_{\tau}^*(N)$ of a smooth manifold $N$ are isomorphic to appropriately defined groups of cobordism classes of $\tau$-maps of smooth manifolds into $N$. For example, the oriented cobordism groups $\mathbf{MSO}^*(N)$ (respectively, complex cobordism groups $\mathbf{MU}^*(N)$) of a possibly non-orientable manifold $N$ are isomorphic to the groups of appropriately defined cobordism classes of oriented maps (respectively, maps with $\mathbf{MU}$-structure) into $N$ . 

It turns out that the cobordism theory of smooth maps with prescribed singularities into smooth manifolds fit the classical setting well. In this case one identifies $\tau$ with a set of prescribed singularities of smooth maps and defines a {\it $\tau$-map} of smooth manifolds to be a smooth map with singularities of type $\tau$. Then under a general conditions on $\tau$ the notion of $\tau$-maps can be formally extended to the category of topological spaces and continuous maps so that the appropriately defined cobordism groups of $\tau$-maps form a cobordism theory (see Theorem~\ref{th:9.0}). 

\begin{remark} It is interesting that while the classical cobordism theories are defined by means of spectra $\mathbf{MG}=M(G)$ associated with subgroups $G\subset O$ such as $G=\O, \SO, \mathop\mathrm{U},...$, the cobordism theories of $\tau$-maps are defined by means of spectra $\mathbf{X}_{\tau}=M(\sqcup G_{\alpha})$ associated with relative symmetry groups $G_{\alpha}\subset O$ of singularity types $\alpha\in \tau$ (see section~\ref{s:kaz}).  
\end{remark}

Cobordism groups of various versions of Morin maps of non-negative dimension are studied in a series of papers by Ando~\cite{An1}, \cite{An2}, \cite{An3}, Kalmar~\cite{Ka1}, \cite{Ka2}, \cite{Ka3}, \cite{Ka4}, Saeki~\cite{Sae1}, \cite{Sae}, Ikegami-Saeki~\cite{IS}, \cite{IS2}, Ikegami~\cite{I}, and the author~\cite{Sa5}. For a study in the case $d<0$ we refer an interested reader to the paper of Szucz~\cite{Sz}. 

As an application of our study of invariants of $\A_r$-maps, we determine the rational oriented cobordism groups of \emph{$\A_r$-maps}. In fact, we show that these groups are completely determined by rational characteristic classes, and on the other hand, our computation shows that the algebra of rational characteristic classes is a tensor algebra of polynomial and exterior algebras (Theorem~\ref{th:main}) in terms of above defined invariants and generalized Miller-Morita-Mumford classes. More precisely, let $\Omega^{\infty}\mathbf{A_r}$ denote the infinite loop space of the spectrum $\mathbf{A_r}$ of the oriented  cobordism group of $\A_r$-maps (see section~\ref{s:8}). Then, the cobordism theory of $\A_r$-maps associates to a topological space $X$ a sequence of groups 
\[
     \mathbf{A}^i_r(X)\colon=[X, \Omega^{\infty-i}\mathbf{A_r}]
\]
where $i\ge 0$. On the other hand the spectrum $\Omega^{\infty}\mathbf{A_r}$ is rationally equivalent to a product of rational Eilenberg-Maclane spaces $K(\Q, n)$ (see section~\ref{s:final}). Consequently the rational oriented cobordism class of an $\A_r$-map $f: M\to N$ classified by a continuous map $u: S^d\wedge N_+\to \Omega^{\infty}\mathbf{A}$ is completely determined by the set of rational {\it characteristic classes} 
\[
\{\  \chi(f):=j[u^*(\chi)]\ |\ \chi\in H^*(\Omega^{\infty}\mathbf{A})\ \},
\] 
where $j$ is the Thom isomorphism $\tilde{H}^*(S^d\wedge N_+)\approx H^{*-d}(N)$. 
Theorem~\ref{th:main} below
%lists 
describes the algebras $H^*(\Omega^{\infty}\mathbf{A})$ of characteristic classes of oriented Morin maps according to the dimension $d$.

\begin{theorem}\label{th:main} The algebra $H^*(\Omega^{\infty}\mathbf{A_r})$ is a tensor product of polynomial and exterior algebras. The generators of the polynomial algebra (respectively exterior algebra) are in bijective correspondence with the homology classes of even (respectively odd) order in $H_*(\mathbf{A_r})$. The ranks of these groups are listed in Theorems~\ref{th:4s}--\ref{th:4s+2}, while the generators of the dual cohomology algebra are listed in the proofs of Theorems~\ref{th:4s}--\ref{th:4s+2}.
\end{theorem}

\begin{corollary} For every $r\ge 1$ there is an $\A_r$-map of even dimension that is not rationally cobordant (and hence, homotopic) through $\A_r$-maps to a map without $\A_r$-singularities. On the other hand, for example, if $r+1$ is divisible by $4$, then every $\A_{r}$-map of odd dimension is rationally cobordant in the class of $\A_{r}$-maps to a map without $\A_{r}$ singular points. 
\end{corollary}

Our main tools are the {\it bordism version of the Gromov h-principle} \cite{Sa10}  (see also \cite{An4}, \cite{El}, \cite{Sa}, \cite{Sz}) associating to a set $\tau$ of prescribed singularities a classifying infinite loop space $\mathbf{X}_{\tau}$ (see section~\ref{s:kaz}), and a so-called {\it Kazarian spectral sequence} \cite{Kaz1}, \cite{Kaz2}, \cite{Kaz3}, \cite{Kaz4}, which facilitates computations of cohomology groups of $\mathbf{X}_{\tau}$.

\addcontentsline{toc}{part}{Morin singularities}

\section{Singularity types}\label{s:3}

Let $M$ and $N$ be two manifolds. By definition, a \emph{map germ} from $M$ to $N$ at a point $x\in M$ is an equivalence class of maps from a neighborhood of $x$ in $M$ to $N$. Two maps represent the same map germ at $x$ if
their restrictions to some neighborhood of $x$ in $M$ coincide. Slightly abusing notation we write $f\colon (M,x)\to (N, f(x))$ for the map germ at $x$ represented by a map $f$; we note that the map $f$ that defines a map germ at $x$ may not be defined over $M$. 

A map germ $f: (M,x)\to (N,f(x))$ represented by a map of smooth manifolds is called \emph{singular} if the differential of $f$ at $x$ is of rank less that $\mathrm{min}(\dim M, \dim N)$. Otherwise $f$ is said to be \emph{regular}. Two map germs $f,g: (M,x)\to (N,f(x))$ are said to be {\it right-left equivalent}, or simply {\it $\mathcal{RL}$-equivalent}, if there are diffeomorphism germs $\alpha$ and $\beta$ of neighborhoods of $x$ in $M$ and $f(x)$ in $N$ respectively such that $f=\beta\circ g\circ \alpha^{-1}$. By definition, the \emph{singularity type} of a map germ is an equivalence class of map germs under the relation generated by the right-left equivalence relation and the equivalences $[f]=[f\times \mathrm{id}_{\R}]$, where $\id_{\R}$ is the identity map germ of $(\R, 0)$. 
Thus the equivalence relation given by singularity types is the minimal coordinate-free equivalence relation on map germs that allows us to relate singularities of maps $M\to N$ with singularities of homotopies $M\times [0,1]\to N\times [0,1]$. Every singularity type contains a so-called {\it minimal} map germ that is not right-left equivalent to any germ of the form $f\times \id_{\R}$.
%\footnote{What is an example of two minimal map germs that are not %$\mathcal{RL}$-equivalent, but represent the same type}.

The \emph{codimension} $q$ of a map of an $m$-manifold into an $n$-manifold is defined to be the difference $n-m$. Similarly, the \emph{dimension} $d$ of such a map is the difference $m-n$. $\mathcal{RL}$-equivalent map germs are of the same dimension (and codimension). In what follows, with a few exceptions, we will assume that the dimension $d$ of considered maps is positive. In the case $d=0$ the computations can be carried out in a similar way with a few modifications. 

\begin{example} For each number $d\ge 0$, there is a submersion singularity type of map germs of dimension $d$. Its minimal map germ is given by a unique map $(\R^d,0)\to (\R^0,0)$. Similarly for each number $q>0$, there is an immersion singularity type of map germs of codimension $q$. Its minimal map germ is given by the inclusion $(\R^0,0)\to (\R^q,0)$.  
\end{example}

The simplest measure of complexity of a map germ $f: (M,x)\to (N,f(x))$ is the {\it corank} of $f$ which is defined to be the number 
\[
\mathrm{min}(\dim M, \dim N) - \mathrm{rk}f(x),
\]
where $\mathrm{rk}f(x)$ stands for the rank of the differential of $f$ at $x$. For example, immersion and submersion map germs are precisely the corank zero map germs. Of course, the simplest singularities of map germs are of corank one. Among corank one singularities there is a countable series of so-called {\it Morin singularities} which are the simplest corank one singularities\footnote{In the case of maps of negative dimension the complement in the space of corank one map germs to the space of Morin map germs is of infinite codimension. However in the negative codimension case, $D$ and $E$ type singularities are also of corank one.}. We will define Morin singularity types below by listing their minimal map germs. A map is said to be {\it Morin} if it has only Morin singularities.

\section{Minimal Morin map germs}

Every minimal Morin map germs is $\mathcal{RL}$-equivalent to one of the following map germs (e.g., see \cite{AVG}). 

\subsection{Regular map germ $\A_0$} The minimal \emph{regular} germ of dimension $d$ is the projection $\R^{d}\to \R^{0}$. For any smooth map $f\colon M\to N$ of manifolds, the set of regular points of $f$ is an open submanifold $M$. 

\subsection{Fold germ $\A_1(a,b)$} The canonical form of a minimal \emph{fold} germ 
\begin{equation}\label{e:3.1}
     (\R^{d+1}, 0) \longrightarrow (\R^1, 0)
\end{equation}
of dimension $d$, with $d>0$, is given by a quadratic form
\begin{equation}\label{e:3.2}
     y=Q(a,b):= - x_1^2-\cdots - x_a^2 + x_{a+1}^2 + \cdots + x_{a+b}^2, \quad a+b=d+1,\  a\le b,
\end{equation}
where $x_1, ..., x_{d+1}$ are the standard coordinates in $\R^{d+1}$ and $y$ is the standard coordinate in the target manifold $\R^1$. 
In other words a minimal fold map germ is nothing but a Morse function, while a general fold map germ is right-left equivalent to the product $f\times \id_{\R^s}$ of a Morse function $f$ and the identity map germ $\id_{\R^s}$ of $\R^s$.  The integer $a$ in the equation (\ref{e:3.2}) is said to be the \emph{index} of the map germ (\ref{e:3.1}). The singularity types of fold map germs with different indices are \emph{not} $\mathcal{RL}$-equivalent. We will denote by $\A_1(a,b)$ and $\A_1$ the singularity type of the map germ (\ref{e:3.2}) and the union of all fold singularity types $\A_1(a,b)$ respectively. For a general position map $f\colon M\to N$ of dimension $d$, the set of fold singular points of $f$ is a submanifold of $M$ of codimension $d+1$, which may not be closed in $M$. If $f$ is a Morin map, then the closure of the set of fold singular points of $f$ is a closed submanifold of $M$.  

\subsection{Morin map germ $\A_{2r}(a,b)$, $r>0$} There is one Morin singularity type $\A_{2r}(a,b)$ of map germs of dimension $d>0$ for each positive integer $r$ and two non-negative integers $a,b$ such that $a\le b$ and $a+b=d$. 
The canonical form of a minimal $\A_{2r}(a,b)$ germ  
\[
    (\R^{a+b+2r}, 0) \longrightarrow (\R^{2r}, 0)
\]
is given by
\[
  y_i = t_i, \quad i=1, \dots, 2r-1,
  \]
  \[
  y_{2r}= Q(a,b) + t_1x + t_2x^2 + \cdots + t_{2r-1}x^{2r-1} + x^{2r+1},
\]
where $(x_1, ..., x_d, x, t_1, ..., t_{2r-1})$ are the standard coordinates in $\R^{a+b+2r}$ and $(y_1,...,y_{2r})$ are the standard coordinates in $\R^{2r}$. We put 
\[
\A_{2r}=\mathop{\cup}_{a,b} \A_{2r}(a,b). 
\]
The set of $\A_{2r}$-singular points of a general position map of dimension $d$ is a submanifold of the source manifold of codimension $d+2r$. For a Morin map of dimension $d$, the closure of the set of $\A_{2r}$-singular points is a submanifold of the source manifold of codimension $d+2r$. A singular point of type $\A_2$ is called a {\it cusp} singular point.  

\subsection{Morin map germ $\A^{\pm}_{2r+1}(a,b)$, $r>0$} There is one singularity type $\A^{+}_{2r+1}(a,b)$ and one singularity type $\A^{-}_{2r+1}(a,b)$ of map germs of dimension $d$ for each positive integer $r$ and a pair of non-negative integers $(a,b)$ such that $a<b$ and $a+b=d$. If $d$ is even, then there is also a singularity type $\A^{+}_{2r+1}(a,a)=\A^{-}_{2r+1}(a,a)$ with $a=d/2$. The canonical form of a minimal $\A^{\pm}_{2r+1}(a,b)$ germ  
\[
    (\R^{a+b+2r+1}, 0) \longrightarrow (\R^{2r+1}, 0)
\]
is given by
\[
  y_i = t_i, \quad i=1, \dots, 2r,
  \]
  \[
  y_{2r+1}= Q(a,b) + t_1x + t_2x^2 + \cdots + t_{2r}x^{2r} \pm x^{2r+2},
\]
where $(x,_1,...,x_d, x, t_1,..., t_{2r})$ and $(y_1,...,y_{2r+1})$ are the standard coordinates in $\R^{a+b+2r+1}$ and $\R^{2r+1}$ respectively. The Morin singularity types $\A^{+}_{2r+1}(a,b)$ and $\A^{-}_{2r+1}(a,b)$ are the same if and only if $a=b$. We put 
\[
\A_{2r+1}=\mathop{\cup}_{a,b}\A^{\pm}_{2r+1}(a,b). 
\]
The set of $\A_{2r+1}$-singular points of a general position map of dimension $d$ is a submanifold of the source manifold of codimension $d+2r+1$. The closure of $\A_{2r+1}$-singular points of a Morin map is a closed submanifold. A singular point of type $\A_3$ is called a \emph{swallowtail} singular point.  

For a smooth map $f$, the set of singular points of $f$ of type $\A_i$ will be denoted by $\A_i(f)$. In general the closure of the set of Morin singular points of a smooth map $f:M\to N$ may not be a submanifold of $M$. However, as has been mentioned, for a Morin map $f$ of dimension $d>0$, there is a sequence of manifolds and embeddings 
\[
     M=\bar{\A}_0(f) \stackrel{\supset}\longleftarrow \bar{\A}_1(f) \stackrel{\supset}\longleftarrow \bar{\A}_2(f) \stackrel{\supset}\longleftarrow \bar{\A}_3(f) \stackrel{\supset}\longleftarrow \cdots, 
\]
where $\A_0(f)$ stands for the set of regular points of $f$. 

\subsection{Vector bundle morphisms associated with Morin map germs}\label{s:Boardman}

Let $f: M\to N$ be a Morin map of a manifold of dimension $m$ into a manifold of dimension $n$. It gives rise to a series of homomorphisms corresponding to partial derivatives of $f$. We recall some of these homomorphisms here (for details see the foundational paper by Boardman~\cite{Bo}). 

Over the set $\A_0(f)$ the differential $df$ of $f$ gives rise to an exact sequence of vector bundle morphisms:
\begin{equation}\label{e:3.3}
   0\longrightarrow K_0\longrightarrow TM\stackrel{df}\longrightarrow f^*TN\longrightarrow 0,
\end{equation}
where $TM$ and $TN$ stand for the tangent bundles of $M$ and $N$ respectively. The vector bundle $K_0$ over $\A_{0}(f)$, defined by the exact sequence (\ref{e:3.3}), is called the \emph{kernel} bundle, it is of dimension $d=m-n$. We emphasize that $K_0$ is \emph{not} defined over the manifold $M$, and all vector bundles in the exact sequence (\ref{e:3.3}) are over (or restrictions to) the open subset $\A_0(f)$ of $M$.  

The closure of the set $\A_1(f)$ of Morin singularities of $f$ is a closed submanifold of $M$ of dimension $n-1$. The differential $df$ of $f$ gives rise to an exact sequence of vector bundle morphisms
\begin{equation}
     0\longrightarrow K_1\longrightarrow TM \stackrel{df}\longrightarrow f^*TN\longrightarrow C_1 \longrightarrow 0
\end{equation}
over $\bar{\A}_1(f)$, where the canonical \emph{kernel} bundle $K_1$ is of dimension $m-n+1$ and the canonical \emph{cokernel} bundle $C_1$ is of dimension $1$. In local coordinates, the above exact sequence is defined by means of partial derivatives of $f$ of order $1$. The restriction 
\[
    f|\bar{\A}_1(f): \bar{\A}_1(f) \longrightarrow N
\]
is a smooth map of manifolds. Its singular set is again a smooth submanifold of $M$; in fact it coincides with the closure $\bar{\A}_2(f)$ of the set of cusp singular points of $f$. If the set $\A_2(f)$ is empty, then the map $f$ induces a non-degenerate symmetric bilinear form 
\begin{equation}\label{ee:3.4}
   K_1\otimes K_1\longrightarrow C_1.
\end{equation}
In particular, if $\dim K_1$ is odd, then $C_1$ is orientable , and there is a canonical orientation of $C_1$ with respect to which the index of the quadratic form associated with the bilinear form (\ref{ee:3.4}) is less than $\frac{1}{2}\dim K_1$.  

If the set $\A_2(f)$ is non-empty, then the map $f$ induces an exact sequence of vector bundle morphisms
\[
   0\longrightarrow K_2 \longrightarrow K_1 \longrightarrow \Hom(K_1,C_1),
\] 
over $\bar{\A}_2(f)$  by means of partial derivatives of $f$ of order $2$, where $K_2$ is a line bundle. In the standard coordinates listed above, the line bundle $K_2$ is spanned by $\frac{\partial}{\partial x}$. It is tangent to the submanifold of singular points of $f$. 

Over $\bar{\A}_2(f)$ there is a canonically defined non-degenerate symmetric bilinear form
\begin{equation}\label{eq:5}
   K_1/K_2 \otimes K_1/K_2 \longrightarrow C_1. 
\end{equation}
For $i>2$, the set $\bar{\A}_i(f)$ is a submanifold of $M$ and the set of singular points of the map  
\[
    f|\bar{\A}_i(f): \bar{\A}_i(f) \longrightarrow N
\]
coincides with $\bar{\A}_{i+1}(f)$. Over the set of $\A_i$ singular points, with $i\ge 2$, the map $f$ induces an isomorphism of line bundles
\[
   K_2\stackrel{\approx}\longrightarrow \Hom(\circ_i K_2, C_1),
\]
where $\circ_i K_2$ stands for the symmetric tensor product of $n$ copies of $K_2$. For details we refer to the paper of Boardman~\cite{Bo}.

\begin{remark}\label{r:orientation} For an oriented Morin map $f\colon M\to N$, the canonical kernel bundle $K_1$ over the set of regular points $\mathcal{A}_0(f)$ carries a canonical orientation since $K_1\approx TM\ominus f^*TN$. Similarly, if the cokernel bundle $C_1$ is oriented over $\mathcal{A}_r(f)$, then the kernel bundle over $\mathcal{A}_{r}(f)$ carries a canonical orientation since $K_1\approx TM\ominus f^*TN\oplus C_1$ over $\mathcal{A}_{r}(f)$. We will always choose the orientation of $C_1$ so that the index of the quadratic form of $f$ is less than $\frac{1}{2}\dim K_1$. 
\end{remark}

\addcontentsline{toc}{part}{Relative symmetry groups}

\section{Definition of relative symmetries}

We recall that $d$ is a fixed integer (we are primarily interested in the case $d>0$, but the definitions of this section are valid for a general integer $d$). Let $N$ be a non-negative integer such that $N+d>0$. For $i\ge 0$ let $\D(\R^i, 0)$ denote the group of diffeomorphism germs of $(\R^i, 0)$. Then the group 
\[
\D(\R^{N+d}, 0)\times \D(\R^{N}, 0)
\] 
is the group of \emph{coordinate changes} in the configuration space $\R^{N+d}\times \R^N$. The space of map germs 
\[
   (\R^{N+d}, 0) \longrightarrow (\R^N, 0). 
\]
admits a so-called right-left action of the group of coordinate changes; an element $(\alpha, \beta)$ of the group takes a map germ $h$ to the map germ $\beta\circ h\circ \alpha^{-1}$. The {\it symmetry group} of a map germ $h$ is a subgroup $\mathop{\mathrm{Stab}}(h)$ of the group of coordinate changes that consists of elements acting trivially on $h$. 

\begin{theorem}[J\"{a}nich-C.T.C. Wall, \cite{J}, \cite{Wa}] Let $f\colon \R^n\to \R^p$ be a finitely determined map germ. Then any compact subgroup of $\mathop{\mathrm{Stab}}(f)$ is contained in a maximal such group, and any two maximal compact subgroups are conjugate. 
\end{theorem}

The {\it relative symmetry group} of a map germ $h$ is a maximal compact subgroup of 
\[
     \mathrm{Stab}(h) \cap \{\ (\alpha, \beta)\in \D(\R^{N+d}, 0)\times \D(\R^{N}, 0)\ | \ \mathrm{det}(d\alpha|0)\cdot \mathrm{det}(d\beta|0)>0 \ \}.
\]
By definition, the {\it right representation}  (respectively, {\it left representation}) of a relative symmetry group $G$ is its representation on the source space $(\R^{N+d}, 0)$ (respectively, on the target space $(\R^{N},0)$). In other words, the right and left representations of the group of coordinate changes are the respective projections 
\[
\D(\R^{N+d}, 0)\times \D(\R^{N}, 0)\longrightarrow  \D(\R^{N+d}, 0)
\]
and
\[
\D(\R^{N+d}, 0)\times \D(\R^{N}, 0)\longrightarrow  \D(\R^{N}, 0)
\]
onto its factors. We say that a right (respectively, left) representation of the symmetry group is \emph{orientation preserving}, if the right (respectively, left) representation of every element of the symmetry group is orientation preserving. We note that the right (respectively, left) representation is orientation preserving if and only if the left (respectively, right) representation is. The image of a relative symmetry group $G$ under the right and left representations will be denoted by  $\mathcal{R}G$ and $\mathcal{L}G$ respectively. For an element $g\in G$, its image in $\mathcal{R}G$ and $\mathcal{L}G$ will be denoted by $\mathcal{R}_g$ and $\mathcal{L}_g$ respectively.

We will often use the Bochner local linearization theorem.

\begin{theorem}[Bochner, \cite{Boc}]\label{l:Boch} If $G$ is a compact group acting on $(\R^n, 0)$, then there exist local coordinates with respect to which $G$ acts linearly.  
\end{theorem}
% \cite{Boc}[page 206]

In the following sections we determine the relative symmetry groups and their right and left representations for $k$-jets of minimal Morin map germs.

\section{Preliminary lemmas}~\label{s:5}

Our method of computation of relative symmetry groups is different from the standard approach (c.f. Rim\'{a}nyi~\cite{Ri}). Our computation of the relative symmetry groups of Morin map germs will rely on a few lemmas below. 

Let us recall that a curve $\gamma: [0,1]\to \R^n$ is called {\it regular of order $m$} at a point $t$ if the vectors $\{\gamma'(t), \gamma''(t),..., \gamma^{(m)}(t)\}$ are linearly independent in $\R^n$. In particular, a regular curve possesses a canonical orthonormal frame called the {\it Frenet frame}, which is constructed from the $m$ derivatives by the Gram--Schmidt orthogonalization algorithm.  For example, the curve 
\[
   \delta\colon \R \longrightarrow  \R^r,
\]
\[
  \delta\colon t\mapsto (t, t^2, ..., t^r)
\]
is regular at $t=0$ of order $r$, and its Frenet frame is the standard one.

\begin{lemma}\label{l:1} For $r\ge 1$ the maximal compact subgroup $G$ of the symmetry group of the curve germ $\delta\colon t\mapsto (t, t^2, ..., t^r)$ at $0$ is $\Z_2$. 
\end{lemma}
\begin{proof}
Let $\hat G$ be the kernel of the homomorphism $G\to \Z_2$ that takes an element $g$ of $G$ to $1$ if and only if the action of $\mathcal{L}_g$ on $(\R, 0)$ reverses the orientation. Since $\hat G$ is a compact group, the differential of $\mathcal{L}_g$ is trivial for any $g\in \hat G$.  Since the curve $\delta$ is regular of order $r$ in a neighborhood of $0$, it possesses a Frenet frame preserved by each element of the symmetry group of $\delta$. Consequently the differential of $\mathcal{L}_g$ is also trivial for each $g\in \hat G$. On the other hand with respect to some coordinate system the action of $\hat G$ is linear. Hence $\hat G$ is trivial. Thus $G$ contains at most $1$ non-trivial element. Such an element exists, the action of its right representation is given by $t\mapsto -t$.   
\end{proof} 

\begin{lemma} For $r\ge 2$ the maximal compact subgroup $H$ of the symmetry group of the curve germ $\eta: x\mapsto (x^2, ..., x^r)$ at $0$ is $\Z_2$. 
\end{lemma}
\begin{proof} Every element of the symmetry group of the curve germ $\eta$ lifts to an element in the symmetry group of the curve germ $\delta$. Consequently, the group $H$ is a compact subgroup of the symmetry group of $\delta$. Hence we may choose the maximal compact subgroup $G$ so that $H$ is the subgroup of $G=\Z_2$. On the other hand $H$ contains a non-trivial element. Thus $H=\Z_2$.  
\end{proof}

\begin{lemma}\label{l:3} For $r\ge 1$ the left representation of the maximal compact subgroup of the symmetry group of $\A_r$ is at most $\Z_2$. 
\end{lemma}
\begin{proof} Let $f$ be the standard minimal map germ of an $\A_r$-singularity and let $\gamma$ be the curve of $\A_{r-1}$-singular points of $f$. Then in some coordinates, the curve $\gamma$ is given by $x\mapsto (x^2, ..., x^{r+1})$. In particular the left representation of the maximal compact subgroup of the symmetry group of $\A_r$ is a subgroup of the maximal compact subgroup $H$ of the symmetry group of the curve germ $\eta$, which is $\Z_2$. 
\end{proof}

We claim that for every element $h$ of the relative symmetry group of a Morin map germ $f$, the right representation $\mathcal{R}_h$ of $h$ completely determines the left representation $\mathcal{L}_h$ of $h$. In particular, the relative symmetry group of $f$ can be (and will be) described in terms of a subgroup of the orthogonal group acting on the source of $f$. 

\begin{lemma}\label{l:7.4} The right representation of the relative symmetry group of a Morin map germ is faithful. 
\end{lemma}
\begin{proof}
We may choose coordinates in which both the right and left representations of $h$ are linear. In particular, the left representation $\mathcal{L}_h$ of $h$ is determined by the action of $\mathcal{L}_h$ on any open set of vectors. On the other hand, the image of $f$ contains an open set, and if $y$ is in the image of $f$, then $\mathcal{L}_h(y)=f\circ \mathcal{R}_h(x)$ for any point $x$ with $f(x)=y$. Thus the right representation $\mathcal{R}_h$ completely determines $h$. 
\end{proof}

\section{Relative symmetry groups of Morin map germs}\label{s:rsg}
In this section we compute the relative symmetry groups of Morin map germs and their right representations. In fact, we will describe only the right representation of Morin map germs; in view of Lemma~\ref{l:7.4}, these completely determines the groups. Furthermore, since the right representation $\mathcal{R}G$ of the relative symmetry group $G$ of a minimal Morin map germ $(\R^m,0)\to (\R^n,0)$ is isomorphic to a subgroup of the orthogonal group $\O_m$ we will identify $G$ with a subgroup of $\O_m$.  

For every non-negative integer $n$, let $\O_n$ and $\SO_n$ denote the $n$-th orthogonal and special orthogonal groups respectively.  For non-negative integers $a,b$, we put 
\[
    [{\O}_a\times {\O}_b]^+=({\O}_a\times {\O}_b) \cap {\SO}_{a+b}.
\]
For $d>0$, in the group ${\O}_{2d}$ there are elements $h_d$ and $r_d$ that act on $\R^d\times \R^d$ by exchanging the factors and reflecting the first factor along a hyperplane respectively.  

\subsection{Regular germ $\A_0$} The canonical form of a minimal regular germ is given by projecting $(\R^d,0)$ to $(\R^0,0)$. Its relative symmetry group $G$ is $\SO_d$. The action of $\mathcal{R}G$ on on $(\R^d, 0)$ is given by the standard action of the special orthogonal group. 

\subsection{Fold germ $\A_1(a,b)$} Let $f$ be a minimal fold map germ of the form
\[
     f: (\R^{d+1}, 0) \longrightarrow (\R^1, 0),
\]
\[
     y=Q(a,b):= - x_1^2-\cdots - x_a^2 + x_{a+1}^2 + \cdots + x_{a+b}^2, \quad a+b=d+1,\  a\le b,
\]
where $(x_1,...,x_{d+1})$ and $(y)$ are the standard coordinates in $\R^{d+1}$ and $\R^1$.

\begin{lemma}\label{l:8.1} The relative symmetry group $G$ of $f$ is as described in Table~\ref{t:1}. The action of $\mathcal{R}G$ is orientation preserving if and only if $a\ne b$. 
\end{lemma}
\begin{proof} By Lemma~\ref{l:3}, there is an exact sequence of group homomorphisms
\[
   0\longrightarrow \hat{G}\longrightarrow G \longrightarrow \Z_2,
\] 
where $\hat G$ is the kernel of the left representation of $G$. By the Bochner Theorem (see Theorem~\ref{l:Boch}) we may choose coordinates in which the right action of $\hat G$ is linear. Since $G$ preserves the quadratic form $Q$ associated with $f$, the group $\hat G$ is isomorphic to a subgroup of $[\O_{2c}\times \O_{2c}]^+$. On the other hand the group $\hat G$ contains $[\O_{2c}\times \O_{2c}]^+$. Thus $\hat G$ coincides with $[\O_{2c}\times \O_{2c}]^+$. If $a\ne b$, then $G=\hat G$. If $a=b=2c$, where $c\ge 1$, then the left representation of $r_{2c}\circ h_{2c}\in G$ is not trivial, and therefore $G/\hat G=\Z_2$. Similarly, if $a=b=2c+1$ with $c\ge 1$, then the left representation of $h_{2c+1}\in G$ is not trivial, and therefore again $G/\hat{G}=\Z_2$. 
\end{proof}

\begin{table}
\caption{The relative symmetries of $\A_1(a,b)$.}\label{t:1}
\begin{tabular}{|c|c|}
\hline
 & $G<\O_{a+b}$ \\
\hline
$a\ne b$ & $[{\O}_a\times {\O}_b]^+$ \\
\hline
$a=b=2c$ &  $<[{\O}_{2c}\times {\O}_{2c}]^+, r_{2c}\circ h_{2c}>$ \\
\hline
$a=b=2c+1$ & $<[{\O}_{2c+1}\times {\O}_{2c+1}]^+, h_{2c+1}>$ \\
\hline
\end{tabular}
\end{table}

\subsection{Morin map germ $\A_{2r}(a,b)$, $r>0$} The canonical form of an $\A_{2r}(a,b)$ map germ $f$ of dimension $d$ is given by 
\[
  y_i = t_i, \quad i=1, \dots, 2r-1,
\]
\[
  y_{2r}= Q(a,b) + t_1x + t_2x^2 + \cdots + t_{2r-1}x^{2r-1} + x^{2r+1},
\]
where $a\le b$ and $a+b=d$, and $(x_1,...,x_d, x, t_1, ..., t_{2r-1})$ and $(y_1, ..., y_{2r})$ are the standard coordinates in $\R^{d+2r}$ and $\R^{2r}$ respectively. Put
\[
   \tau_{2r}: \R^{2r} \longrightarrow \R^{2r},
\]
\[
   (t_1, \dots, t_{2r-1}, x) \mapsto (t_1, -t_2, \dots, (-1)^{i+1}t_i, \dots, t_{2r-1}, -x).
\]

\begin{table}
\caption{The relative symmetries of $\A_{2r}(a,b)$.}\label{t:2}
\begin{tabular}{|c|c|}
\hline
 & $G<\O_{a+b}$ \\
\hline
$a\ne b$ & $[{\O}_a\times {\O}_b]^+$ \\
\hline
$a=b=2c$ &  $<[{\O}_{2c}\times {\O}_{2c}]^+, h_{2c}\times \tau_{2r} >$ \\
\hline
$a=b=2c+1$ & $<[{\O}_{2c+1}\times {\O}_{2c+1}]^+, (r_{2c+1}\circ h_{2c+1})\times \tau_{2r}>$ \\
\hline
\end{tabular}
\end{table}

\begin{table}
\caption{The relative symmetries of $\A_{2r+1}^{\pm}(a,b)$.}\label{t:3}
\begin{tabular}{|c|c|}
\hline
 & $G<\O_{a+b}$ \\
\hline
all $a,b$ & $<[{\O}_a\times {\O}_b]^+, h>$ \\
\hline
\end{tabular}
\end{table}

\begin{lemma}\label{l:8.2} The relative symmetry group $G$ of the $\A_{2r}(a,b)$ map germ $f$ is described in Table~\ref{t:2}. The action of $\mathcal{R}G$ is orientation preserving if and only if $r$ is even. 
\end{lemma}
\begin{proof}
In view of Lemma~\ref{l:3}, there is an exact sequence of homomorphisms
\[
   0\longrightarrow \hat{G}\longrightarrow G \longrightarrow \Z_2,
\] 
where $\hat G$ is the kernel of the left representation of $G$. Let us show that $\hat G$ is isomorphic to $[\O_a\times \O_b]^+$. We may choose coordinates in which the right action of $\hat G$ is linear. We recall that $\mathcal{R}G$ preserves the vector space $K_2$, while $\mathcal{R}\hat G$ acts trivially on the cokernel $C_1$ of $df$. Over the set $\A_{2r}(f)$ there is a canonical isomorphism 
\begin{equation}\label{equ:1.1}
    K_2\stackrel{\approx}\longrightarrow \Hom(\circ_{2r} K_2, C_1).
\end{equation}
Hence, the action of $\mathcal{R}\hat G$ is trivial on $K_2$, and, in particular, there is a well-defined action of $\mathcal{R}\hat G$ on the vector space $K_1/K_2$. The action of $\mathcal{R}\hat G$ on $K_1/K_2$ preserves the quadratic form $Q$ associated with the map germ $f$. Hence $\mathcal{R}\hat G|(K_1/K_2)$ is a subgroup of $[\O_a\times \O_b]^+$. On the other hand, $\mathcal{R}\hat G|(K_1/K_2)$ contains $[\O_a\times \O_b]^+$. Thus we have an exact sequence 
\[
   0\longrightarrow \tilde{G}\longrightarrow \hat G\longrightarrow [{\O}_a\times {\O}_b]^+ \longrightarrow 0, 
\]
where the third homomorphism is the right representation of $\hat G$ on $K_1/K_2$. There is a well-defined linear action of $\tilde G$ on the vector space $\R^{2r+a+b}/K_1$, which is the coimage of $df$ at $0$. Since the action of $\mathcal{L}\tilde G$ is trivial on the image $\mathop{\mathrm{Im}}(df)$, the action of $\mathcal{R}\tilde G$ is trivial on the coimage of $df$, and therefore $\tilde G$ is trivial. Consequently, $\hat G=[\O_a\times \O_b]^+$.

If $a=b=2c$ and $c>0$, then the left representation of $h_{2c}\times \tau_{2r}\in G$ is non-trivial and therefore $G/\hat G=\Z_2$. In the case $a=b=2c+1$ and $c>0$ the left representation 	is non-trivial for $(r_{2c+1}\circ h_{2c+1})\times \tau_{2r}$, and again, $G/\hat G=\Z_2$.  Finally, in the case $a\ne b$ we observe that, in view of the proof of Lemma~\ref{l:3}, every element $g$ in $G$ with non-trivial left representation acts non-trivially on $K_2$. Hence, in view of the isomorphism (\ref{equ:1.1}) the action of $\mathcal{L}_g$ on $C_1$ is non-trivial, which is impossible in the case $a\ne b$ since $G$ preserves the quadratic form $Q$ associated with $f$. Thus $G=\hat G$. 
\end{proof}

%Since the vectors $\partial/\partial y_i$, with $i=1, ..., 2r-1$, are preserved under the left action of $\tilde G$, %the coordinate vectors in $\R^{2r+a+b}$ corresponding to $\partial/\partial t_1, ..., \partial/\partial t_{2r-1}$ are  %preserved under the right representation of $\tilde G$. Hence $\tilde G$ is trivial. Consequently, $\hat G=[O_a\times %O_b]^+$.

\subsection{Morin map germ $\A^{\pm}_{2r+1}(a,b)$, $r>0$}\label{sub:6} The canonical form is given by 
\[
  y_i = t_i, \quad i=1, \dots, 2r,
  \]
  \[
  y_{2r+1}= Q(a,b) + t_1x + t_2x^2 + \cdots + t_{2r}x^{2r} \pm x^{2r+2}, 
\]
where $a\le b$, $a+b=d$, and $(x_1, ..., x_d, x, t_1, ..., t_{2r})$ and $(y_1, ..., y_{2r+1})$ are standard coordinates in $\R^{d+1+2r}$ and $\R^{2r+1}$ respectively. We recall that Morin map germs of types $\A^{+}_{2r+1}(a,b)$ and $\A^{-}_{2r+1}(a,b)$ are equivalent if and only if $a=b$.  Let $h$ denote an element in $\O_{d+1+2r}$ 
that on the first factor of $(\R^{a}\times \R^b\times \R^{2r+1})$ acts as an arbitrarily chosen orientation reversing element of $\O_a\times \O_b$ and maps 
\[
(t_1, \dots, t_{2r}, x) \mapsto (-t_1, t_2, \dots, (-1)^{i}t_i, \dots, t_{2r}, -x).
\]

\begin{lemma}\label{l:8.3} The relative symmetry group $G$ of a minimal Morin map germ $f$ of type $\A_{2r+1}$ is generated by $[\O_a\times \O_b]^+$ and the element $h$.  The action of $\mathcal{R}G$ is orientation preserving if and only if $r$ is even.
\end{lemma}
\begin{proof} It is easily verified that $h$ belongs to the relative symmetry group $G$ and $\mathcal{L}_h$ is non-trivial. Hence to complete the proof of Lemma~\ref{l:8.3} it suffices to show that the kernel $\hat{G}$ of the left representation of $G$ is isomorphic to $[\O_a\times \O_b]^+$.

For $\A_{2r+1}$ the argument of Lemma~\ref{l:8.2} fails to show that the action of $\mathcal{R}\hat{G}$ is trivial on $K_2$. On the other hand, the vector space $K_2$ is tangent at $0$ to the curve $\bar{\A}_{2r}(f)$, which outside of $0$ immerses by $f$ into $\R^{2r+1}$. Consequently, if the action of $\mathcal{R}\hat{G}$ is non-trivial on $K_2$, then the action of $\mathcal{L}\hat{G}$ takes $\bar{\A}_{2r}(f)$ into itself non-trivially. This contradicts the definition of $\hat{G}$. Thus, the action of $\mathcal{R}\hat{G}$ is trivial on $K_2$ and therefore there is an action of $\mathcal{R}\hat{G}$ on $K_1/K_2$. The rest of the argument is similar to the argument in the proof of Lemma~\ref{l:8.2}.    
\end{proof}

Let $f$ be an arbitrary minimal Morin map germ, and $G$ its relative symmetry group. The kernel $K$ of $df$ at $0$ is invariant under the action of $\mathcal{R}G$. In particular the restriction of the action of $\mathcal{R}G$ to the kernel $K$ is  a new representation of $G$ which we call the \emph{kernel representation} of the relative symmetry group $G$. 

\begin{lemma}\label{l:8.4} For every minimal Morin map germ $f$, the kernel representation of $G$ is faithful.  
\end{lemma}
\begin{proof} The statement of Lemma~\ref{l:8.4} immediately follows from the computation of relative symmetry groups $G$ and their right representations.  
\end{proof}

\addcontentsline{toc}{part}{Kazarian spectral sequence}

\section{Classifying loop space of $\A_i$-maps}\label{s:8}

We recall that an \emph{$\A_i$-map} is a smooth map of manifolds such that each of its singular points is of type $\A_j$ for some $j\le i$. The classifying loop space of $\A_i$-maps is defined to be the infinite loop space of a spectrum $\mathbf{A_i}$ for $i=0,1,...,\infty$. The $t$-th term of $\mathbf{A_i}$ is constructed by means of the universal oriented vector bundle $\xi_t: E_t\to \BSO_t$ of dimension $t$. Let $S_t=S_t(i)$ denote the space of maps $f\colon \R^{t+d}\to E_t$ such that 
\begin{itemize}
\item the image of $f$ belongs to a single fiber $E_t|b$ of $\xi_t$ over some point $b\in \BSO_t$,
\item $f(0)$ is the zero in the vector space $E_t|b$, and
\item the map germ $f\colon (\R^{t+d}, 0)\to (E_t|b, 0)$ is an $\A_i$-map germ.  
\end{itemize}
The space $S_t$ is endowed with an obvious topology so that the map $\pi_t: S_t\to \BSO_t$ that takes $f$ to $b$ has a structure of a fiber bundle. By definition the $t$-th term $[\mathbf{A_i}]_t$  of the classifying loop space of $\A_i$-maps is the Thom space $\Th\pi_t^*\xi_t$ of the fiber bundle $\pi_t^*\xi_t$
over $S_t$. In other words, the spectrum in question is defined to be the 
Thom spectrum with the $t$-th term $\Th\pi_t^*\xi_t$. To simplify notation we will write $\mathbf{A}$ for $\mathbf{A}_{\infty}$. 

The spectrum $\mathbf{A_i}$ comes with a natural filtration; it is filtered by the subspectra $\mathbf{A_j}$ with $j\le i$. The spectral sequence yielded by the filtration is called the {\it Kazarian spectral sequence}. 

Next we give a description of the filtration; this description is essentially based on an observation of Kazarian (e.g., see \cite{Kaz3}). For each $i$ there are commutative diagrams
\[
   \begin{CD}
   S_t(i)@>j_t>> S_{t+1}(i)\\
   @V\pi_tVV @VV\pi_{t+1}V \\
   \BSO_t @>i_t>>\BSO_{t+1}.
   \end{CD}
\]
The map $i_t$ in the diagram is the canonical inclusion, while $j_t$ sends a map germ $f\colon \R^{t+d}\to E|b$ to the map germ of the composition
\[
   \R^{t+d+1}\equiv \R^{t+d}\times \R\stackrel{f\times \id_{\R}}\longrightarrow E_t|b\times \R\stackrel{\approx}\longrightarrow E_t|i_t(b),
\] 
in which we identify the fiber $E_t|b\times \R$ of the vector bundle $\xi_t\oplus\varepsilon$ over $b$ with the fiber $E_t|i_t(b)$ of the vector bundle $\xi_{t+1}|\BSO_t$ (for details see \cite{Sa}). Here and below $\varepsilon$ stands for the trivial line bundle over an appropriate space. 

Kazarian observed a beautiful fact that the direct limit $S(i)$ of spaces $S_t(i)$ with respect to inclusions $j_t$ is the disjoint union of spaces $S(\alpha)$, one for each $\A_i$-singularity type $\alpha$. In its turn, for each singularity type $\alpha$, the corresponding space $S(\alpha)$ is the classifying space of the relative symmetry group of the minimal map germ of type $\alpha$ (see Theorem~\ref{th:11.2}). This allows us to give a nice description of the spectrum $\mathbf{A_i}$ in terms of stable vector bundles. 

\begin{definition}\label{d:v}
A \emph{stable vector bundle} of dimension $n$ over a CW complex $X$ is an equivalence class $\xi\ominus \eta$ of pairs $(\xi, \eta)$ of vector bundles over $X$ with $n=\dim \xi-\dim \eta$. The equivalence relation is generated by the equivalences of the form $(\xi,\eta)\sim (\xi',\eta')$ for vector bundles with $\xi\oplus\eta'=\xi'\oplus \eta$. For any choice of $n$, stable vector bundles of dimension $n$ over $X$ are in bijective correspondence with homotopy classes $[X, \BO]$. Oriented stable vector bundles of dimension $n$ are defined similarly. Over $X$ these are in bijective correspondence with homotopy classes $[X, \BSO]$. 

Thus we may identify a stable vector bundle of dimension $-d$ with a (homotopy class of a) map $f\colon B\to \BSO$. Such a map determines a cohomology theory $h^*$; the $(n+d)$-th term of the spectrum of $h^*$ is the Thom space of the vector bundle $(f|B_n)^*\xi_n$ where $B_n=f^{-1}(\BSO_n)$. 
\end{definition}

Thus, the spectrum $\mathbf{A_i}$ is equivalent to the spectrum corresponding to a stable vector bundle $f\colon \sqcup S(\alpha)\to \BSO$ where $\alpha$ ranges over $\A_i$-singularity types. Our next step is to determine the map $f$. In the next section we determine the restriction $f|S(\alpha)$ where $\alpha$ is one of the $\A_i$-singularity types.

\section{Kazarian Theorem}\label{s:kaz}

 It follows from an observation due to Kazarian that for a sufficiently big $t$, the factor space $[\mathbf{A_j}]_t/[\mathbf{A_{j-1}}]_t$ is an approximation of the Thom space of a vector bundle over the classifying space $BG$, where $G$ is the relative symmetry group of $\A_j$-map germ. The Kazarian observation greatly facilitates
the computation of cohomology groups of the classifying loop spaces of maps with prescribed singularities. 

Let $\alpha$ be an $\A_i$ singularity type. Its minimal map germ is represented by a map $\R^{n+d}\to \R^n$ for some integer $n$. Let $G$ denote the relative symmetry group of $\alpha$. Let $\xi=\xi_{\alpha}$ denote the vector bundle over $S(\alpha)$ of dimension $d+n$ associated with the right representation of the group $G$. Similarly, let $\eta$ denote the vector bundle over $S(\alpha)$ of dimension $n$ associated with the left representation of $G$. We also have counterparts of $\xi$ and $\eta$, denoted by the same symbols, over the space $BG_{\alpha}$. 

Given a vector bundle $\xi$ over a space $X$ and a map $e: Y\to X$, we will often write $\xi$ for the fiber bundle $e^*\xi$ in order to simplify the notation. We put
\begin{equation}\label{eq:11.1}
   {\BO}_{\infty+n+d}\times {\BO}_{\infty+n}\colon = \lim_{t\to\infty} {\BO}_{t+n+d}\times {\BO}_{t+n}.  
\end{equation}
There is a double cover 
\begin{equation}\label{eq:11.2}
   {\BO}_{\infty+n+d}\tilde{\times} {\BO}_{\infty+n}  \longrightarrow    {\BO}_{\infty+n+d}\times {\BO}_{\infty+n}  
\end{equation}
classified by the first Stiefel-Whitney class $w_1\times 1 +1\times w_1$, its total space is the classifying space $BH$ of the group 
\[
H=[{\O}_{\infty+n+d}\times {\O}_{\infty+n}]^+\colon= \lim_{t\to \infty} [{\O}_{t+n+d}\times {\O}_{t+n}]^+. 
\]
Let $L$ denote the space of all equivalence classes of map germs $(\R^{n+d+t},0)\to (\R^{n+t},0)$ of type $\alpha$ under the equivalence $f\sim f\times \id_{\R}$. Thus an element in $L$ is a map germ 
\[
(\R^{\infty+n+d},0)\to (\R^{\infty+d})
\] 
represented by a product $f\times \id$ of a map germ $f$ on a finite dimensional vector space and the identity map germ $\id$ on an infinite vector space. The group $H$ acts on $L$ by right-left coordinate changes, and the Borel construction yields a fiber bundle 
\begin{equation}\label{eq:11.3}
   EH\times_{H} L \longrightarrow {\BO}_{\infty+n+d}\tilde{\times} {\BO}_{\infty+n},
\end{equation}  
where $EH$ is the total space of the universal principle $H$-bundle. 

The next Kazarian theorem is well-known; we recall its short proof for completeness.  

\begin{theorem}[Kazarian]\label{th:Kazarian} The space $EH\times_H L$ is homotopy equivalent to the classifying space $BG_{\alpha}\times \BO$. 
\end{theorem}
\begin{proof} The space $L$ is the orbit of the standard minimal map germ $f$ of type $\alpha$ under the action of the group of all right-left coordinate changes $(\alpha, \beta)$ that preserve the relative orientation, i.e.,  
\[
\mathop{\mathrm{det}}d\alpha\cdot \mathop{\mathrm{det}}d\beta>0.
\] 
Since $H$ is a topological deformation retract of this group, the orbit $L'=Hf$ of the map germ $f$ is a deformation retract of $L$. In particular, there is a homotopy equivalence 
\[
EH\times_{H}L\simeq EH\times_H L'.
\] 
The stabilizer of $f$ under the action of $H$ is the group $G_{\alpha}\times \O$. Consequently, 
\[
EH\times_{H} L\simeq EH\times_{H} L'= EH/(G_{\alpha}\times {\O})\simeq BG_{\alpha}\times \BO,
\] 
since $L\simeq L'=H/(G_{\alpha}\times {\O})$. 
\end{proof}

We note that in view of the argument in the proof of the Kazarian Theorem, the map (\ref{eq:11.3}) factors by
\[
     BG_{\alpha}\times \BO \xrightarrow{(\xi, \eta) \times \id_{\BO}} ({\BO}_{n+d}\times {\BO}_n)\times \BO\xrightarrow{(\xi_{n+d}\oplus \gamma)\tilde{\times}(\xi_n\oplus\gamma)}  {\BO}_{\infty+n+d}\tilde{\times} {\BO}_{\infty+n},
\]
where $\gamma$ denotes the universal vector bundle over the space $\BO$. 
Let us consider a commutative diagram
\[
\begin{CD}
BG_{\alpha}\times \BO @>p >> S(\alpha) \\
@VVV @V\lim \pi_tVV \\
{\BO}_{\infty+n+d}\tilde{\times} {\BO}_{\infty+n} @>\eta\ominus\xi>> \BSO,
\end{CD}
\]
where the right vertical map $\lim \pi_t$ is the direct limit of projections $\pi_t$. The bottom horizontal map is the one classifying the stable vector bundle $\eta\ominus \xi$ of dimension $-d$. Finally, the top horizontal map $p$ takes a map germ
\[
     f\times \id_{\gamma}\colon (\xi\oplus \gamma)_x\longrightarrow (\eta\oplus \gamma)_x
\]
over $x\in BH$ to the map germ 
\[
    f\times \id_{\xi^{\perp}}\colon \R^{\infty}=\xi_x\oplus \xi_x^{\perp} \longrightarrow \eta_x\oplus \xi_x^{\perp},
\]
where $\xi^{\perp}$ is the negative of $\xi$; over a (finite dimensional) plane $l$ in $\R^{\infty+n+d}$ representing a point in $\BO_{\infty+n+d}$ the fiber of $\xi$ consists of vectors in $l$, while the fiber of $\xi^{\perp}$ consists of vectors in $\R^{\infty+n+d}$ orthogonal to $l$. The projection $p$ is a trivial fiber bundle with fiber $\BO$. Consequently, by taking the restriction of $p$ to a slice $BG_{\alpha}\times\{*\}$, we deduce the following form of the Kazarian Theorem.

\begin{theorem}\label{th:11.2}
The spaces $BG_{\alpha}$ and $S(\alpha)$ are homotopy equivalent.   
\end{theorem}

\begin{remark} The homotopy equivalence that we constructed should be compared with the homotopy equivalence in the paper \cite{Sz}. The homotopy equivalence in \cite{Sz} is not suitable for our purpose; it is not compatible with canonical vector bundles over $BG_{\alpha}$ and $S(\alpha)$, and in particular, the triangular diagram below in the proof of Lemma~\ref{l:11.3} is not homotopy commutative if taken with the slanted map from the paper \cite{Sz}. 
\end{remark}

\begin{lemma}\label{l:11.3} The map $f|S(\alpha)\colon S(\alpha)\to \BSO$ is classifying the stable vector bundle $\eta\ominus \xi$ of dimension $-d$.  
\end{lemma}
\begin{proof}
We have constructed a commutative diagram of maps 
\[
\xymatrix{
                                           &       &  S(\alpha)      \ar[d]^{f|S(\alpha)} & \\
                                   & BG_{\alpha}\ar[r]^{\eta\ominus \xi} \ar[ur]   &  \BSO,                    &
}
\]
where the horizontal map $\eta\ominus \xi$ is the map classifying the stable vector bundle $\eta\ominus \xi$ of dimension $-d$, while the slanted map takes a map germ $f: \xi_x\to \eta_x$ to the map germ 
\[
       \R^k=\varepsilon^k_x\stackrel{\equiv}\longrightarrow \xi_x\oplus \xi_x^{\perp} \xrightarrow{f_x\times \id} \eta_x\oplus \xi_x^{\perp} 
\]
where $\id$ is the identity map of the vector space $\xi_x^{\perp}$. Since the slanted map is a homotopy equivalence constructed (see Theorem~\ref{th:11.2}), this concludes the proof of Lemma~\ref{l:11.3}.
\end{proof}

\begin{theorem}\label{th:16.5} The normal bundle of $S(\alpha)$ in $S(i)$ is isomorphic to $\xi$. 
\end{theorem}
\begin{proof}
Let $E_{\xi}$ and $E_{\eta}$ denote the total spaces of the fiber bundles $\xi$ and $\eta$ over $BG_{\alpha}$ respectively.  Each point $x$ of $BG_{\alpha}$ comes with a minimal map germ $E_{\xi}|x\to E_{\eta}|x$ of type $\alpha$. We may choose a map $u_x$ representing this map germ for each point $x$ in $BG_{\alpha}$ so that we get a continuous family $u=\{u_x\}$ of maps,
\[
    u\colon E_{\xi}\longrightarrow E_{\eta},
\]  
whose restriction to the fiber of $\xi$ over each point $x$ of the base is a smooth map into the fiber of $\eta$ over $x$. By the construction in the proof of Lemma~\ref{l:11.3} we obtain a commutative diagram
\[
\xymatrix{
                                           &       &  S(i)      \ar[d]^{f|S(i)} & \\
                                   & E_{\xi}\ar[r]^{\eta\ominus \xi} \ar[ur]^{s}   &  \BSO.                    &
}
\]
The slanted map $s$ takes each fiber $E_{\xi}|x$ of $\xi$ to a fiber of $f|S(i)$. Furthermore, since $u_x$ is an $\A(i)$-map, the restriction of $s$ to $E_{\xi}|x$ is transversal to the stratum $S(\alpha)$ in $S(i)$. Consequently, the map $s$ pulls the normal vector bundle of $S(\alpha)$ in $S(i)$ to $\xi$.

\end{proof}

\section{The first term of the Kazarian spectral sequence}

In what follows all cohomology groups are with coefficients in a field $\k$, where $\k$ is either $\Q$ or $\Z/p\Z$ for some prime $p\ne 2$. In this section we compute the first term of the Kazarian spectral sequence of Morin maps associated to the filtration 
\[
    \mathbf{A_0} \subset \mathbf{A_1} \subset \mathbf{A_2} \subset \cdots
\]
where $\mathbf{A_r}$ corresponds to Morin singular points of types $\A_{i}(a,b)$ and $\A^{\pm}_i(a,b)$ with $i\le r$. 

Given an $\A_r$ singularity type $\alpha$, we recall that $\xi_{\alpha}$ stands for the vector bundle over $BG_{\alpha}$ associated with the right representation of the relative symmetry group $G_{\alpha}$. The Thom space of $\xi_{\alpha}$ is denoted by $\Th\xi_{\alpha}$.

\begin{lemma}\label{l:12.1}
The first term of the Kazarian spectral sequence is of the form 
\[
   E_1^{0,i} = H^i({\BSO}_q), 
\]
\[
    E_1^{r,i} = \mathop{\oplus}_{\alpha} \tilde{H}^{i+r}(\Th\xi_{\alpha})
\]
for $r>0$, where $\alpha$ ranges over $\A_r$ singularity types.     
\end{lemma}
\begin{proof} Let us, for example, compute $E_1^{r,i}$ for $r>0$ and some $i\ge 0$. To this end, let us recall that the $t$-th terms of spectra $\mathbf{A}_r$ and $\mathbf{A}_{r-1}$ are Thom spaces over spaces
$S_{t,r}$ and $S_{t,r-1}$ respectively. So, 
\[
   E_1^{r,i} = H^{i+r}(\mathbf{A_r},\mathbf{A_{r-1}})=\mathrm{lim}\, H^{i+r+t}([\mathbf{A_r}]_t, [\mathbf{A_{r-1}}]_t) 
   \]
   \[
  = \mathrm{lim}\, H^{i+r}(S_{t,r}, S_{t,r-1})= H^{i+r}(\mathrm{lim}\, S_{t,r}, \mathrm{lim}\, S_{t,r-1})= \oplus_{\alpha}\, \tilde{H}^{i+r}(\Th\xi_{\alpha}),
\]
where the last equality follows from the Theorems~\ref{th:11.2} and \ref{th:16.5}. 
\end{proof}

%In our case for each $r$ the indexing set of singularities is given by the set of pairs $(a,b)$ and the $\pm$ %superscript in the case of $A^{\pm}_{2r+1}(a,b)$. Let us list the corresponding cohomology groups. To this end, let us %remind the cohomology groups of the classifying spaces $\BO_n$ and $\BSO_n$. 

In the rest of the section we will compute the cohomology groups $\tilde{H}^{*}(\Th\xi_{\alpha})$. Two preliminary reminders are in order. 

\subsection{Smith exact sequence}

Let $S$ be a subgroup of a group $G$ of index $2$. Then the inclusion $S\to G$ gives rise to a double covering $p: BS\to BG$ of classifying spaces. We may replace $BG$  by a triangulated space weakly homotopy equivalent to $BG$, e.g., by the geometric realization of the singular simplicial set of $BG$; and assume that $BS$ is also triangulated so that $p$ takes simplices of $BS$ isomorphically to simplices of $BG$. Let $\xi$ be a vector bundle over $BG$. We may choose a CW structure on the Thom space of $\xi$ so that every cell in $\Th\xi$ is a cone over a simplex in $BG$. Similarly we may introduce a CW structure on the Thom space of $p^*\xi$. To simplify the notation we will denote the singular covering $\Th p^*\xi\to \Th \xi$ by the same symbol $p$. Then there is a commutative diagram 
\[
\xymatrix{
 &                      &  \tilde H^*(\Th p^*\xi) \ar[dr]^{\mathop{\mathrm{Tr}}}& \\
  &  \tilde H^*(\Th \xi) \ar[ur]^{p^*} \ar@{->}[rr]^{\times 2} &  &\tilde H^*(\Th \xi),   & 
}
\]
of cohomology groups with coefficients in a field of characteristic $\ne 2$. In the diagram $\mathop{\mathrm{Tr}}$ stands for the transfer homomorphism, which descends from a homomorphism of cochains; given a cochain $c$ on $\Th p^*\xi$, the value of the cochain $\mathop{\mathrm{Tr}}(c)$ on a cell $\Delta$ in $\Th \xi$ over a simplex in $BG$ is the sum of values of $c$ on the two cells $p^{-1}(\Delta)$ in $\Th p^*\xi$.   
In particular, the map $p^*$ is a monomorphism. 

On the other hand, let $\chi$ denote the self map of the Thom space of $p^*\xi$ that exchanges the sheets of the singular covering $p$. We note that $\chi$ maps cells into sells. Hence there is a short exact sequence 
\[
     0\longrightarrow C^*_{eq}(\Th p^*\xi)\longrightarrow C^*(\Th p^*\xi)\longrightarrow C^*_{skew}(\Th p^*\xi) \longrightarrow 0,
\]
where the second homomorphism is the inclusion of the subgroup of $\chi$-equivariant cochains and the third homomorphism is a projection that takes a cochain $x$ to $x-\chi$. The corresponding long exact sequence of cohomology groups splits since $C^*_{eq}(\Th p^*\xi)\approx C^*(\Th \xi)$ and the map $p^*$ is a monomorphism. Thus, there is a Smith exact sequence 
\[
   0\longrightarrow \tilde H^*(\Th \xi)\longrightarrow \tilde H^*(\Th p^*\xi)\longrightarrow \tilde H^*_{skew}(\Th p^*\xi) \longrightarrow 0.
\]  

In particular the following lemma takes place.

\begin{lemma}\label{l:4} The group $\tilde{H}^*(\Th \xi)$ can be identified with a subgroup of the group $\tilde{H}^*(\Th p^*\xi)$ so that an element $x$ of $\tilde{H}^*(\Th p^*\xi)$ is in $\tilde{H}^*(\Th \xi)$ if and only if $\chi^*x=x$. 
\end{lemma}

We will also need the information about rational cohomology groups of the classifying spaces $\BO_n$ and $\BSO_n$. 

\begin{lemma} The rational cohomology groups of $\BO_n$ are given by 
\[ 
H^*({\BO}_n; \Q)=\Q[p_1, ..., p_{\left\lfloor n/2\right\rfloor}].
\] 
The rational cohomology groups of $\BSO_n$ are given by 
\[
H^*({\BSO}_n; \Q)=\Q[p_1, ..., p_{\left\lfloor n/2\right\rfloor}]+ e_n\Q[p_1, ..., p_{\left\lfloor n/2\right\rfloor}],
\]
where $e_n$ is the rational Euler class of the universal vector bundle. The class $e_n$ is trivial if $n$ is odd, and $e_n^2=p_{n/2}$ if $n$ is even. 
\end{lemma}

\subsection{Fold germ $\A_1(a,b)$, with $a+b=d+1$.} 
Let 
\[
\P(a,b)=\k[p_1, \dots , p_{\left\lfloor a/2\right\rfloor}, p'_1, \dots, p'_{\left\lfloor b/2\right\rfloor}]
\] 
denote the polynomial ring in terms of classes $p_i$ and $p'_i$ of degree $4i$; and let $\SP(a,b)$ and $\AP(a,b)$ be the submodules of $\P(a,b)$ that consist of polynomials respectively invariant and skew invariant under the substitution exchanging $p_i$ with $p'_i$. We will also denote the polynomial ring $\k[p_1, \dots , p_{\left\lfloor a/2\right\rfloor}]$ by $P(a)$. 

Let $G$ denote the relative symmetry group of a minimal $\A_1(a,b)$-germ.  Let $\xi_{a,b}$ denote the universal vector $G$-bundle over the classifying space $BG$. 
%If $a\ne b$, then the vector bundle $\xi_{a,b}$ is orientable and, by the Thom isomorphism, 
%\[
%\tilde H^{i+1}(T\xi_{a,b})= \tilde H^{i-q}(B[O_a\times O_b]^+).
%\]  
The vector bundle $\xi_{0, d+1}$ is isomorphic to the universal vector bundle of dimension $d+1$, and therefore  
\[
    \tilde{H}^*(\Th\xi_{0,d+1})=U_{0,d+1,1}\cup (\mathcal{P}(d+1)\oplus e_{0,d+1}\mathcal{P}(d+1)), 
\]
where $U_{0,d+1,1}$ and $e_{0,d+1}$ stand for the Thom and Euler classes of the universal vector bundle respectively. 

\begin{remark} At the moment we assume that all orientable vector bundles under consideration are oriented; we do not specify orientations though. However, in later sections, we will choose canonical orientations, and thus explicitly specify our choices of the Thom and Euler classes for oriented bundles under consideration. 
\end{remark}

In what follows we will assume that $a\ne 0$ and describe the cohomology group of $\Th \xi_{a,b}$ in terms of an essentially simpler cohomology group of the Thom space $\Th \mathfrak{o}$ of the universal vector $\SO_a\times \SO_b$-bundle $\mathfrak{o}=\xi^{so}_{a,b}$. 
%Indeed, in the case $a\ne b$, in view of the subgroup $SO_a\times SO_b$ of $[O_a\times O_b]^+$ of index two, the claim %is a direct consequence of Lemma~\ref{l:4}. In the case $a=b$, we first apply Lemma~\ref{l:4} to the subgroup %$SO_a\times SO_b$ of $[O_a\times O_b]^+$ and then to the subgroup $[O_a\times O_b]^+$ of $G$.  
Let 
\[
U_{a,b,1}\in H^{d+1}(\Th \mathfrak{o}) \qquad \mathrm{and} \qquad e_{a,b}\in H^{d+1}(B\mathfrak{o})
\]
denote the Thom class and the Euler class of the oriented vector bundle $\mathfrak{o}$ respectively. We recall that $e_{a,b}$ is trivial in the case where either $a$ or $b$ is odd. By the relative K\"unneth formula for the Thom space $\Th \mathfrak{o}$ we have
\[
   \tilde{H}^*(\Th \mathfrak{o})=\tilde{H}^*(\Th \xi_{a}\wedge \Th \xi_{b})=\tilde{H}^*(\Th \xi_{a})\otimes \tilde{H}^*(\Th \xi_{b}),
\]
where $\xi_a$ and $\xi_b$ denote the universal vector $\SO$-bundles over the classifying spaces $\BSO_a$ and $\BSO_b$ respectively. 

\begin{lemma}\label{l:10} Under the above notation, 
\begin{itemize} 
\item if $a\ne b$, then  $\tilde H^*(\Th \xi_{a,b})= U_{a,b,1}\s (\P(a,b)\oplus e_{a,b}\P(a,b))$,
\item if $a=b$, then $\tilde H^{i+1}(\Th \xi_{a,b})=U_{a,b,1}\s (\AP(a,b)\oplus e_{a,b}\SP(a,b))$, 
\end{itemize}
where $e_{a,b}$ is trivial whenever either $a$ or $b$ is odd. 
\end{lemma}
\begin{proof}
Let $U_a$ and $U_b$ denote the classes of the cohomology group $H^*(\Th \xi_a)\otimes H^*(\Th \xi_b)$ given by the Thom classes of the oriented bundles $\xi_a$ and $\xi_b$ respectively. Here and below we identify $U_a$ with $U_a\otimes{1}$, $U_b$ with ${1}\times U_b$ and the vector bundles $\xi_a$ and $\xi_b$ with their obvious lifts. 
By the Thom isomorphism, the group 
$H^*(\Th \xi_a)\otimes H^*(\Th \xi_b)$ in positive degrees is isomorphic to the group $H^*(\BSO_a)\otimes H^*(\BSO_b)$ generated by the Euler classes $e_a$ and $e_b$ and the Pontrjagin classes $p_i$ and $p_i'$ of $\xi_a$ and $\xi_b$ respectively.

Let us consider the case where $a\ne b$, in which case the group $G$ is isomorphic to the group $[\O_a\times \O_b]^+$, and therefore $\tilde H^*(\Th \xi_{a,b})$ is isomorphic to the subgroup of 
\[
\tilde H^*(\Th \mathfrak{o})=[U_a\s (P(a)\oplus e_aP(a))]\otimes [U_b\s (P(b)\oplus e_bP(b))]
\] 
that consists of elements invariant under the homomorphisms induced by the map $\beta$ on $\Th \xi_a\wedge \Th \xi_b$ given by the involutions on the two factors at the same time. We have
\[
   \beta^*U_{a}=-U_{a}, \quad \beta^*U_b=-U_b, \quad \beta^*e_a=-e_a, \quad \beta^*e_b=-e_b, \quad  \beta^*p_i=p_i, \quad \beta^*p_i'=p_i'.
\]
We observe that the class $U_{a,b,1}$ can be identified with the class $U_a\otimes U_b$ in the group 
$H^*(\Th \xi_a)\otimes H^*(\Th \xi_b)$. 
Similarly, the class $e_{a,b}$ can be identified with the class $e_a\otimes e_b$ in 
$H^*(\BSO_a)\otimes H^*(\BSO_b)$.
Hence 
\[
\beta^*(U_{a,b,1})=\beta^*(U_a\otimes U_b)=-U_a\otimes -U_b=U_a\otimes U_b=U_{a,b,1},
\]
\[
\beta^*(e_{a,b})=\beta^*(e_a\otimes e_b)=-e_a\otimes -e_b=e_a\otimes e_b=e_{a,b}.
\]
Now the elements in $\tilde H^*(\Th \xi_{a,b})$ invariant with respect to the action of $\beta^*$ are easily determined; these are listed in the statement of Lemma~\ref{l:10}.   

Let us prove the claim in the case where $a=b$ and $a$ is even. In this case the symmetry group $G$ is generated by the group $[\O_a\times \O_b]^+$ and the element that acts on $\R^a\times \R^b$ by exchanging the factors and reflecting the first factor along a hyperplane. By applying Lemma~\ref{l:4} twice, we obtain that $\tilde H^*(\Th \xi_{a,b})$ is isomorphic to the subgroup of 
\[
\tilde H^*(\Th \mathfrak{o})=[U_a\s (P(a)\oplus e_aP(a))]\otimes [U_b\s (P(b)\oplus e_bP(b))]
\] 
that consists of elements invariant under the homomorphisms induced by the maps $\alpha, \beta: \Th \mathfrak{o}\to \Th \mathfrak{o}$, where $\alpha$ is the map induced by the composition 
\[
   {\BSO}_a\times {\BSO}_b\longrightarrow {\BSO}_b\times {\BSO}_a \longrightarrow {\BSO}_b\times {\BSO}_a,
\]
where the first map exchanges the factors, and the second map is given by the involution on the first factor; the map $\beta$ is given by the involutions on the two factors at the same time. The actions of $\alpha^*$ and $\beta^*$ are summarized in the first three columns of Table~\ref{t:ab1}, where, for example, 
\[
\alpha^*(U_{a,b,1})=\alpha^*(U_a\otimes U_b)=-U_b\otimes U_a=-U_a\otimes U_b=-U_{a,b,1},
\]
\[
\alpha^*(e_{a,b})=\alpha^*(e_a\otimes e_b)=-e_b\otimes e_a=-e_a\otimes e_b=-e_a\otimes e_b=-e_{a,b}.
\]
Again, an easy verification shows that the elements in $\Th \mathfrak{o}$ invariant with respect to the action of $\alpha$ and $\beta$ are those listed in the statement of Lemma~\ref{l:10}.

Finally, let us prove the claim in the case where $a=b$ and $a$ is odd. In this case the symmetry group $G$ is generated by the group $[\O_a\times \O_b]^+$ and the element that acts on $\R^a\times \R^b$ by exchanging the factors. We deduce that $\tilde H^*(\Th \xi_{a,b})$ is isomorphic to the subgroup of 
\[
\tilde H^*(\Th \mathfrak{o})=[U_a\s P(a)]\otimes [U_b\s P(b)]
\] 
that consists of elements invariant under the homomorphisms induced by the maps $\alpha, \beta: \Th \mathfrak{o}\to \Th \mathfrak{o}$, where $\alpha$ is the map on $\Th \xi_a\wedge \Th \xi_b$ exchanging the two factors, while the map $\beta$ is given by the involutions on the two factors at the same time. The action of $\alpha^*$ and $\beta^*$ on cohomology classes of $\Th \mathfrak{o}$ is described in the forth and fifth columns of Table~\ref{t:ab1}, 
which allows us to verify the statement of Lemma~\ref{l:10} in the case where $a=b$ and $a$ is odd.

\end{proof}

\begin{table}
\caption{}\label{t:ab1}
\begin{tabular}{|c||c|c||c|c||c|c|}
\hline 
    & \multicolumn{2}{|c||}{$\A_1, a=b=2s$ }     & \multicolumn{2}{|c||}{$\A_1, a=b=2s+1$ } & $\A_1, a\ne b$   \\
\hline
       &  $\alpha$ & $\beta$ & $\alpha$ & $\beta$ & $\beta$  \\
\hline
\hline
$U_a$       & $-U_b$       & $-U_a$      & $U_b$        & $-U_a$      &  $-U_a$  \\
 \hline
$U_b$       & $U_a$        & $-U_b$      & $U_a$        & $-U_b$      &  $-U_b$  \\
 \hline
$e_a$       & $-e_b$       & $-e_a$      &       0     &  0      &  $-e_a$  \\
 \hline
$e_b$       & $e_a$        & $-e_b$      &      0      &   0    & $-e_b$ \\
 \hline
$U_{a,b,1}$ & $-U_{a,b,1}$ & $U_{a,b,1}$ & $-U_{a,b,1}$ & $U_{a,b,1}$ &  $U_{a,b,1}$  \\
 \hline
$e_{a,b}$   & $-e_{a,b}$   & $e_{a,b}$   &        0      & 0    & $e_{a,b}$ \\
 \hline
$p_i$       & $p_i'$       & $p_i$       & $p_i'$       & $p_i$       & $p_i$  \\
 \hline
$p_i'$      & $p_i$        & $p_i'$      & $p_i$        & $p_i'$      &  $p_i'$   \\
 \hline
 \end{tabular}
 \end{table}

\subsection{Morin map germ $\A_{2r}(a,b)$, with $r>0$, $a+b=d$.}\label{s:16.3}

Let $G$ denote the maximal compact subgroup of a minimal $\A_{2r}(a,b)$-germ.  Let $\xi_{a,b}$ denote the universal vector $G$-bundle over the classifying space $BG$. Again, we will describe the cohomology group of $\Th \xi_{a,b}$ in terms of an essentially simpler cohomology group of $\Th \mathfrak{o}$, where $\mathfrak{o}$ stands for the Whitney sum of the universal vector $\SO_a\times \SO_b$-bundle and the trivial $2r$-bundle over $B[\SO_a\times \SO_b]$. Let 
\[
U_{a,b,2r}\in H^{a+b+2r}(\Th \mathfrak{o}) \quad \mathrm{and}\quad e_{a,b}\in H^{a+b}(B\mathfrak{o})
\]
denote the Thom class of $\mathfrak{o}$ and the Euler classes of the universal vector bundle over $B[\SO_a\times \SO_b]$ respectively.   

\begin{lemma}\label{l:12.5} Under the above notation, for $a\ne b$,
\begin{itemize} 
\item $\tilde{H}^{*}(\Th \xi_{a,b})= U_{a,b,2r}\s (\P(a,b)\oplus e_{a,b}\P(a,b))$;
\end{itemize}
while for $a=b$, 
\begin{itemize}
\item if $r$ is even, then $\tilde{H}^{*}(\Th \xi_{a,b})=U_{a,b,2r}\s(\SP(a,b)\oplus e_{a,b}\SP(a,b))$,
\item if $r$ is odd, then $\tilde{H}^*(\Th \xi_{a,b})=U_{a,b,2r}\s (\AP(a,b)\oplus e_{a,b}\AP(a,b))$,
%\item ($a=b$ and $r$ is odd) \quad $H^{i+2r}(\Th \xi_{a,b})=U_{a,b,2r}\s %e_{a,b}\SP(a,b)$.
\end{itemize}
where $e_{a,b}$ is trivial whenever either $a$ or $b$ is odd.
\end{lemma}
\begin{proof} In the case $a\ne b$ the computations are precisely the same as in the previous lemma. Suppose that $a=b$ is even. Then the group $G$ is generated by a subgroup $[\O_a\times \O_b]^+$ and an element $h\times \tau$. Consequently, the cohomology group of $\Th \xi_{a,b}$ is isomorphic to the subgroup of the cohomology group of $\Th \mathfrak{o}=\Th \xi_a\wedge \Th \xi_b\wedge \Th \varepsilon^{2r}$ of elements that are invariant under the involution $\beta$ on the first two factors at the same time and the action $\alpha$ induced by $h\times \tau$. The Thom classes of the three factors are denoted by $U_a, U_b, U_{2r}$, while the Euler classes of the first two factors are denoted by $e_a, e_b$ respectively. The actions of $\alpha$ and $\beta$ depend on the parity of $r$ (see table Table~\ref{t:ab2}). 

If $a=b$ is odd, then the action $\alpha$ is the composition of $h\times \tau$ and the involution on the first factor of $\Th \mathfrak{o}$. Again, the actions of $\alpha$ and $\beta$ depend on the parity of $r$ (see Table~\ref{t:ab2}), and, for example, for $r$ odd we have 
\[
  \alpha^*U_{a,b,2r}=\alpha^*(U_a\otimes U_b\otimes U_{2r})= -U_b\otimes U_a\otimes -U_{2r}=-U_a\otimes U_b\otimes U_{2r}=-U_{a,b,2r}.
\]
The statement of Lemma~\ref{l:12.5} follows immediately from the above computation.
\end{proof}

\begin{table}
\caption{}\label{t:ab2}
\begin{tabular}{|c||c|c||c|c||c|c||c|c|}
\hline 
& \multicolumn{2}{|c||}{$\A_{2r}$}      & \multicolumn{2}{|c||}{$\A_{2r}$ } & \multicolumn{2}{|c||}{$\A_{2r}$ }  & \multicolumn{2}{|c|}{$\A_{2r}$ } \\
    & \multicolumn{2}{|c||}{$a,r$ even }     & \multicolumn{2}{|c||}{$a$ even, $r$ odd} & \multicolumn{2}{|c||}{$a$ odd, $r$ even} & \multicolumn{2}{|c|}{$a$, $r$ odd}      \\
    
\hline
       &  $\alpha$ & $\beta$ & $\alpha$ & $\beta$ & $\alpha$ & $\beta$ & $\alpha$ & $\beta$  \\
\hline
\hline
$U_a$       & $U_b$       & $-U_a$      & $U_b$        & $-U_a$              & $U_b$      &  $-U_a$          & $U_b$      &  $-U_a$ \\
 \hline 
$U_b$       & $U_a$        & $-U_b$      & $U_a$        & $-U_b$             & $-U_a$     &  $-U_b$           & $-U_a$     &  $-U_b$\\
 \hline
$e_a$       & $e_b$       & $-e_a$      &  $e_b$       & $-e_a$              & 0     & 0            & 0     &  0 \\
 \hline
$e_b$       & $e_a$        & $-e_b$      & $e_a$        & $-e_b$             & 0    & 0            & 0    & 0 \\
 \hline
$U_{a,b,2r}$ & $U_{a,b,2r}$ & $U_{a,b,2r}$ & $-U_{a,b,2r}$ & $U_{a,b,2r}$   & $U_{a,b,2r}$ & $U_{a,b,2r}$  & $-U_{a,b,2r}$ & $U_{a,b,2r}$ \\
 \hline
$e_{a,b}$   & $e_{a,b}$   & $e_{a,b}$   & $e_{a,b}$   & $e_{a,b}$            & 0 & 0  & 0 & 0 \\
 \hline
$p_i$       & $p_i'$       & $p_i$       & $p_i'$       & $p_i$              & $p_i'$       & $p_i$        & $p_i'$       & $p_i$  \\
 \hline
$p_i'$      & $p_i$        & $p_i'$      & $p_i$        & $p_i'$             & $p_i$   &  $p_i'$        & $p_i$   &  $p_i'$   \\
 \hline
 \end{tabular}
 \end{table}

\subsection{Morin map germ $\A^{\pm}_{2r+1}(a,b)$, $r>0$, $a+b=d$.}\label{s:16.4}

Let $G$ denote the relative symmetry group of a minimal $\A_{2r+1}^{\pm}(a,b)$-germ.  Let $\xi^{\pm}_{a,b}$ denote the universal vector $G$-bundle over the classifying space $BG$. Once again, the cohomology group of $\Th \xi^{\pm}_{a,b}$ can be identified with a subgroup of the cohomology group of $\Th \mathfrak{o}$, where $\mathfrak{o}$ stands for the Whitney sum of the universal vector $\SO_a\times \SO_b$-bundle and the trivial $(2r+1)$-bundle over $B[\SO_a\times \SO_b]$. Let 
\[
U^{\pm}_{a,b,2r+1}\in H^{a+b+2r+1}(\Th \mathfrak{o})\quad \mathrm{and} \quad e_{a,b}\in H^{a+b}(B\mathfrak{o})
\]
denote the Thom class of $\mathfrak{o}$ and the Euler class of the universal vector bundle over $B[\SO_a\times \SO_b]$ respectively.
 
\begin{lemma}\label{l:12.6} Under the notation as above,
\begin{itemize} 
\item if $r$ is odd, then $\tilde{H}^{*}(\Th \xi^{\pm}_{a,b})= U^{\pm}_{a,b,2r+1}\s e_{a,b}\P(a,b)$,
\item if $r$ is even, then $\tilde{H}^{*}(\Th \xi^{\pm}_{a,b})= U^{\pm}_{a,b,2r+1}\s \P(a,b)$,
\end{itemize}
where $e_{a,b}$ is trivial whenever either $a$ or $b$ is odd.
\end{lemma}
\begin{proof} The group $G$ is generated by $[\O_a\times \O_b]^+$ and an element that is an orientation reversing involution on the first and the third factors of $\Th \xi_a\wedge \Th \xi_b\wedge \Th \varepsilon^{2r+1}$ if $r+1$ is odd and an orientation reversing involution on the first factor only if $r+1$ is even. Let $U_r$ denote the Thom class of the third factor. Then for transformations $\alpha$ and $\beta$ defined as before, we have
\[
   \alpha^*U_a=-U_a, \quad \alpha^*U_b=U_b, \quad \alpha^*e_a=-e_a, \quad \alpha^*e_b=e_b, 
\quad   \alpha^*p_i=p_i, \quad \alpha^*p_i'=p_i',
\]
\[
   \beta^*U_{a}=-U_{a}, \quad \beta^*U_b=-U_b, \quad \beta^*e_a=-e_a, \quad \beta^*e_b=-e_b,
\quad   \beta^*p_i=p_i, \quad \beta^*p_i'=p_i',
\]  
\[
   \alpha^*U_r=(-1)^{r+1}U_r, \quad \beta^*(U_r)=U_r,
\]
\[
\alpha^*(U_{a,b,r})=(-1)^{r}U_{a,b,r},
\quad \beta^*(U_{a,b,r})=U_{a,b,r},
\quad \alpha^*(e_{a,b})=-e_{a,b},
\quad \beta^*(e_{a,b})=e_{a,b}.
\]
The statement of Lemma~\ref{l:12.6} is easily verified now. 
\end{proof}

Non-zero classes in the first term of the Kazarian spectral sequence are listed in Tables~\ref{t:11}--\ref{t:14}. Here we use the fact that, for example, for $r\ge 0$ the singularity types $\A^+_{4r+3}(2s,2s)$ and $\A^-_{4r+3}(2s,2s)$ coincide.

\begin{table}
\caption{The case $d=4s$.}\label{t:11}
\begin{tabular}{|c|c|c|c|c|c|c|}
\hline 
       &   $A_1$& $A_2$   & $A_3$ & $A_4$ & $A_5$   &$A_6$  \\
$a+b=$ &  $4s+1$& $4s$ & $4s$ & $4s$ & $4s$ & $4s$  \\
\hline
\hline
$a=2s$ & $U\s \P$ & $U\s \AP$ & $U^+\s e\P$ & $U\s S\P$ & $U^+\s \P$  & $U\s \AP$  \\
      &            & $U\s e\AP$ &               & $U\s e\SP$ &           & $U\s e\AP$   \\
 \hline
$a=2s-1$ & $U\s \P$ & $U\s \P$ & 0 & $U\s \P$ & $U^+\s \P$ & $U\s \P$  \\
      &            &            &  &             & $U^-\s \P$ &    \\
 \hline
$a=2s-2$ & $U\s \P$ & $U\s \P$ & $U^+\s e\P$ & $U\s \P$ & $U^+\s \P$  & $U\s \P$  \\
      &            & $U\s e\P$ & $U^-\s e\P$ & $U\s e\P$ & $U^-\s \P$ & $U\s e\P$   \\
 \hline
$a=2s-3$ & $U\s \P$ & $U\s \P$ & 0 & $U\s \P$ & $U^+\s \P$ & $U\s \P$  \\
      &            &            &  &             & $U^-\s \P$ &    \\
 \hline
\dots &\dots &\dots &\dots &\dots &\dots &\dots  \\
\hline
$a=3$ & $U\s \P$ & $U\s \P$ & 0 & $U\s \P$ & $U^+\s \P$ & $U\s \P$  \\
      &            &            &  &             & $U^-\s \P$ &    \\
 \hline
$a=2$ & $U\s \P$ & $U\s \P$ & $U^+\s e\P$ & $U\s \P$ & $U^+\s \P$  & $U\s \P$  \\
      &            & $U\s e\P$ & $U^-\s e\P$ & $U\s e\P$ & $U^-\s \P$ & $U\s e\P$   \\
 \hline
$a=1$ & $U\s \P$ & $U\s \P$ & 0 & $U\s \P$ & $U^+\s \P$ & $U\s \P$  \\
      &            &            &  &             & $U^-\s \P$ &    \\
 \hline
$a=0$ & $U\s \P$ & $U\s \P$ & $U^+\s e\P$ & $U\s \P$ & $U^+\s \P$  & $U\s \P$  \\
      &            & $U\s e\P$ & $U^-\s e\P$ & $U\s e\P$ & $U^-\s \P$ & $U\s e\P$   \\
 \hline
 \end{tabular}
 \end{table}

%\begin{landscape}
\begin{table}
\caption{The case $d=4s+1$.}
%\begin{tabular}{|c|c|c|c|c|c|c|c|}
\begin{tabular}{|c|c|c|c|c|c|c|}
\hline 
       &   $A_1$& $A_2$   & $A_3$ & $A_4$ & $A_5$   &$A_6$    \\
$a+b=$ &  $4s+2$& $4s+1$ & $4s+1$ & $4s+1$ & $4s+1$ & $4s+1$  \\
\hline
\hline
$a=2s+1$ & $U\s \AP$ &     &  &   &   &    \\
         &       &  &  &   &  &  \\
\hline
$a=2s$ & $U\s \P$ & $U\s \P$ & 0 & $U\s \P$ & $U^+\s \P$  & $U\s \P$ \\
& $U\s e\P$ & & & & $U^-\s \P$ &  \\
 \hline
$a=2s-1$ & $U\s \P$ & $U\s \P$ & 0 & $U\s \P$ & $U^+\s \P$  & $U\s \P$  \\
 & & & & & $U^-\s \P$ &  \\
 \hline
$a=2s-2$ & $U\s \P$ & $U\s \P$ & 0 & $U\s \P$ & $U^+\s \P$  & $U\s \P$  \\
& $U\s e\P$ & & & & $U^-\s \P$ &  \\
 \hline
\dots &\dots &\dots &\dots &\dots &\dots &\dots  \\
\hline
$a=3$ & $U\s \P$ & $U\s \P$ & 0 & $U\s \P$ & $U^+\s \P$  & $U\s \P$ \\
 & & & & & $U^-\s \P$ &  \\
 \hline
$a=2$ & $U\s \P$ & $U\s \P$ & 0 & $U\s \P$ & $U^+\s \P$  & $U\s \P$  \\
& $U\s e\P$ & & & & $U^-\s \P$ &  \\
 \hline
$a=1$ & $U\s \P$ & $U\s \P$ & 0 & $U\s \P$ & $U^+\s \P$  & $U\s \P$  \\
 & & & & & $U^-\s \P$ &  \\
 \hline
$a=0$ & $U\s \P$ & $U\s \P$ & 0 & $U\s \P$ & $U^+\s \P$  & $U\s \P$  \\
 & $U\s e\P$  & & & & $U^-\s \P$ &  \\
 \hline
 \end{tabular}
 \end{table}
%\end{landscape}

\begin{table}
\caption{The case $d=4s+2$.}
\begin{tabular}{|c|c|c|c|c|c|c|}
\hline 
       &   $A_1$& $A_2$   & $A_3$ & $A_4$ & $A_5$   &$A_6$  \\
$a+b=$ &  $4s+3$& $4s+2$ & $4s+2$ & $4s+2$ & $4s+2$ & $4s+2$  \\
\hline
\hline
$a=2s+1$ & $U\s \P$ & $U\s \AP$ & 0 & $U\s \SP$ & $U^+\s \P$ & $U\s \AP$  \\
         &          &           &   &           &             &            \\  
 \hline
$a=2s$ & $U\s \P$ & $U\s \P$  & $U^+\s e\P$ & $U\s \P$  & $U^+\s \P$  & $U\s \P$  \\
      &           & $U\s e\P$ & $U^-\s e\P$ & $U\s e\P$ & $U^-\s \P$ & $U\s e\P$   \\
 \hline
$a=2s-1$ & $U\s \P$ & $U\s \P$ & 0 & $U\s \P$ & $U^+\s \P$ & $U\s \P$  \\
      &            &               &  &               & $U^-\s \P$ &    \\  
 \hline
$a=2s-2$ & $U\s \P$ & $U\s \P$ & $U^+\s e\P$ & $U\s \P$ & $U^+\s \P$  & $U\s \P$  \\
      &            & $U\s e\P$ & $U^-\s e\P$ & $U\s e\P$ & $U^-\s \P$ & $U\s e\P$   \\
\hline
\dots &\dots &\dots &\dots &\dots &\dots &\dots  \\
\hline
$a=3$ & $U\s \P$ & $U\s \P$ & 0 & $U\s \P$ & $U^+\s \P$ & $U\s \P$  \\
      &            &               &  &               & $U^-\s \P$ &    \\  
 \hline
$a=2$ & $U\s \P$ & $U\s \P$ & $U^+\s e\P$ & $U\s \P$ & $U^+\s \P$  & $U\s \P$  \\
      &            & $U\s e\P$ & $U^-\s e\P$ & $U\s e\P$ & $U^-\s \P$ & $U\s e\P$   \\
 \hline
$a=1$ & $U\s \P$ & $U\s \P$ & 0 & $U\s \P$ & $U^+\s \P$ & $U\s \P$  \\
      &            &               &  &               & $U^-\s \P$ &    \\  
 \hline
$a=0$ & $U\s \P$ & $U\s \P$ & $U^+\s e\P$ & $U\s \P$ & $U^+\s \P$  & $U\s \P$  \\
      &            & $U\s e\P$ & $U^-\s e\P$ & $U\s e\P$ & $U^-\s \P$ & $U\s e\P$   \\
 \hline
 \end{tabular}
 \end{table}

\begin{table}
\caption{The case $d=4s+3$.}\label{t:14}
\begin{tabular}{|c|c|c|c|c|c|c|}
\hline 
       &   $A_1$& $A_2$   & $A_3$ & $A_4$ & $A_5$   &$A_6$    \\
$a+b=$ &  $4s+4$& $4s+3$ & $4s+3$ & $4s+3$ & $4s+3$ & $4s+3$  \\
\hline
\hline
$a=2s+2$ & $U\s \AP$ &     &  &  &   &  \\
         & $U\s e\SP$           &               &               & &    \\  
 \hline
$a=2s+1$ & $U\s \P$ & $U\s \P$ & 0 & $U\s \P$ & $U^+\s \P$  & $U\s \P$ \\
 & & & & & $U^-\s \P$ &  \\
 \hline
$a=2s$ & $U\s \P$ & $U\s \P$ & 0 & $U\s \P$ & $U^+\s \P$  & $U\s \P$ \\
& $U\s e\P$ & & & & $U^-\s \P$ &  \\
 \hline
$a=2s-1$ & $U\s \P$ & $U\s \P$ & 0 & $U\s \P$ & $U^+\s \P$  & $U\s \P$  \\
 & & & & & $U^-\s \P$ &  \\
 \hline
\dots &\dots &\dots &\dots &\dots &\dots & \dots \\
\hline
$a=3$ & $U\s \P$ & $U\s \P$ & 0 & $U\s \P$ & $U^+\s \P$  & $U\s \P$  \\
 & & & & & $U^-\s \P$ &  \\
 \hline
$a=2$ & $U\s \P$ & $U\s \P$ & 0 & $U\s \P$ & $U^+\s \P$  & $U\s \P$  \\
& $U\s e\P$ & & & & $U^-\s \P$ &  \\
 \hline
$a=1$ & $U\s \P$ & $U\s \P$ & 0 & $U\s \P$ & $U^+\s \P$  & $U\s \P$  \\
 & & & & & $U^-\s \P$ &  \\
 \hline
$a=0$ & $U\s \P$ & $U\s \P$ & 0 & $U\s \P$ & $U^+\s \P$  & $U\s \P$  \\
 & $U\s e\P$  & & & & $U^-\s \P$ &  \\
 \hline 
 \end{tabular}
 \end{table}

\addcontentsline{toc}{part}{The differentials in the Kazarian spectral sequence}

\section{The differential $d^{0,*}_{*}$}\label{s:diff}

Our next step is to compute the differentials in the spectral sequence. 

%\subsection{Regular germ $A_0$} The differential is trivial. 

To compute the differential $d_1^{*,*}$ we need to know the geometry of adjacencies of singular strata. In those cases where the normal bundles of singular strata possess a natural orientation, the computation of the coboundary homomorphism associated with $d_1^{*,*}$ is not hard. However, for  example, in the case of $\A_2(2s,2s)$ and $\A_3(2s,2s)$ singularities, the corresponding normal bundles are non-orientable. From the geometry of the adjacency of the set of $\A_2(2s,2s)$-singular points to the set of $\A_3(2s,2s)$-singular points, it immediately follows that the corresponding coboundary homomorphism is either trivial or given by a ``multiplication by $\pm 2$"; but to determine weather the coboundary homomorphism is trivial or not, one needs to use a more careful argument. In fact, in those cases where the coboudary homomorphism is non-trivial, we will need to determine the sign of the ``multiplication by $\pm 2$".  An obvious approach to computing the coboundary homomorphism in these cases is to introduce and work with cohomology groups with coefficients in certain sheafs. We follow a similar approach by introducing different types of coverings of singular strata.

Let $\P(d)$ denote the polynomial ring $\Q[p_1, \dots , p_{\left\lfloor d/2\right\rfloor}]$. For $d$ even, let $\mathcal{T}(a,b)$ denote the image of the algebra homomorphism
\[
    \P(d)\longrightarrow \P(a, b),
\] 
\[  
      p_i \mapsto \sum_{j=0}^{i} p_jp'_{i-j}, 
\]
where $d=a+b$. Here we adopt the convention that in $\P(a,b)$ there are relations $p_{j}=0$ for $j>a$ and $p'_{i-j}=0$ for $i-j>b$. For example, $\mathcal{T}(4,8)$ is spanned by 
\[
    \{\ p_1+p_1', p_2+p_1p_1'+p_2', p_2p_1'+p_1p_2'+p_3', p_2p_2'+p_1p_3'+p_4', p_1p_4', p_2p_4' \ \}.
\]
The algebra $\mathcal{T}(a,b)$ is a homomorphic image of $\SP(d/2, d/2)$.

The first column in the Kazarian spectral sequence corresponds to 
\[
   H^*(\mathbf{A_0})\approx H^*(A_0)\approx H^*({\BSO}_d),
\]
where $A_0=S(0)$ (see Lemma~\ref{l:12.1}).

\begin{theorem}\label{th:9.7} For $d$ odd, the differential $d^{0,*}_*$ is trivial and 
\[
    E^{0,*}_{\infty}=\cdots =E_1^{0,*}=H^*({\BSO}_d)=\Q[p_1, ..., p_{\frac{d-1}{2}}].
\]
For $d$ even,
\[
    E^{0,*}_{\infty}=\cdots =E_2^{0,*}=\mathrm{ker}\ d^{0,*}_{1}=\Q[p_1, ..., p_{d/2}]\subset H^*({\BSO}_d),
\]
\[
   \mathrm{im}\, d^{0,*}_1=\{\, \sum_{a,b} (-1)^a U_{a,b,1}\s Q  \, \},
\]
where $Q\in \SP(d/2,d/2)$, and $a$ ranges over $0,..., d/2$ and $b=d+1-a$. 
\end{theorem} 

\begin{remark} The signs in the formula for the image of $d^{0,*}_1$ in the case where $d$ is even (and hence $a<b$) depends on our choice of the Thom classes $U_{a,b,1}$, or equivalently, in the choice of coorientation of the normal bundle $\nu_1$ of the stratum of fold map germs of index $a$ in the set $A_1$ for each $a$. The normal bundle $\nu_1$ over fold points of index $a$ is isomorphic to the canonical kernel bundle, which in its turn has a canonical orientation (see Remark~\ref{r:orientation}). We choose $U_{a,b,1}$ so that it is compatible with the canonical orientation of the kernel bundle.  
\end{remark}

\begin{remark}\label{r:change} Consider an oriented Morin map $f\colon M\to N$ of even dimension of closed manifolds with only fold and cusp singularities. Let $S$ be a component of singular points of $f$ that contains fold points of different indices $a_1$ and $a_2=a_1+1$. Then $S$ is a closed smooth submanifold in $M$ and its normal bundle in $N$ is orientable. We observe that if a coorientation of $S$ is compatible with the canonical orientation of the kernel bundle of $f$ over the set of fold points of index $a_1$, then it is not compatible with the the canonical orientation of the kernel bundle of $f$ over the set of fold points of index $a_2$.    
\end{remark}

\begin{proof} The differential $d^{0,*}_1$ coincides with the upper coboundary homomorphism in the commutative diagram of exact sequences
\[
\begin{CD}  
H^*(\mathbf{A_1})@>>> H^*(\mathbf{A_0}) @>d^{0,*}_1 >> H^{*+1}(\mathbf{A_1}, \mathbf{A_0}) \\
@AAA @AAA @AAA\\
H^{*}(A_1)@>>> H^*(A_0) @>\delta >> H^{*+1}(A_1, A_0), \\
\end{CD}   
\]
where $A_r$ stands for the space $S(r)$ with $r\ge 0$ and the vertical maps are Thom isomorphisms (see the proof of Lemma~\ref{l:12.1}). 

Suppose that the dimension $d$ is odd. Then every cohomology class of $A_0$ extends to a cohomology class of $A_1$. Consequently, the differential $d^{0,*}_1$ is trivial. On the other hand, if $d^{0,*}_{i}$ is trivial for all $i<r$, then the homomorphism $d^{0,*}_r$ coincides with the upper horizontal homomorphism in the commutative diagram 
\[
\begin{CD}  
H^*(\mathbf{A_0}) @>d^{0,*}_r >> H^{*+1}(\mathbf{A_r}, \mathbf{A_{r-1}}) \\
@AAA @AAA \\
H^*(A_0) @>\delta >> H^{*+1}(A_{r}, A_{r-1}), \\
\end{CD}   
\]
of Thom isomoprhisms, where the lower horizontal homomorphism takes a cohomology class $x$ to the coboundary class $\delta(y)$ in $H^{*+1}(A_r, A_{r-1})$, where $y$ is an extension of $x$ over $A_{r-1}$. Arguing as above, we conclude that $\delta$ is trivial. Thus, by induction, all homomorphisms $d^{0,*}_*$ are trivial. 

Suppose now that $d$ is even. To begin with we observe that the spectrum $\mathbf{A_1/A_{0}}$ is the wedge union of $1+d/2$ spectra
\[
\mathbf{A_1/A_0} =\bigvee \mathbf{A_1(a,b)/A_0}
\]
where $\mathbf{A_1}(a,b)$ is the spectrum for $\A_1(a,b)$ and $\A_0$ map germs, and therefore the differential $d^{0,*}_1$ has $1+d/2$ components, one component for each index $a=0,...,d/2$. The $a$-th component of $d^{0,*}_1$ is the upper horizontal homomorphism in the commutative diagram 
\[
\begin{CD}  
H^*(\mathbf{A_0}) @>>> H^{*+1}(\mathbf{A_1(a,b)}, \mathbf{A_0}) \\
@AAA @AAA \\
H^*(A_0) @>\delta >> H^{*+1}(A_{1}(a,b), A_0) \\
\end{CD}   
\]
of Thom isomorphisms. In the case $a=0$, the lower horizontal homomorphism $\delta$ coincides with the coboundary homomorphism (which we also denote by $\delta$) in the exact sequence of the pair $(\BSO_{d+1}, \BSO_d)\colon$
\[
     0\longrightarrow H^*({\BSO}_{d+1})\longrightarrow H^*({\BSO}_{d}) \longrightarrow H^*({\BSO}_{d+1}, {\BSO}_d)\longrightarrow 0,
\]
\begin{equation}\label{eq:star}
    \delta(p_i)=0, \qquad \delta(e\s p_i) = \mathop{\mathrm{Th}}(p_i).
\end{equation}
where $\mathop{\mathrm{Th}}$ is the Thom isomorphism.  This implies the statement of Theorem~\ref{th:9.7} for the $a$-th component of $d^{0,*}_1$ for $a=0$. 

To compute the $a$-th component of $d^{0,*}_1$ for $a\ne 0$ we need an auxiliary construction. Let $f\colon \R^{a+b}\to \R$ denote the standard Morse function. There is a standard action of the group $G=\SO_a\times\SO_b$ on the source space of $f$. If we let $G$ act trivially on the target space of $f$, then $f$ is $G$-equivariant. Let
\[
    \gamma\colon E_\gamma=EG\times_{G} \R^{a+b}\longrightarrow {\BSO}_a\times {\BSO}_b= B_\gamma
\]
be the universal vector $G$-bundle.  Consider the commutative diagram of vector bundles
%\[
%    \begin{CD}
%    S_{\gamma} @>I>> S_{\gamma} @>K>> \BSO_d \\
%    @VVV @VVV @VVV \\
%    \BSO_a\times \BSO_b @>\id >> \BSO_a\times \BSO_b @>>> \BSO_{d+1},  
%    \end{CD}
%\]
\[
    \begin{CD}
    E_{\gamma} @>I>> {\mathop{\mathrm{ESO}}}_{d+1} \\
     @V\gamma VV @V\xi_{d+1} VV \\
    \BSO_a\times \BSO_b @>>> \BSO_{d+1},  
    \end{CD}
\]
where $\xi_{d+1}$ is the universal oriented vector bundle of dimension $d+1$, the bottom horizontal map is induced by the injective group homomorphism  of $\SO_a\times \SO_b$ into $\SO_{a+b}$, and the map $I$ is the fiberwise involution
\begin{equation}\label{in}
     EG\times_G\R^{a+b}=EG\times_G[\R^a\times \R^b] \xrightarrow{\id\times_G[(-1)\times 1]} EG\times_G [\R^a\times \R^b]  
\end{equation}
composed with the pullback map. We recall that the total space of the spherical bundle of $\xi_{d+1}$ is homotopy equivalent to $\BSO_d$, and the restriction of $\xi_{d+1}^*{\mathop\mathrm{ESO}}_{d+1}$ to $\BSO_d$ canonically splits into the Whitney sum of the trivial line bundle over $\BSO_d$ and $\xi_d$. Let $S_{\gamma}$ denote the total space of the spherical bundle associated with $\gamma$. Let $K$ denote the kernel of $df$. Then there is a canonical orientation preserving isomorphism 
\begin{equation}\label{eq:split}
 \gamma^*E_\gamma=   \gamma^*[EG\times_G \mathop\mathrm{grad}f]\oplus [EG\times_G K]\approx I^*\varepsilon\oplus I^*\xi_d=I^*(\xi_{d+1}^*{\mathop\mathrm{ESO}}_{d+1})
\end{equation}
of oriented vector bundles over $S_{\gamma}$, where $\varepsilon$ is the trivial vector bundle over ${\mathop\mathrm{ESO}}_d$.

\begin{figure}[ht]
	\centering
			\includegraphics[draft=false, width=60mm]{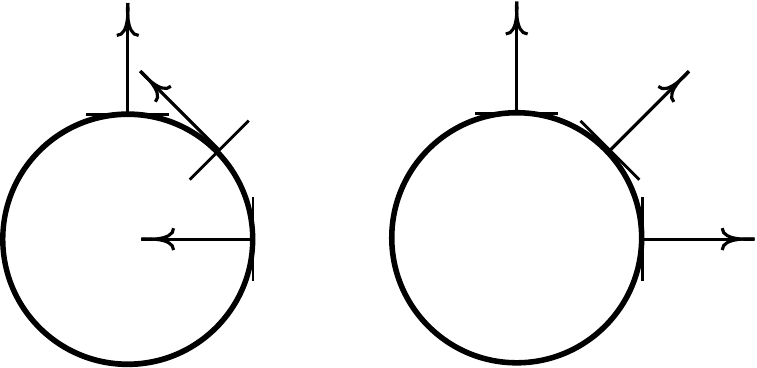}
\caption{The gradient and the kernel bundle for $f=-x_a^2+x_b^2$ (on the right hand side) and for $f=x_a^2+x_b^2$ (on the left hand side).}
\label{fig:5}
\end{figure}

\begin{lemma} The canonical isomorphism (\ref{eq:split}) splits into the direct sum of two orientation preserving isomorphisms of vector bundles over $S_\gamma$
\[
   \gamma^*[EG\times_G \mathop\mathrm{grad}f]\approx I^*\varepsilon \qquad\mathrm{and}\qquad \gamma^*[EG\times_G K]\approx I^*\xi_{d}.
\] 
\end{lemma}
\begin{proof} In the standard coordinates, the gradient of the standard Morse function $f$ is given by 
\[
  \mathop{\mathrm{grad}}f=(-2x_1, \dots, -2x_a, 2x_{a+1}, \dots, 2x_{a+b}).
\]
So the vector field given by the involution (\ref{in}) of $\mathop{\mathrm{grad}}f$ at a point $\vec{x}\in \R^d$ is given by $2\vec{x}$. It remains to observe that the kernel bundle of $f$ over non critical points is given by the oriented orthogonal complement of the span of $\mathop{\mathrm{grad}}f$. 
\end{proof}

In what follows we will identify the base space of a vector bundle with the zero section of that bundle. 

\begin{corollary} The coboundary homomorphism $\delta\colon H^*(E_{\gamma}\setminus B_{\gamma})\to H^*(E_\gamma, E_{\gamma}\setminus B_{\gamma})$ takes the Euler class of the vector bundle $K$ to $(-1)^aU$, where $U$ is the Thom class of the vector bundle $\gamma\colon E_{\gamma}\to B_{\gamma}$.  
\end{corollary}
\begin{proof} Let $\tilde{\xi}_d$ be the fiber bundle over $S_{\gamma}$ that consists of oriented planes tangent to the fibers of the projection $S_{\gamma}\to B_\gamma$. Then the coboundary homomorphism takes the Euler class of $\tilde\xi_d$ to the Thom class of $\gamma$. On the other hand the involution (\ref{in}) is covered by an automorphism of $\tilde\xi_d$ that changes the orientation if and only if $(-1)^a$ is negative. Thus for vector bundles over $S_\gamma$ we have 
\[
      e(\gamma^*(EG\otimes_G K))=I^*e(\tilde\xi_d)=(-1)^a(\tilde\xi_d). 
\]
This implies the statement of the corollary. 
\end{proof}

Let $E_{\xi_{a,b}}$ and $B_{\xi_{a,b}}$ respectively denote the total space and the base space of the universal vector bundle $\xi_{a,b}$ over the classifying space of the relative symmetry group of $\A_1(a,b)$. Then there is a commutative diagram 
\[
\begin{CD}
     H^*({\mathop\mathrm{ESO}}_{d+1}\setminus {\BSO}_{d+1}) @>\xi_{d+1}>> H^*({\mathop\mathrm{ESO}}_{d+1}, {\mathop\mathrm{ESO}}_{d+1}\setminus {\BSO}_{d+1}) \\
     @VVV @VVV \\
     H^*(E_{\xi_{a,b}}\setminus B_{\xi_{a,b}}) @>>> H^*(A_1(a,b), A_0) \\
     @VVV @Vk VV \\
     H^*(E_{\gamma}\setminus B_\gamma) @>>> H^*(E_\gamma, E_\gamma\setminus B_\gamma),
\end{CD}
\] 
where the bottom square is induced by a pullback square. Since the homomorphism $k$ in the above diagram is injective, chasing this diagram we compute the differential $d^{0,*}_1$. Since every cohomology class in $\ker d^{0,*}_1$ extends to a cohomology class in $\mathbf{A_r}$ for $r>1$ we conclude that all differentials $d^{0,*}_r$ are trivial for $r>1$. 
\end{proof}

\begin{remark}\label{r:17.3} In \cite{Sa} we have shown that the spectrum $\mathbf{A_1}$ splits as
\[
     \mathbf{A_1}=[\mathbf{A_1(0,d+1)}]\vee [\mathbf{A_1(1,d)/A_0}] \vee \cdots \vee [\mathbf{A_1(d/2,d/2+1)/A_0}]
\] 
for $d$ even, and 
\[
     \mathbf{A_1}=\mathbf{Y}\vee [\mathbf{A_1(1,d)/A_0}] \vee \cdots \vee [\mathbf{A_1(d-1/2,d-1/2)/A_0}]
\] 
for $d$ odd, where $\mathbf{Y}$ is the spectrum for $\A_0$, $\A_1(0,d+1)$ and $\A_1(\frac{d+1}{2}, \frac{d+1}{2})$ singularity types. This suggests a different choice of convenient generators for the cohomology groups of $\mathbf{A_1}$. We will see, however, that for our purposes the choices made in the current paper are more appropriate. 
\end{remark}

In the next section we will discuss in detail computations of the differential $d_1^{1,*}$ and the component of the differential $d_1^{2,*}$ in the case $d=4s$. Computations of the other differentials are similar and mentioned in remarks.

\section{The differential $d_1^{1,*}$}

In what follows we will often use the property that for a topological space $X$ the inclusion $\{1\}\times X\subset [0,1]\times X$ gives rise to the coboundary isomorphism 
\[
    H^*(\{1\}\times X)\stackrel{\delta}\longrightarrow \tilde{H}^{*+1}([0,1]\times X, (\{0\}\sqcup \{1\})\times X).
\]
Furthermore, the coboundary homomorphism takes each class $x$ to $\Th x$, where $\Th$ is the Thom isomorphism. Similarly, the coboundary homomorphism 
\[
    H^*(\{0\}\times X)\stackrel{\delta}\longrightarrow \tilde{H}^{*+1}([0,1]\times X, (\{0\}\sqcup \{1\})\times X),
\]
takes each class $x$ to $-\Th x$. These properties follow from the exact sequence of the pair $([0,1]\times X, (\{0\}\sqcup \{1\})\times X)$. 

\subsection{The case $a=0$.}
The stratum of $\A_1(0,d+1)$ points in $A_{\infty}$ is adjacent only to the stratum of $\A_2(0, d+1)$ points. Consequently in this case the differential $d_1^{1,*}$ has only one component. Let $A_1'$ denote the subset in $A_{\infty}=S(\infty)$ of $\A_1(0,d+1)$ and $\A_0$ points. Let $A_2'\subset A_{\infty}$ be the subset that contains $A_1$ and the set of $\A_2(0,d)$ points. Then there are homotopy equivalences
\[
    A_1'\setminus A_0 \simeq {\BSO}_{d+1}  \qquad \mathrm{and} \qquad  A_2'\setminus A_1\simeq {\BSO}_d,
\]
and a commutative diagram of Thom isomorphisms
\[
\begin{CD}
   \tilde{H}^{*+d+1}(A_1'/A_0) @>d^{1,*}_1>> \tilde{H}^{*+d+2}(A_2'/A_1) \\
   @AAA @AAA \\
   H^*(A_1'\setminus A_0) @>>> H^{*+1}(A_2'\setminus A_0, A_1\setminus A_0),
\end{CD}
\]
where the bottom horizontal homomorphism $\delta$ takes a cohomology class $x$ to the cohomology class $\delta(y)$ for a trivial extension of $x$ over $A_1\setminus A_0$. The vertical homomorphisms in the diagram depend on the choices of Thom classes, which in their turn depend on our choice of orientation of the normal bundle of $A_2'\setminus A_0$ in $A_2'$. In fact, by excision we only need to specify an orientation of the normal bundle in the complement to the set of $\A_{1}(i,d+1-i)$ points for $i>0$; in the complement the normal bundle is isomorphic to the canonical kernel bundle which carries a canonical orientation since the canonical cokernel bundle is oriented. 

\begin{figure}[ht]
	\centering
			\includegraphics[draft=false, width=60mm]{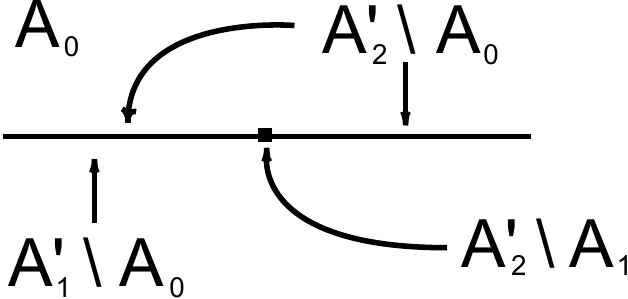}
%\caption{The four fold covering $T_1'\to A_1\setminus A_0$. One of the two components of $T_1'$ adjacent to %$T_2'\setminus T_1$ (bold).}
\label{fig:6}
\end{figure}

Next we choose a generator in the group $H^1(A_2'\setminus A_0, A_1\setminus A_0)$, i.e., the Thom class, so that it is compatible with the coorientation of the $A_2'\setminus A_1$ stratum in the space $A_2'\setminus A_0$ in the direction of $A_1'$. Finally, we choose the Thom class $U_{0,d,2}$, or, equivalently, the coorientation of $A_2'\setminus A_1$ in $A'_2$ to be the coorientation of $A_2'\setminus A_1$ in $A_2'\setminus A_0$ followed by the coorientation of $A_2'\setminus A_0$ in $A_2'$. Then 
\begin{equation}\label{equ:14.1}
d^{1,*}_1(U_{0,d+1,1}\s p)= U_{0,d,2}\s p 
\end{equation}
for each class $p\in \mathcal{P}(0,d+1)$. 

\begin{remark} By a similar argument, we deduce that for $a=0$, 
\[
d^{1,*}_1(U_{0,d+1,1}\s p + U_{0,d+1,1}\s eq)= U_{0,d,2}\s p 
\]
for each integer $d>0$ (not necessarily of the form $d=4s$) and every pair of classes $p, q\in \mathcal{P}(0,d+1)$; the Euler class $e$ is trivial if $d+1$ is odd. In particular, if $d+1$ is even, then the class $U_{0,d+1,1}\s p$ maps to the trivial class for each $p\in p'_{\frac{d+1}{2}}\P(0,d+1)$. 
\end{remark}

\subsection{The case $a=1,..., 2s-1$.}\label{ss:18.2}
We will construct a double covering over the space $A_{\infty}\setminus A_0$ by means of a cohomology class $w$; by definition, this double covering will be the spherical bundle associated with the unique line bundle whose first Stiefel-Whitney class is $w$. Each component of the space $A_1\setminus A_0$ is homotopy equivalent to the classifying space $BG_{\alpha}$ of a relative symmetry group $G_{\alpha}$ of a singularity type $\alpha$. The cohomology class $w$ is first defined on the components in $A_1\setminus A_0$; it is a class $w'$ in $H^1(BG_{\alpha}; \Z_2)$ that assumes value $1$ only on those pointed loops that represent classes in $\pi_1(BG_{\alpha})\approx \pi_0(G_{\alpha})$ that correspond to path components of $G_{\alpha}$ that act on $\R^a\times \R^b$ by involution on each of the two factors. There is a similarly defined class $w''$ on the set $A_{\infty}\setminus A_0$. We claim that there is a unique cohomology class $w$ over $A_{\infty}\setminus A_0$ that restricts to $w'$ and $w''$. Indeed, consider a part of the Mayer-Vietoris exact sequence 
\[
    H^0(U)\oplus H^0(V)\longrightarrow H^0(U\cap V)\longrightarrow H^1(U\cup V) \longrightarrow H^1(U)\oplus H^1(V)\longrightarrow H^1(U\cap V)
\] 
for a space $V=A_{1}\setminus A_0$ and an open neighborhood $U$ of the space $A_{\infty}\setminus A_1$ in the space $A_{\infty}\setminus A_0$. In what follows we will identify the cohomology groups of the spaces $U$ and $A_{\infty}\setminus A_1$. The element $w'\oplus w''$ belongs to the kernel of the forth homomorphism and therefore lifts to an element $w$ in the cohomology group of $U\cup V$. This lift is unique up to the image of the second homomorphism, which is trivial since the second homomorphism in the exact sequence 
\[
    H^0(U\cup V) \longrightarrow H^0(U)\oplus H^0(V) \longrightarrow H^0(U\cap V),
\]
\[
    \Z \longrightarrow (2s+1)\Z\oplus (2s+1)\Z\longrightarrow (4s+1)\Z,
\]
%first homomorphism 
%\[
%   H^0(U)\oplus H^0(V)\approx 2s\Z\oplus 2s\Z \longrightarrow H^0(U\cap V)\approx 2(s-1)\Z\oplus \Z,
%\] 
%\[
%   (a_1, ..., a_{2s})\oplus (b_1, ..., b_{2s})\mapsto (a_0-b_0, a_1-b_0, a_1-b_1, ..., a_{2s-1}-b_{2s-1}, %a_{2s}-b_{2s-1})\oplus (a_{2s}-b_{2s}),
%\]
is surjective. Thus the element $w$ that restricts to $w'$ and $w''$ is indeed unique.  

Let $R_{\infty}$ denote the total space of the double covering $\pi_R$ associated with $w$ over the space $A_{\infty}\setminus A_0$. The space $R_{\infty}$ inherits a filtration by subspaces
\[
   R_r=\pi_R^{-1}(A_r\setminus A_0) \qquad \mathrm{and} \qquad R^{\pm}_r(a,b)=\pi_R^{-1}(A^{\pm}_r(a,b)\setminus A_0),
\] 
where $r=1,2,...$, and the set $A^{\pm}_r(a,b)$ stands for the union of $A_{r-1}$ and the set of $\A^{\pm}_r(a,b)$ points in $A_{\infty}$. Our computation of the differential relies on the commutative diagram
\[
\begin{CD}
    H^*(\Th\pi_R^*\nu_1) @>d^{1,*}_1>> H^*(\Th\pi_R^*\nu_2) \\
    @AAA @AAA \\
    H^*(\Th\nu_1) @>d^{1,*}_1>> H^*(\Th\nu_2),     
\end{CD}
\]
where $\nu_k$ stands for the normal bundle of $A_k\setminus A_{k-1}$ in $A_k$. We observe that the vector bundle $\pi_R^*\nu_1$ consists of components $\mathfrak{o}$ that are described in the proof of Lemma~\ref{l:10}, while the Thom space of $\pi_R^*\nu_2$ consists of the Thom spaces that appear in the computation of the cohomology group of $\Th\nu_2|A_2$. Under the assumption that $a\ne b$, we have computed that the $a$-th components of the two groups at the bottom of the diagram are subgroups of the $a$-th components of the two groups at the top of the diagram:  
\[
     \tilde{H}^*(\Th\nu_1) = U_{a,b,1}\s \mathcal{P}(a,b),
\]
\[
    \tilde{H}^*(\Th\nu_2)= U_{a,b,2}\s (\mathcal{P}(a,b)+e_{a,b}\mathcal{P}(a,b)),
\]
and $e_{a,b}=0$ for $a$ odd.

Since the $A_1'=A_1(a, b)$ is bounded by the sets  
\[
A_2'=A_2(a, b-1)\qquad \mathrm{and} \qquad A_2''=A_2(a-1, b),
\]
the differential $d^{1,*}_1$ has two components. The first component will be computed by means of the commutative diagram of Thom isomorphisms
\[
\begin{CD}
    \tilde{H}^{4s+1+*}(\Th\pi_R^*\nu_1) @>>> \tilde{H}^{4s+2+*}(\Th\pi_R^*\nu'_2) \\
    @AAA @AAA \\
    H^*(R_1') @>>> H^{*+1}(R'_2, R_1), 
\end{CD}
\]
where $\nu'_2$ is the restriction of $\nu_2$ to $A_2'$, $R'_2=\pi_R^{-1}(A_2')$ and $R'_1=\pi_R^{-1}(A_1')$. The vector bundle $\nu_1$ is canonically oriented since it is isomorphic to the canonical kernel bundle. Consequently the vertical homomorphism in the above diagram on the left hand side is canonical. We choose a canonical Thom class $U'(a,b-1)$ for the vertical homomorphism on the right hand side so that it is compatible with the Thom class of $\pi_R^*\nu_1$. Next we choose a generator in the group $H^1(R'_2, R_1)$, which is a Thom class of a line bundle, so that it is compatible with the coorientation of $R'_2\setminus R_1$ in the space $R'_2$ in the direction of $R_1'$ (of fold map germs of the smaller index $a$). Finally we choose the class $U_{a, b-1,2}$ so that it corresponds to the coorientation of $R_2'\setminus R_1$ in $\Th \pi^*_R\nu_2'$ given by the coorientation of $R_2'\setminus R_1$ in $R_2'$ followed by the coorientation of $R_2'$ in $\Th \pi^*_R\nu_2'$. 

Similarly we have a commutative diagram for the second component of $d^{1,*}_1$ 
\[
\begin{CD}
    \tilde{H}^{4s+1+*}(\Th\pi_R^*\nu_1) @>>>\tilde{H}^{4s+2+*}(\Th\pi_R^*\nu''_2) \\
    @AAA @AAA \\
    H^*(R_1') @>>> H^{*+1}(R''_2, R_1), 
\end{CD}
\]
where $\nu''_2$ is the restriction of $\nu_2$ to $A_2''$, $R''_2=\pi_R^{-1}(A_2')$ and $R'_1=\pi_R^{-1}(A_1')$. Again, the Thom homomorphism on the left hand side is chosen to be the orientation class for the canonical kernel bundle, while the Thom class $U''(a-1,b)$ for the vertical homomorphism on the right hand side is chosen so that the diagram is commutative. Then, by Remark~\ref{r:change}, 
\[
U'(a,b)=-U''(a,b) \qquad \mathrm{for} \qquad a=1,..., 2s-1, \quad b=4s+1-a.
\] 
We choose a generator in the group $H^1(R''_2, R_1)$ so that it is compatible with the coorientation in the direction opposite to $R_1'$ (i.e., in the direction of map germs of smaller index $a$). We note that the chosen coorientation class $U_{a-1,b,2}$ coincides with negative the coorientation of $R_2''\setminus R_1$ in $R_2''$ followed by the coorientation of $U''(a-1,b)$ Chasing the two above diagrams we conclude that 
\begin{equation}\label{eq:18.18}
     d^{1,*}_1(U_{a,b,1}\cup p) = U_{a-1,b,2}\cup p + U_{a, b-1, 2}\cup p
\end{equation}
for every cohomology class $p\in \mathcal{P}(a,b)$, where all Thom classes $U_{a,b,1}$ are chosen to be compatible with $U'(a,b)$. Note that the classes $U_{a-1,b,2}\s p$ are trivial if $a$ is even and $p\in p_{\frac{a}{2}}\mathcal{P}(a,b)$; while the classes $U_{a,b-1,2}\s p$ are trivial if $b$ is even and $p\in p'_{\frac{b}{2}}\mathcal{P}(a,b)$.

\begin{remark} By a similar argument, we deduce that for $a=1,...,\left\lfloor \frac{d-1}{2}\right\rfloor$, 
\[
d^{1,*}_1(U_{a,b,1}\s p + U_{a,b,1}\s eq)= U_{a-1,b,2}\cup p + U_{a, b-1, 2}\cup p,
\]
where $b=d+1-a$, for each integer $d>0$ and every pair of classes $p, q\in \mathcal{P}(a,b)$; the Euler class $e$ is trivial if $d+1$ is odd; the classes $U_{a-1,b,2}\s p$ are trivial if $a$ is even and $p\in p_{\frac{a}{2}}\mathcal{P}(a,b)$; while the classes $U_{a,b-1,2}\s p$ are trivial if $b$ is even and $p\in p'_{\frac{b}{2}}\mathcal{P}(a,b)$. 
\end{remark}

\subsection{The case $a=2s$.}

Let $T_{\infty}$ denote the total space of the double covering $\pi_T$ associated with the cokernel bundle $C_1$ over the space $R_{\infty}$. Every filtration that we have over the space $R_{\infty}$ induces a filtration of the space $T_{\infty}$. We will write $T_r$ for $\pi_T^{-1}(R_r)$ and so on. 
The set $A_1'=A_{1}(2s, 2s+1)$ is bounded by the sets 
\begin{equation}\label{equ:14}
          A_2'=A_2(2s, 2s) \qquad \mathrm{and} \qquad A_2''=A_2(2s-1, 2s+1),
\end{equation}
and therefore the coboundary homomorphism may have two components. The restrictions of $\nu_2$ to the
 two spaces in (\ref{equ:14}) will be denoted by $\nu_2'$ and $\nu_2''$ respectively. Our computation of the $A_2'$ component of the differential relies on the commutative diagram
\[
\begin{CD}
    H^*(\Th\pi^*\nu_1) @>\delta>> H^*(\Th\pi^*\nu_2') \\
    @AAA @AAA \\
    H^*(\Th\nu_1) @>\delta>> H^*(\Th\nu_2'),     
\end{CD}
\]
where $\pi$ stands for the composition $\pi_R\circ \pi_T$. We observe that the Thom space $\Th\pi^*\nu_2'$ coincides with the Thom space $\Th\mathfrak{o}$ that we used in order to compute the cohomology group of the Thom space of $\nu_2'$. On the other hand, the space $\Th\pi^*\nu_1$ is a singular double cover of the corresponding space $\Th\mathfrak{o}$.

\begin{figure}[ht]
	\centering
			\includegraphics[draft=false, width=40mm]{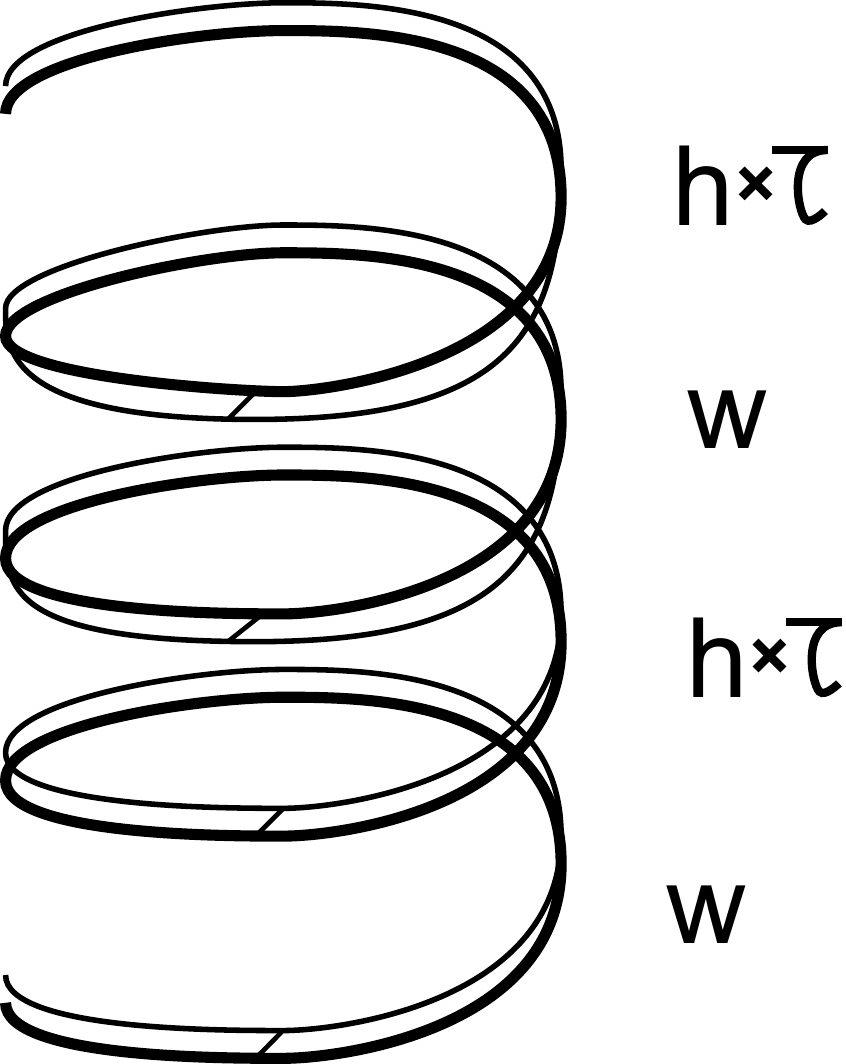}
\caption{The four fold covering $T_1'\to A_1\setminus A_0$. One of the two components of $T_1'$ adjacent to $T_2'\setminus T_1$ (bold).}
\label{fig:3}
\end{figure}

The deck transformation corresponding to the class $w$ preservers the two components of $\Th\pi^*\nu_1$, while the deck transformation corresponding to the class $w_1(C_1)$ exchanges the two components.  Next we choose Thom classes  $U'_{2s,2s+1,1}$ and $U''_{2s,2s+1,1}$ in the two path components of $\Th\pi^*\nu_1$. We choose the class $U'_{2s,2s+1,1}$ so that it is compatible with the canonical orientation of the kernel bundle and $U''_{2s,2s+1,1}$ to be the one that is not compatible with the canonical orientation of the kernel bundle, so that the classes $U'_{2s,2s+1,1}$ 
and $U''_{2s,2s+1,1}$ correspond to the same local orientation on the normal bundle of $A_{2s,2s,2}\setminus A_0$ in $A_{2s,2s,2}$. This guarantees that we may choose a Thom class for the vertical homomorphism on the right hand side of the diagram so that the diagram of Thom isomorphisms is commutative. 

\begin{lemma} The deck transformation of $w$ maps
\[
     U'_{2s, 2s+1, 1}\mapsto U'_{2s,2s+1, 1}, \qquad U''_{2s,2s+1, 1}\mapsto U''_{2s,2s+1,1}, 
\]
 while the deck transformation of $w_1(C_1)$ maps 
\[
     U'_{2s, 2s+1,1}\mapsto U''_{2s,2s+1,1}, \qquad U''_{2s,2s+1,1}\mapsto U'_{2s,2s+1,1}, 
\]
\[
     p_i\mapsto p_i' \qquad \mathrm{and} \qquad p_i'\mapsto p_i
\]
\end{lemma}
\begin{remark} Let us consider the deck transformation of $w_1(C_1)$. Consider the subgroup  $H\approx \Z_2\times \Z_2$ of the symmetry group of the minimal $\A_2(2s,2s)$ map germ generated by the two elements $\{w, h\times \tau\}$, where $w$ is an element of order $2$ in $[\O_a\times \O_a]^+$ that acts on $\R^{a}\times \R^{a}$ by reflecting the two copies of $\R^{a}$ along hyperplanes. The group $H$ acts on $T_2'\setminus T_1$ by deck transformations and on $A_2'\setminus A_1$ as a subgroup of the symmetry group of the minimal $\A_2(2s,2s)$ map germ; and the covering 
\[
  T_2'\setminus T_1\longrightarrow A_2'\setminus A_1
\] 
is $H$-equivariant. Furthermore, the action of $H$ extends over neighborhoods of $T_2'\setminus T_1$ in $T_2'$ and $A_2'\setminus A_1$ in $A_2'$ so that the covering of neighborhoods is still $H$-equivariant. We warn the reader, however, that the action of $H$ over the neighborhood of $T_2'\setminus T_1$ in $T_2'$ is different from the deck transformations (see Figure~\ref{fig:1}). For example, the transformation $H$ of the neighborhood of $T_2'\setminus T_1$ preserves its coorientation in $T_2'$, while the deck transformation of $h\times \tau$ reverses this coorientation.
\end{remark}

\begin{proof} The action of the deck transformation of $w$ can be deduced from the action of the  corresponding transformation $\beta$ in Table~\ref{t:ab1}. The deck transformation of $h\times \tau$ corresponds to the deck transformation of $\alpha$ in section~\ref{s:16.3}. Consequently,  
\[
   (h\times \tau)^*U_{2s,2s,2}=-U_{2s,2s,2}
\]
(see the action of $\alpha$ in Table~\ref{t:ab2}). The deck transformation of $h\times \tau$ reverses the coorientation of $T_2'\setminus T_1$ in $T_2'$. Hence 
\[
   (h\times \tau)^*U'_{2s,2s,1}=U''_{2s,2s,1}. 
\]
%On the other hand, we chose $U'_{2s,2s,1}$ and $U''_{2s,2s,1}$ so that these two classes correspond to the same local %coorientation of $A_2\setminus A_0$ in $A_2$. Therefore, the deck transformation of $w_1(C_1)$ maps
%\[
%     U'_{2s, 2s+1,1}\mapsto U''_{2s,2s+1,1}. 
%\]
The action of the deck transformation of $w_1(C_1)$ on $U''_{2s,2s+1,1}$, $p_i$ and $p_i'$ is computed similarly.

\end{proof}

Thus, the vertical homomorphism on the left hand side embeds the group $H^*(\Th\nu_1)$ to a subgroup of
\[
    \left\langle \ U'_{2s,2s+1,1}\s p_i + U''_{2s,2s+1,1}\s p'_i    \right\rangle,
\]
 while the vertical homomorphism on the right hand side embeds the group $H^*(\Th\nu_2')$ into the subgroup 
\[
  H^*(A'_2, A_1)=U'_{2s, 2s,2}\s (\mathcal{AP}+e\mathcal{AP}).
\]
Since the vector bundle $\pi^*\nu_1$ is orientable, there is a commutative diagram of Thom isomorphisms
\[
\begin{CD}
    \tilde{H}^{4s+1+*}(\Th\pi^*\nu_1) @>>> \tilde{H}^{4s+2+*}(\Th\pi^*\nu'_2) \\
    @AAA @AAA \\
    H^*(T_1') @>\delta>> H^{*+1}(T'_2, T_1). 
\end{CD}
\]
We note that the normal bundle of $T_2'\setminus T_1$ is orientable, but there is no canonical orientation. We fix one of the two possible orientations. We also note that the set $T_2'\setminus T_1$ is connected, while $T_1$ has two components. Without loss of generality we may assume that the coorientation of $T_2'\setminus T_1$ is directed toward the first component of $T_1$. 

Let $f_1$ be the zero cochain in $T_1$ that assumes values $1$ and $0$ on the first and second components of $T_1$ respectively. Similarly, let $f_2$ be the zero cochain in $T_1$ that assumes $1$ and $0$ on the second and first components respectively. Then $\delta(f_1)$ and $\delta(f_2)$ equal the generator and negative generator $\iota$ of $H^1(T'_2, T_1)$ respectively. Furthermore, if $p$ is a cohomology class of $T_1$ that on both components of $T_1$ is the same polynomial of the form $\mathcal{SP}(2s,2s)$, then 
\[
\delta(f_1p)=\iota\cup p \qquad \mathrm{and} \qquad \delta(f_2p)=-\iota\cup p.
\]
On the other hand the Thom isomorphism on the left hand side of the above diagram takes the class $(f_1+f_2)p=p$ to the class
\[
    U'_{2s,2s+1,1}\cup p + U''_{2s,2s+1,1}\cup p.
\]
Since $\delta[(f_1+f_2)p]=0$, we conclude that the $A_2'$ component of $d^{1,*}_1$ is trivial on all classes of the form $U_{2s,2s+1,1}\cup p$ for $p$ in $\mathcal{SP}(2s,2s)$. 

Let now $p$ be a cohomology class on $T_1$ that on both components of $T_1$ is the same polynomial of the form $\mathcal{AP}(2s, 2s)$. Then 
\[
   \delta[(f_1-f_2)p]=2\iota\cup p,
\]
while the Thom isomorphism on the left hand side of the above diagram takes $\delta[(f_1-f_2)p]$ to
\[
    U'_{2s,2s+1,1}\cup p - U''_{2s,2s+1,1}\cup p.
\]
Consequently, the $A_2'$ component of $d^{1,*}_1(U_{2s,2s+1,1}\s p)$ equals $2U_{2s, 2s, 2}\s p$ for $p$ of the form $\mathcal{AP}(2s, 2s)$.

Next we consider the $A_2''$ component of the differential $d^{1,*}_1$. In this case we use the argument of the computation of the $A''_2$ component of the differential $d^{1,*}_1$ for $a=1,..., 2s-1$; in particular we utilize the double covering $\pi_R$ instead of $\pi$. We deduce from the equation (\ref{eq:18.18}) that the $A''_2$ component of $d^{1,*}_1$ maps
\[
     U_{2s, 2s+1, 1}\cup p \mapsto U_{2s-1, 2s+1, 2}\cup p
\]
for every class $p\in \mathcal{P}(2s, 2s+1)$. To summarize, we have computed that   
\[    
    d^{1,*}_1(U_{2s,2s+1, 1}\s p)=U_{2s-1, 2s+1,2}\s p  
\]
for $p$ in $\mathcal{SP}(2s, 2s)$ and 
\[    
    d^{1,*}_1(U_{2s,2s+1, 1}\s p)=U_{2s-1, 2s+1,2}\s p + 2U_{2s,2s,2}\s p 
\]
for $p$ in $\mathcal{AP}(2s, 2s)$.

\begin{remark} In the general case of a map of dimension $d>0$ the computations are slightly different. The deck transformation of $w_1(C_1)$ maps
\[
    U'\mapsto U'',\qquad U''\mapsto U'
\] 
for $d=4s$ and $d=4s+3$, and 
\[
    U'\mapsto -U'',\qquad U''\mapsto -U'
\] 
for $d=4s+1$ and $d=4s+2$. We deduce that
\[
   d^{1,*}_1(U_{a,b,1}\s p)= U_{a-1,b,2}\s p, \qquad \mathrm{for} \ p\in \mathcal{SP}(a,b),
\]
\[
   d^{1,*}_1(U_{a,b,1}\s p)= U_{a-1,b,2}\s p+2U_{a,b-1,2}\s p, \qquad \mathrm{for} \ p\in \mathcal{AP}(a,b)
\]
 for $d$ even and $a=d/2$, $b=d+1-a$; and
\[
   d^{1,*}_1(U_{a,b,1}\s ep)= 0, \qquad \mathrm{for} \ p\in \mathcal{SP}(a,b),
\]
\[
   d^{1,*}_1(U_{a,b,1}\s p)= U_{a-1,b,2}\s p, \qquad \mathrm{for} \ p\in \mathcal{AP}(a,b)
\]
for $d$ odd and $a=\frac{d+1}{2}$, $b=d+1-a$. In the case $a=b=2s$, the choice of a auxiliary covering is different is different; on the other hand, in this case one may avoid computations by chasing Table~\ref{t:13} and using the property that $d^{1,*}_1\circ d^{0, *}_1=0$.
\end{remark}

\begin{theorem}
In the case where the dimension is of the form $d=4s$, the kernel of $d^{1,*}_1$ is generated by the classes 
\[
      \tau_{2s,p}= U_{0, 4s+1, 1}\s p - U_{1, 4s, 1}\s p,
\]
for $p\in p'_{2s}\P(0, 4s)$, 
\[
      \tau_{2s-1,p}= U_{0, 4s+1, 1}\s p - U_{1, 4s, 1}\s p + U_{2, 4s-1, 1}\s p - U_{3, 4s-2, 1}\s p,
\]
for $p\in p'_{2s-1}\P(3, 4s-2)$, 
\[
 \cdots
\]
\[
      \tau_{s+1,p}= U_{0, 4s+1, 1}\s p +\cdots + (-1)^a U_{a, b, 1}\s p +\cdots + \cdots + (-1)U_{2s-1, 2s+1, 1}\s p,
\]
for $p\in p'_{s+1}\P(2s-1, 2s+1)$, and
\[
   \sigma_Q = U_{0, 4s+1, 1}\s Q - U_{1, 4s,1}\s Q +\cdots + U_{2s, 2s+1,1}\s Q 
\]
for each $Q\in \SP(2s, 2s)$. 
\end{theorem}
\begin{proof} Let $d_{k,l}$ denote the component of the differential $d^{1,*}_1$ that is a homomorphism from the subgroup $\left\langle U_{k, d+1-k,1}\s p\right\rangle$ to the subgroup $\left\langle U_{l, d-l,2}\s p\right\rangle$. In particular $d_{k,l}$ is non-trivial only if $l=k$ or $l=k-1$. Consider $d_{2s,2s}$. Its kernel is $U\s \SP(2s,2s)$. Chasing the Table~\ref{table:10} from the top to the bottom, we deduce that the classes $\sigma_{Q}$ are in the kernel of $d^{1,*}_1$. Next we consider $d_{2s-1, 2s-1}$. Suppose $x$ is a polynomial in $\Q[p_1,..., p_{s-1}, p_1', ..., p'_{s+1}]$, $x$ is not the homomorphic image of a class in $\SP(2s,2s)$ and $U_{2s-1, 2s+2,1}\s x$ is a component of a class in the kernel of $d^{1,*}_1$. Then $d_{2s-1, 2s-1}(x)=0$. Consequently $x$ is in the form $p'_{s+1}\mathcal{P}$, which gives rise to elements $\tau_{s+1,p}$ in the kernel of $d^{1,*}_1$. Next we consider the homomorphism $d_{2s-2, 2s-2}$ and repeat the argument. In $2s$ steps we determine all elements in the kernel $d^{1,*}_1$. 
\end{proof}

\begin{lemma}\label{l:123} The sequence of linear homomorphisms
\[
 0\longrightarrow \P(4s) \stackrel{\mathfrak{s}}\longrightarrow \P(2s,2s)\longrightarrow \AP(2s,2s)\longrightarrow 0
\]
is exact, where the homomorphism $\mathfrak{s}$ is an algebra homomorphism that takes $p_i$ to
\[
   p_i+p_{i-1}p_1'+\cdots p_1p_{i-1}'+p_i,
\]
while the next homomorphism is a projection along $\SP(2s,2s)$. 
\end{lemma}
\begin{proof} Let $\P(4s)\to \Q[t_1,...,t_{2s}]$ be an algebra homomorphism that takes $p_i$ to the $i-th$ symmetric polynomial, i.e., the symmetric polynomial that contains $t_1t_2\cdots t_i$. It is known that this homomorphism is an isomorphism onto the subalgebra of symmetric polynomials. This map factors as a composition of $\mathfrak{s}$ and the homomorphism $\P(2s,2s)\to \Q[t_1,...,t_{2s}]$ that takes $p_i$ and $p_i'$ to the $i$-th symmetric polynomials in $t_1,..., t_s$ and $t_{s+1}, ..., t_{2s}$ respectively. Consequently $\mathfrak{s}$ is injective an its image consists of symmetric polynomials $\SP(2s,2s)$. 
\end{proof}

\begin{lemma} The classes $\tau_{j,p}$ form a linear basis of $\ker d^{1,*}_1/ \im d^{0,1}_1$.
\end{lemma}
\begin{proof} Every class in $\ker d^{1,*}_1$ has $2s+1$ components, one for each index $i=0,...,2s$. The $(2s)$-th component of each linear combination of the classes $\tau_{j,p}$ is trivial while the $(2s)$-th component of every element in $\im d^{0,*}_1$ is non-trivial. On the other hand, by Lemma~\ref{l:123}, the dimension of the vector space spanned by the classes $\tau_{j,p}$ coincides with the dimension of the vector space $\ker d^{1,*}_1/ \im d^{0,1}_1$. 
\end{proof}

\section{The differential $d^{2r,*}_{1}$}
\subsection{The case $a\ne b$.} Over the space $A_{\infty}\setminus A_1$ there is a well-defined vector space $K_1/K_2$ with first Stiefel-Whitney class $v=w_1(K_1/K_2)$. We consider a four fold cover $\pi_P$ of the space $A_{\infty}\setminus A_1$ by a space $P_{\infty}$ classified by the two cohomology classes $v$ and $w$. Again, each of the defined filtrations of $A_{\infty}\setminus A_1$ determines a filtration on the space $P_{\infty}$; we will write $P_k$ for $\pi_P^{-1}(A_k\setminus A_1)$ for each $k>1$. 

The set $A_{2r}'=A_{2r}(a, b)$ is bounded only by the set $A_{2r+1}'=A_{2r+1}(a, b)$ and therefore the coboundary homomorphism has only one component, which we compute by means of the commutative diagram  
\[
\begin{CD}
    H^*(\Th\pi_P^*\nu_{2r}') @>\delta>> H^*(\Th\pi_P^*\nu_{2r+1}') \\
    @AAA @AAA \\
    H^*(\Th\nu_{2r}') @>\delta>> H^*(\Th\nu_{2r+1}'),     
\end{CD}
\]
where $\nu_{2r}'$ and $\nu_{2r+1}'$ are the restrictions of $\nu_{2r}$ and $\nu_{2r+1}$ to the components $A_{2r}'$ and  $A_{2r+1}'$ respectively. The Thom space $\Th\pi^*\nu_{2r+1}'$ coincides with the space $\Th\mathfrak{o}$ used in the computation of $H^*(\Th\nu_{2r+1}')$. In fact we expressed the latter group as a subgroup 
\[
    \tilde{H}^*(\Th\nu_{2r+1}')\approx U^{\pm}_{a,b,2r+1}\s e_{a,b}\mathcal{P}(a,b)
\]
if $r$ is odd, and 
\[
    \tilde{H}^*(\Th\nu_{2r+1}')\approx U^{\pm}_{a,b,2r+1}\s \mathcal{P}(a,b)
\]
if $r$ is even.  On the other hand, the space $\Th\pi^*\nu_{2r}'$ is a singular double cover of the corresponding space $\Th\mathfrak{o}$. Let $U'$ and $U''$ be two Thom classes in the two path components of $\Th\pi^*\nu_{2r}'$ chosen so that $U'$ and $U''$ correspond to the same local orientation of the normal bundle of $A_{2r+1}(a,b)\setminus A_{2r-1}$ in $A_{2r+1}(a,b)$.  The deck transformation of $w$ preserves the two sheets in $\Th\pi^*\nu_{2r}$ and $w_*$ maps
\[
    U'\mapsto U', \qquad    U''\mapsto U'',
\] 
(see the corresponding transformation $\beta$ in Table~\ref{t:ab1}), while the deck transformation of $v$ exchanges the two sheets in $\Th\pi^*\nu_{2r}'$.

\begin{figure}[ht]
	\centering
			\includegraphics[draft=false, width=40mm]{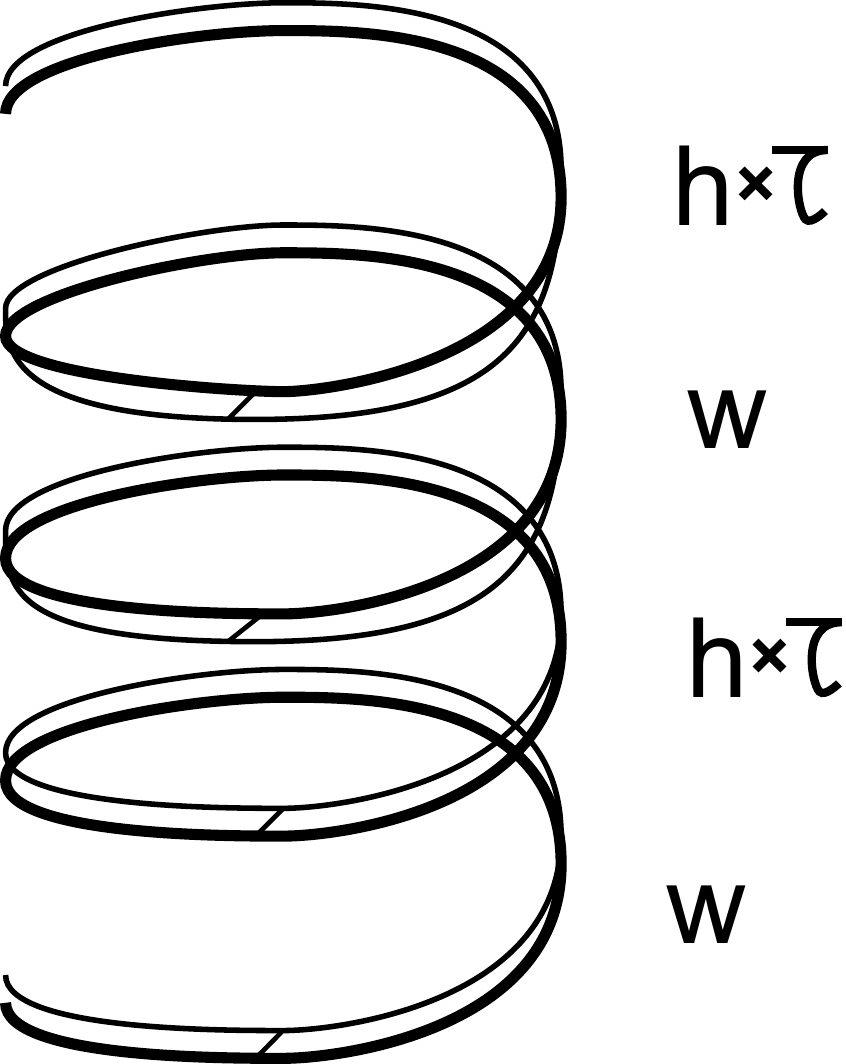}
\caption{The four fold covering $P_{2r}'\to A_{2r}\setminus A_1$. One of the two components of $P_{2r}'$ adjacent to $P_{2r+1}'\setminus P_{2r}$ (bold).}
\label{fig:1}
\end{figure}

\begin{lemma}~\label{l:19.1} The action of $v$ maps
\[
    U'\mapsto (-1)^{r+1}U'', \quad U''\mapsto (-1)^{r+1}U', \quad p_i\mapsto p_i, \quad p'_i\mapsto p'_i, \quad e\mapsto -e.
\]
\end{lemma}
\begin{proof} Let us assume that $r+1$ is odd, the argument in the case of $r$ even is similar and will be omitted. Then the relative symmetry group of the minimal $\A_{2r+1}(a,b)$ map germ contains a subgroup $H\approx \Z_2\times \Z_2$ generated by an element $h$ and by an element $w$ that acts on $\R^{a}\times \R^b\times \R^{2r+1}$ by orientation reversing involutions on the first and the third factors. The group $H$ acts both on the space $A_{2r+1}'\setminus A_{2r}$ and on $P_{2r+1}'\setminus P_{2r}$ so that the covering 
\[
    A_{2r+1}'\setminus A_{2r} \longrightarrow P_{2r+1}'\setminus P_{2r}
\]
is $H$-equivariant. This action extends to the action of $H$ on neighborhoods of the two spaces in $P_{2r+1}'$ and $A_{2r+1}'$ respectively, so that the double covering of neighborhoods is still $H$-equivariant. We observe that $h$ reverses the coorientation of the set $A_{2r+1}'\setminus A_{2r}$ in the set $A_{2r+1}'$ since it acts on the linear vector space $K_2$ non-trivially. Consequently, the action of $h$ does not coincide with the deck transformation of $v$. 
Thus the deck transformation of $v$ takes $U'$ to $\pm U''$ and $U''$ to $\pm U'$. 

The deck transformation of $v$ corresponds to the deck transformation of $\alpha$ in section~\ref{s:16.4}, which takes $U_{a,b,2r+1}$ to $U_{a,b,2r+1}$. Since $v$ reverses the coorientation of $P'_{2r+1}\setminus P_{2r}$ in $P'_{2r+1}\setminus P_{2r-1}$ we conclude that the deck transformation of $v$ takes 
\[
U'\mapsto - U'' \qquad \mathrm{and} \qquad 
U''\mapsto - U'. 
\]
Let us choose an Euler class $e$ over $P_{2r+1}'\setminus P_{2r}$, it is an Euler class of the vector bundle $K_1/K_2$. Then we may choose an Euler class over $P_{2r}'$ so that it is compatible with the Euler class $e$. Then by the computation in Lemma~\ref{l:12.6}, the deck transformation of $v$ takes $e$ to $-e$ (see the action of $\alpha$ on $e_{a,b}$).  Similarly we deduce that $h$ takes $p_i$ to $p_i$ and $p'_i$ to $p'_i$. 
\end{proof}

Consequently, the vertical homomorphism on the left hand side of the diagram embeds the cohomology group of the latter space onto the subgroup
\[
     \{\ U'\s p + (-1)^{r+1}U''\s p\ |\ p\in \mathcal{P}(a,b)    \ \}\oplus      \{\ U'\s ep + (-1)^{r}U''\s ep\ |\ p\in \mathcal{P}(a,b)    \ \}; 
\] 
the Euler class $e$ is trivial for $a$ odd. There is a commutative diagram of Thom isomorphisms
\[
\begin{CD}
    \tilde{H}^*(\Th\pi_P^*\nu_{2r}') @>>> \tilde{H}^*(\Th\pi_P^*\nu_{2r+1}') \\
   @AAA @AAA \\
   H^*(P_{2r}'\setminus P_{2r-1}) @>>> H^*(P_{2r+1}'\setminus P_{2r-1}, P_{2r}\setminus P_{2r-1}). 
\end{CD}
\]
The normal bundle of $P_{2r}'\setminus P_{2r-1}$ in $P_{2r}'$ is isomorphic to the bundle $(K_1/K_2)\oplus 2rK_2$, which is orientable but does not have a canonical orientation. Again we have that the two components of $P_{2r}'\setminus P_{2r-1}$ are adjacent to the one component of $P_{2r+1}'\setminus P_{2r}'$. We may assume that the generator $\iota$ in the group $H^1(P_{2r+1}'\setminus P_{2r-1}/P_{2r}'\setminus P_{2r-1})$ is chosen so that it is compatible with the coorientation in the direction of the first component of $P_{2r}'\setminus P_{2r-1}$. Let $f_1$ be the cochain of degree $0$ that assumes value $1$ on the first component of $P_{2r}'\setminus P_{2r-1}$ and the value $0$ on the other component. Let $f_2$ be the cochain $1-f_1$. For $r+1$ odd we have
\[
   \delta(f_1-f_2)p=\iota\s p+\iota\s p= 2\iota\s p, \qquad (f_1-f_2)p\xrightarrow{\Th} U'\s p-U''\s p,
\]
\[
   \delta(f_1+f_2)ep=\iota\s ep=0, \qquad (f_1+f_2)p\xrightarrow{\Th} U'\s ep + U''\s ep, 
\]
and therefore
\[
d^{2r,*}_{1}(U_{a,b,2r}\s p)=2U^+_{a,b,2r+1}\s p+2U^-_{a,b,2r+1}\s p  \qquad \mathrm{for} \ p\in \mathcal{P}(a,b),
\]
and
\[
d^{2r,*}_{1}(U_{a,b,2r}\s ep)= 0 \qquad \mathrm{for} \ p\in \mathcal{P}(a,b).
\]
For $r+1$ even we have
\[
   \delta(f_1+f_2)p=\iota\s p-\iota\s p=0, \qquad (f_1+f_2)p\xrightarrow{\Th} U'\s p+U''\s p,
\]
\[
   \delta(f_1-f_2)ep=\iota\s ep=2\iota\s ep, \qquad (f_1-f_2)p\xrightarrow{\Th} U'\s ep - U''\s ep, 
\]
and therefore
\[
d^{2r,*}_{1}(U_{a,b,2r}\s p)=0  \qquad \mathrm{for} \ p\in \mathcal{P}(a,b),
\]
and
\[
d^{2r,*}_{1}(U_{a,b,2r}\s ep)= 2U^+_{a,b,2r+1}\s ep+2U^-_{a,b,2r+1}\s ep \qquad \mathrm{for} \ p\in \mathcal{P}(a,b).
\]

\begin{remark} In the general case of maps of dimension $d>0$, computations are similar. We deduce that 
for $r+1$ odd 
\[
d^{2r,*}_{1}(U_{a,b,2r}\s p)=2U^+_{a,b,2r+1}\s p+2U^-_{a,b,2r+1}\s p  \qquad \mathrm{for} \ p\in \mathcal{P}(a,b),
\]
\[
d^{2r,*}_{1}(U_{a,b,2r}\s ep)= 0 \qquad \mathrm{for} \ p\in \mathcal{P}(a,b),
\]
where $e$ is trivial if $d$ is odd; and for $r+1$ even
\[
d^{2r,*}_{1}(U_{a,b,2r}\s p)=0  \qquad \mathrm{for} \ p\in \mathcal{P}(a,b),
\]
\[
d^{2r,*}_{1}(U_{a,b,2r}\s ep)= 2U^+_{a,b,2r+1}\s ep+2U^-_{a,b,2r+1}\s ep \qquad \mathrm{for} \ p\in \mathcal{P}(a,b),
\]
where again $e$ is trivial if $d$ is odd. 
\end{remark}

\subsection{The case $a=b$.}

The argument is similar to that in the case $a\ne b$. We conclude that 
\[
d^{2r,*}_{1}(U_{a,b,2r}\s p)=2U_{a,b,2r+1}\s p  \qquad \mathrm{for} \ p\in \mathcal{SP}(a,b),
\]
and
\[
d^{2r,*}_{1}(U_{a,b,2r}\s ep)= 0 \qquad \mathrm{for} \ p\in \mathcal{SP}(a,b)
\]
for $r+1$ odd, and 
\[
d^{2r,*}_{1}(U_{a,b,2r}\s p)=0  \qquad \mathrm{for} \ p\in \mathcal{AP}(a,b),
\]
and
\[
d^{2r,*}_{1}(U_{a,b,2r}\s ep)= 2U_{a,b,2r+1}\s ep  \qquad \mathrm{for} \ p\in \mathcal{AP}(a,b)
\]
for $r+1$ even.

\begin{remark} In the case where $d$ is odd, we always have $a<b$. In the case $d=4s+2$ the above formulas hold true with $e=0$. 
\end{remark}

\section{The differential $d_1^{2r+1,*}$}

\subsection{The case $a\ne b$.} 
For computation of the differential $d_1^{2r+1,*}$ we will again use the space $P_{\infty}$ together with its filtrations. The set $A_{2r+1}'=A_{2r+1}(a, b)$ is bounded only by the set $A_{2r+2}'=A_{2r+2}(a, b)$; we will compute $d^{2r+1,*}_1$ by means of the commutative diagram  
\[
\begin{CD}
    H^*(\Th\pi_P^*\nu_{2r+1}') @>>> H^*(\Th\pi_P^*\nu_{2r+2}') \\
    @AAA @AAA \\
    H^*(\Th\nu_{2r+1}') @>>> H^*(\Th\nu_{2r+2}'),     
\end{CD}
\]
where $\nu_{2r+1}'$ and $\nu_{2r+2}'$ are the restrictions of $\nu_{2r+1}$ and $\nu_{2r+2}$ to the components $A_{2r+1}'$ and  $A_{2r+2}'$ respectively. We note that $P'_{2r+1}=\pi^{-1}_P(A_{2r+1}')$ consists of a component of $\A^+_{2r+1}$ map germs and a component of $\A^-_{2r+1}$ map germs. The Thom space of $\pi_P^*\nu_{2r+1}'$ coincides with the Thom space $\Th\mathfrak{o}$ used in the computation of the cohomology group of  $A_{2r+1}'/A_{2r}$; we have computed that the Thom isomorphism on the left hand side of the diagram embeds the cohomology group of $\Th\nu_{2r+1}'$ onto the subgroup 
\[
   \tilde{H}^*(\Th\nu_{2r+1}')\approx  U^{\pm}_{a,b,2r+1}\s e_{a,b}\P(a,b)
\] 
if $r$ is odd, and    
\[   
 \tilde{H}^*(\Th\nu_{2r+1}')\approx U^{\pm}_{a,b,2r+1}\s \P(a,b)
\]
if $r$ is even. On the other hand, the Thom space of $\pi_P^*\nu_{2r+2}'$ is a double cover over the space $\Th\mathfrak{o}$ used in the computation of the cohomology group of $A_{2r}'/A_{2r-1}$. We choose the coorientation of $P'_{2r+2}\setminus P_{2r+1}$ in $P'_{2r+2}\setminus P'_{2r-2}$ in the direction of $\A^+_{2r+1}$ map germs. Let $U'$ and $U''$ be the two components of the Thom class of $\pi_P^*\nu_{2r+2}'$ that locally correspond to the same orientation of the normal bundle. 

\begin{lemma}
The deck transformation of $v$ maps 
\[
    U'\mapsto (-1)^{r}U'', \quad U''\mapsto (-1)^{r} U', \quad p\mapsto p, \quad e\mapsto -e. 
\]
\end{lemma}
\begin{proof} The transformation $h\times \tau$ takes $U^{\pm}_{a,b,2r+1}$ to $(-1)^{r}U_{a,b,2r+1}$, and therefore $v$ maps 
\[
   U'\mapsto (-1)^rU'', \qquad U''\mapsto (-1)^rU'.
\]
\end{proof}

Consequently, the vertical homomorphism on the left hand side of the diagram embeds the cohomology group of the latter space onto the subgroup
\[
     \{\ U'\s p + (-1)^rU''\s p\ |\ p\in \mathcal{P}(a,b)    \ \}\oplus      \{\ U'\s ep + (-1)^{r+1}U''\s ep\ |\ p\in \mathcal{P}(a,b)    \ \}; 
\] 
the Euler class $e$ is trivial for $a$ odd. There is a commutative diagram
\[
\begin{CD}
   \tilde{H}^{*+d+2r+1}(P_{2r+1}'/P_{2r}) @>d^{2r+1,*}_{1}>> \tilde{H}^{*+d+2r+2}(P_{2r+2}'/P_{2r+1}') \\
   @AAA @AAA \\
   H^*(P_{2r+1}'\setminus P_{2r}) @>\delta>> H^*(P_{2r+2}'\setminus P_{2r}, P_{2r+1}\setminus P_{2r}). 
\end{CD}
\]
Two components of $P'_{2r+1}$ are adjacent to the two components of $P'_{2r+2}$.

\begin{figure}[ht]
	\centering
			\includegraphics[draft=false, width=40mm]{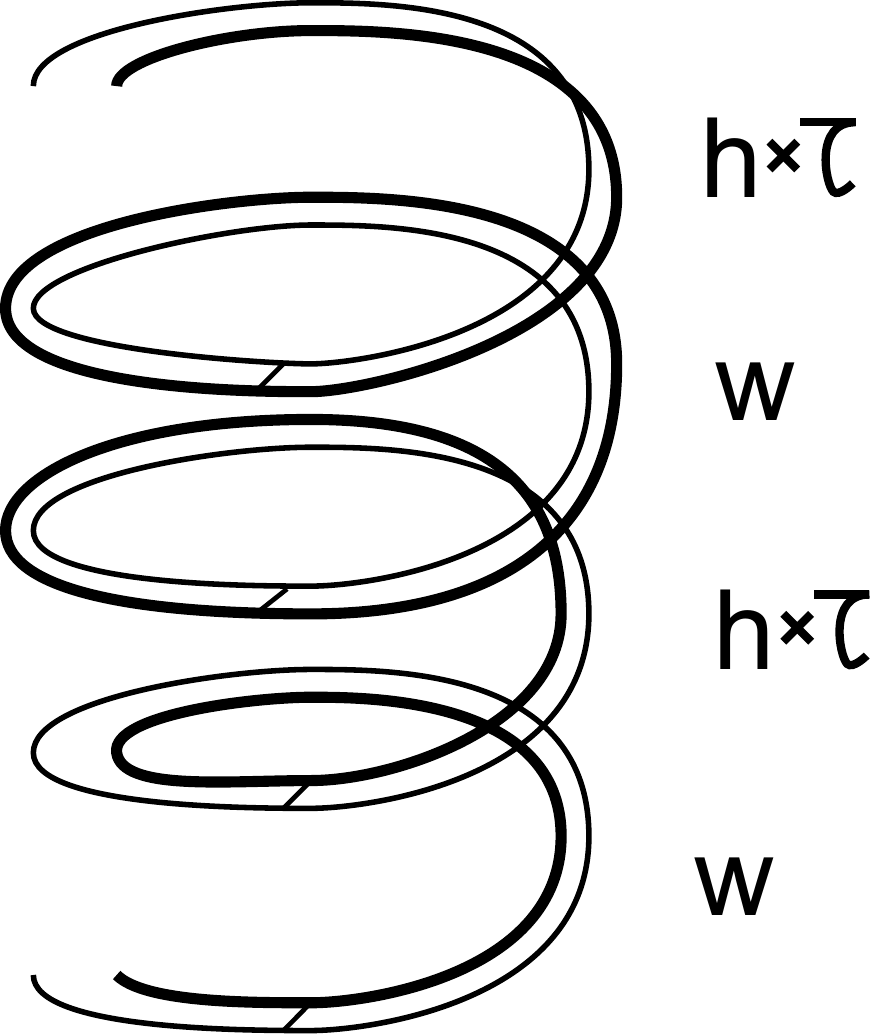}
\caption{The four fold covering $P_{2r+2}'\setminus P_{2r}\to A_{2r+2}'\setminus A_{2r}$. One of the two components of $P_{2r+1}'$ adjacent to $P_{2r+2}'$ (bold).}
\label{fig:2}
\end{figure}

Let $f_1$ be the cochain of degree $0$ that assumes value $1$ on the component $P^+_{2r+1}(a,b)\setminus P_{2r}$ and the value $0$ on the other component. Let $f_2$ be the cochain $1-f_1$. 

Suppose $r$ is odd. Then 
\[
   \delta(f_1ep)=\iota'\s ep+ \iota''\s ep,
\]
\[
   \delta(f_2ep)=-\iota'\s ep - \iota''\s ep,
\]
\[
   \Th(f_1ep)=U^+_{a,b,2r+1}\s ep, \qquad \Th(f_2ep)=U^-_{a,b,2r+1}\s ep,
\]
where $\iota'$ and $\iota''$ are the two Thom classes that correspond to the coorientation in the direction of $\A^+_{2r+1}$ map germs. Consequently, 
\[
  d^{2r+1,*}_1\colon U^+_{a,b,2r+1}\s ep + U^-_{a,b,2r+1}\s ep \mapsto 0, 
\]
\[
  d^{2r+1,*}_1\colon U^+_{a,b,2r+1}\s ep - U^-_{a,b,2r+1}\s ep \mapsto 2U_{a,b,2r+2} \s ep. 
\]
Suppose, now, that $r$ is even. Then 
\[
   \delta(f_1p)=\iota'\s p+ \iota''\s p, 
\]
\[
   \delta(f_2p)=-\iota'\s p- \iota''\s p,
\]
\[
   \Th(f_1p)=U^+_{a,b,2r+1}\s p, \qquad \Th(f_2p)=U^-_{a,b,2r+1}\s p,
\]
Consequently, 
\[
  d^{2r+1,*}_1\colon U^+_{a,b,2r+1}\s p + U^-_{a,b,2r+1}\s p \mapsto 0, 
\]
\[
  d^{2r+1,*}_1\colon U^+_{a,b,2r+1}\s p - U^-_{a,b,2r+1}\s p \mapsto 2U_{a,b,2r+2} \s p. 
\]

\begin{remark} The above formulas hold true for arbitrary $d>0$ with $e=0$ if $d$ is odd. 
\end{remark}

\subsection{The case $a=b$}  Let $\pi_T$ denote the double covering $T_{\infty}\to A_{\infty}\setminus A_1$ associated with $w$ and the class $w_1(C_1)$. To compute the coboundary homomorphism for the adjacency of the set $A_{2r+1}'=A_{2r+1}(a, b)$ to the set $A_{2r+2}'=A_{2r+2}(a, b)$, we consider the commutative diagram  
\[
\begin{CD}
    H^*(\Th\pi_P^*\nu_{2r+1}') @>>> H^*(\Th\pi_P^*\nu_{2r+2}') \\
    @AAA @AAA \\
    H^*(\Th\nu_{2r+1}') @>>> H^*(\Th\nu_{2r+2}'),     
\end{CD}
\]
where $\nu_{2r+1}'$ and $\nu_{2r+2}'$ are the restrictions of $\nu_{2r+1}$ and $\nu_{2r+2}$ to the components $A_{2r+1}'$ and  $A_{2r+2}'$ respectively. We note that the space $T'_{2r+2}=\pi^{-1}_T(A_{2r+2}')$ consists of a single component, and that the Thom space of $\pi_T^*\nu_{2r+2}'$ coincides with the Thom space $\Th\mathfrak{o}$ that we used to compute the cohomology group of  $A_{2r+2}'/A_{2r+1}$; in particular, we have 
\[
   \tilde{H}^*(\Th\nu_{2r+2}')\approx  U_{a,b,2r+2}\s (\mathcal{SP}\oplus  e\mathcal{SP}(a,b))
\] 
if $r$ is even, and    
\[   
 \tilde{H}^*(\Th\nu_{2r+2}')\approx U_{a,b,2r+2}\s (\mathcal{AP}(a,b)\oplus e\mathcal{AP}(a,b))
\]
if $r$ is even. On the other hand, the Thom space of $\pi_P^*\nu_{2r+1}'$ is a double cover over the space $\Th\mathfrak{o}$ used in the computation of the cohomology group of $A_{2r+1}'/A_{2r}$. It has too components and we choose the coorientation of $T'_{2r+2}\setminus T_{2r+1}$ in $T'_{2r+2}\setminus T'_{2r-2}$ in the direction of the first component. The Thom classes $U'$ and $U''$ for the two components of $\pi_T^*\nu_{2r+1}'$ are chosen to locally  correspond to the same orientation of the normal bundle.

\begin{lemma}
The deck transformation of $w_1(C)$ maps 
\[
    U'\mapsto (-1)^rU'', \quad U''\mapsto (-1)^rU', \quad p\mapsto p', \quad e\mapsto e. 
\]
\end{lemma}
\begin{proof} The transformation $h\times \tau$ takes 
\[
U_{a,b,2r+2}=U_{a,b,2r'}\qquad \mathrm{to} \qquad (-1)^{r'}U_{a,b,r'}=(-1)^{r+1}U_{a,b,2r+2}
\] 
(see the transformation $\alpha$ in Table~\ref{t:ab2}), and therefore the deck transformation of $w_1(C)$ maps 
\[
   U'\mapsto (-1)^rU'', \qquad U''\mapsto (-1)^rU'.
\]
\end{proof}

Consequently, the cohomology group $H^*(\Th \nu'_{2r+1})$ is isomorphic to 
\[
     \{\ U'\s p_i + U''\s p_i'\ \} 
\] 
if $r$ is even and 
\[
      \{\ U'\s ep_i -U''\s ep'_i\ \}
\]
if $r$ us odd.

Next we turn to the commutative diagram
\[
\begin{CD}
   \tilde{H}^{*+d+2r+1}(\Th p^*_T\nu'_{2r+1}) @>>> \tilde{H}^{*+d+2r+2}(\Th p^*_T\nu'_{2r+2}) \\
   @AAA @AAA \\
   H^*(T_{2r+1}'\setminus T_{2r}) @>>> H^*(T_{2r+2}'\setminus T_{2r}, T_{2r+1}\setminus T_{2r}). 
\end{CD}
\]

Let $f_1$ be a function on $T_{2r+1}'\setminus T_{2r}$ that on the first component assumes value $1$ and on  the second component assumes value $0$. Let $f_2$ be the cochain $1-f_1$. Then for every $p\in \mathcal{P}\oplus e\mathcal{P}$,
\[
   \delta(f_1p)=\iota\s p, \qquad
   \delta(f_2p)=-\iota\s p,
\]
\[
   \Th(f_1p)=U'\s p, \qquad \Th(f_2p)=U''\s p.
\]
Consequently, for $r$ even and $p\in \mathcal{SP}$,  
\[
  d^{2r+1,*}_1\colon U_{a,b,2r+1}\s p \mapsto 0, 
\]
for $r$ even and $p\in \mathcal{AP}$,  
\[
  d^{2r+1,*}_1\colon U_{a,b,2r+1}\s p \mapsto 2U_{a,b,2r+2}\s p, 
\]
while for $r$ odd and $p\in \mathcal{SP}$, 
\[
  d^{2r+1,*}_1\colon U_{a,b,2r+1}\s ep \mapsto 2U_{a,b,2r+2}\s ep, 
\]
for $r$ oedd and $p\in \mathcal{AP}$,  
\[
  d^{2r+1,*}_1\colon U_{a,b,2r+1}\s ep \mapsto 0.
\]

\begin{remark} The above formulas hold true for maps of dimension $d=4s+2$ as well. In the case where $d$ is odd, we always have $a<b$. 
\end{remark}

\addcontentsline{toc}{part}{Cohomology groups of $\mathbf{A_r}$}
\section{The cohomology group of $\mathbf{A_r}$ for $d=4s$}

The differentials $d^{*,*}_t$ with $t>1$ are trivial since $E_2^{i, *}$ are trivial for $i>1$. Thus we are in position to list the cohomology group of each spectrum $\mathbf{A_r}$ for $d=4s$. For a graded algebra $A$ over a field $\k$, let 
\[
    \rk(G)=g_0+g_1t+g_2t^2+\cdots
\] 
denote the series where each $g_i$ stands for the dimension of the subspace of $A$ of vectors of grade $i$. We will be interested in series 
\[
    \mathfrak{p}_t(a,b)=\rk(\mathcal{P}(a,b)), \qquad    \mathfrak{s}_t(a,b)=\rk(\mathcal{SP}(a,b)),
    \]
    \[ 
    \mathfrak{a}_t(a,b)=\rk(\mathcal{AP}(a,b)), \qquad
   \mathfrak{b}_t(d)=\rk(\mathcal{P}(d)). 
\]

\begin{theorem}\label{th:4s} The rank series of the cohomology ring of $\mathbf{A_1}$ is
\[
   \mathfrak{p}_t(0, d+1) + t^{d+1}\sum_{a=1}^{d/2} \mathfrak{p}_t(a, d+1-a).
\]
\end{theorem}  
\begin{proof}  By Remark~\ref{r:17.3}, for $d$ even the spectrum $\mathbf{A_1}$ splits as
\[
     \mathbf{A_1}=[\mathbf{A_1(0,d+1)}]\vee [\mathbf{A_1(1,d)/A_0}] \vee \cdots \vee [\mathbf{A_1(d/2,d/2+1)/A_0}].
\] 

\end{proof}

\begin{theorem} The rank series of the cohomology ring of $\mathbf{A_{2r}}$ for $r$ odd is
\[
    \mathfrak{b}_t(d) + t^{d+1}\sum_{j=0}^{s-1}t^{4(2s-j)}\mathfrak{p}_t(2j+1, 4s-2j) + 
    \]
    \[ 
    + t^{8s+2r}\mathfrak{a}_t(2s,2s) + t^{8s+2r}\sum_{i=0}^{s-1} \mathfrak{p}_t(2i,4s-2i).
\]
\end{theorem}
\begin{proof} The cohomology algebra of $\mathbf{A_{2r}}$ is generated by $\ker d_1^{0,*}$, which contributes the polynomial $\mathfrak{b}_t(d)$; the classes $\tau_{2s-j,p}$ for $j=0, ...s-1$ (see Table~\ref{t:17}), which contribute the second term; the classes $U_{2i, 4s-2i, 2r}\s ep$ for each $i=0,..., s-1$ and $p\in \mathcal{P}(2i,4s-2i)$, and for $i=s$ and $p\in \mathcal{AP}(2s,2s)$. 
\end{proof}

\begin{theorem}  The rank series of the cohomology ring of $\mathbf{A_{2r+1}}$ for $r$ odd is
\[
    \mathfrak{b}_t(d) + t^{d+1}\sum_{j=0}^{s-1}t^{4(2s-j)}\mathfrak{p}_t(2j+1, 4s-2j) + 
    \]
    \[
    + t^{8s+2r+1}\mathfrak{s}_t(2s,2s) + t^{8s+2r+1}\sum_{i=0}^{s-1}\mathfrak{p}_t(2i,4s-2i).
\]
\end{theorem}
\begin{proof} The cohomology algebra of $\mathbf{A_{2r+1}}$ is generated by $\ker d_1^{0,*}$; the classes $\tau_{2s-j,p}$ for $j=0, ...s-1$; and the cohomology classes 
\[
U^+_{2i, 4s-2i, 2r+1}\s ep- U^{-}_{2i, 4s-2i, 2r+1}\s ep
\] 
for each non-negative integer $i<s$ and $p\in \mathcal{P}(2i, 4s-2i)$; and for $i=s$ and $p\in \mathcal{SP}(2s,2s)$. 
\end{proof}

\begin{theorem}  The rank series of the cohomology ring of $\mathbf{A_{2r}}$ for $r$ even is
\[
    \mathfrak{b}_t(d) + t^{d+1}\sum_{j=0}^{s-1}t^{4(2s-j)}\mathfrak{p}_t(2j+1, 4s-2j) +  
    \]
    \[ 
    +t^{4s+2r} \mathfrak{s}_t(2s,2s) + t^{4s+2r}\sum_{a=0}^{2s-1} \mathfrak{p}_t(a,4s-a).
\]
\end{theorem}
\begin{proof} The cohomology algebra of $\mathbf{A_{2r+1}}$ is generated by $\ker d_1^{0,*}$; the classes $\tau_{2s-j,p}$ for $j=0, ...s-1$;  and the cohomology classes 
$U_{a,4s-a,2r}\s p$ for each $a=0,..., 2s-1$ and $p\in \mathcal{P}(a,4s-a)$; and for $a=2s$ and $p\in \mathcal{SP}(2s,2s)$.  

\end{proof}

\begin{theorem}  The rank series of the cohomology ring of $\mathbf{A_{2r+1}}$ for $r$ even is
\[
    \mathfrak{b}_t(d) + t^{d+1}\sum_{j=0}^{s-1}t^{4(2s-j)}\mathfrak{p}_t(2j+1, 4s-2j) + 
    \]
    \[
    + t^{4s+2r+1}\mathfrak{a}_r(2s,2s)+ t^{4s+2r+1}\sum_{a=0}^{2s-1} \mathfrak{p}_t(a,4s-a).
\]
\end{theorem}
\begin{proof} The cohomology algebra of $\mathbf{A_{2r+1}}$ is generated by $\ker d_1^{0,*}$; the classes $\tau_{2s-j,p}$ for $j=0, ...s-1$;  and the cohomology classes 
\[
U^+_{a, 4s-a, 4r+1}\s p- U^-_{a, 4s-a, 4r+1}\s p
\] 
for each $a=0,..., 2s-1$ and $p\in \mathcal{P}(a, 4s-a)$; and for $a=2s$ and each $p\in \mathcal{AP}(2s,2s)$. 

\end{proof}

\begin{theorem}  The rank series of the cohomology ring of $\mathbf{A_{\infty}}$ is 
\[
    \mathfrak{b}_t(d)+ t^{d+1}\sum_{j=0}^{s-1}t^{4(2s-j)}\mathfrak{p}_t(2j+1, 4s-2j).
\]
\end{theorem}
\begin{proof} 
The cohomology ring of $\mathbf{A_{\infty}}$ is generated by $\ker d_1^{0,*}$, and the classes $\tau_{2s-j,p}$ for $j=0,..., s-1$.
\end{proof}

%%%%%%%%%%%%%%%%%%%%%%%%%%%%%%%%%%%%%%%%%%%%%%%%%%%%%%%%%%%%%%%%%%%%%%%%%%%
%%%%%%%%%%%%%%%%%%%%%%%%%%%%%%%%%%%%%%%%%%%%%%%%%%%%%%%%%%%%%%%%%%%%%%%%%%%
%%%%%%%%%%%%%%%%%%%%%%%%%%%%%%%%%%%%%%%%%%%%%%%%%%%%%%%%%%%%%%%%%%%%%%%%%%%
\section{The cohomology group of $\mathbf{A_r}$ for $d=4s+1$ and $d=4s+3$}

Let $d'$ denote the number $\frac{d+1}{2}$.

\begin{theorem} For $d=4s+1$ the rank series of the cohomology ring of $\mathbf{A_1}$ is
\[
   \mathfrak{b}_t(d)+ t^{d+1}\mathfrak{a}_t(d',d')+
\]
\[
   + t^{d+1}\sum_{a=0}^{d'-1}\mathfrak{p}_t(a,d+1-a)+
   t^{2d+2}\sum_{a=0}^{d'-1/2}\mathfrak{p}_t(2a,d+1-2a).
\]
In the case $d=4s+3$ there is an additional term $t^{2d+2}\mathfrak{s}(d',d')$.
\end{theorem}  
\begin{proof} The cohomology algebra of $\mathbf{A_{1}}$ is generated by $\ker d_1^{0,*}$; and $U_{d', d',1}\s p$ for $p\in \mathcal{AP}(d', d')$; and $U_{a, d+1-a,1}\s p$ for each $a=0, ..., d'-1$ and $p\in \P(a, d+1-a)$; and the classes $I_{Q,a}$ for each even $a$. In the case $d=4s+3$ there are also classes $I_Q$. 
\end{proof}

\begin{theorem} The rank series of the cohomology ring of $\mathbf{A_{2r}}$ for $r$ odd is
\[
   \mathfrak{b}_t(d)+ t^{d+1}\mathfrak{a}_t(d',d')+
\]
\[
   + t^{2d+2}\sum_{a=0}^{d'-1/2}\mathfrak{p}_t(2a,d+1-2a)+t^{d+1}\sum_{j=0}^{s}t^{4(d'-j)}\mathfrak{p}(2j,d+1-j).
\]
In the case $d=4s+3$ there is an additional term $t^{2d+2}\mathfrak{s}(d',d')$.
\end{theorem}  
\begin{proof} The cohomology algebra of $\mathbf{A_{2r}}$ for $r$ odd is generated by $\ker d_1^{0,*}$; the cohomology classes $\sigma_Q$, $I_{Q,a}$,  and $\tau_{d'-s,p},...,\tau_{d',p}$. In the case $d=4s+3$ there are also classes $I_Q$.
\end{proof}

\begin{theorem} The rank series of the cohomology ring of $\mathbf{A_{2r+1}}$ for $r$ odd is
\[
   \mathfrak{b}_t(d)+ t^{d+1}\mathfrak{a}_t(d',d')+
\]
\[
   + t^{2d+2}\sum_{a=0}^{d'-1/2}\mathfrak{p}_t(2a,d+1-2a)+t^{d+1}\sum_{j=0}^{s}t^{4(d'-j)}\mathfrak{p}(2j,d+1-j).
\]
In the case $d=4s+3$ there is an additional term $t^{2d+2}\mathfrak{s}(d',d')$.
\end{theorem}  
\begin{proof} The cohomology algebra of $\mathbf{A_{2r}}$ for $r$ odd is generated by $\ker d_1^{0,*}$; the cohomology classes $\sigma_Q$, $I_{Q,a}$,  and $\tau_{d'-s,p},...,\tau_{d',p}$. In the case $d=4s+3$ there are also classes $I_Q$.
\end{proof}

\begin{theorem} The rank series of the cohomology ring of $\mathbf{A_{2r}}$ for $r$ even is
\[
   \mathfrak{b}_t(d)+ t^{d+1}\mathfrak{a}_t(d',d')+    t^{2d+2}\sum_{a=0}^{d'-1/2}\mathfrak{p}_t(2a,d+1-2a) +
\]
\[
+t^{d+1}\sum_{j=0}^{s}t^{4(d'-j)}\mathfrak{p}(2j,d+1-j)
   +t^{d+2r}\sum_{a=0}^{d'-1}\mathfrak{p}_t(a, d-a).
\]
In the case $d=4s+3$ there is an additional term $t^{2d+2}\mathfrak{s}(d',d')$.
\end{theorem}  
\begin{proof} The cohomology algebra of $\mathbf{A_{2r+1}}$ for $r$ odd is generated by $\ker d_1^{0,*}$; the cohomology classes $\sigma_Q$, $I_{Q,a}$,  and $\tau_{d'-s,p},...,\tau_{d',p}$; and $U_{a, d-a,2r}\s p$ for each $a=0, ..., d'-1$ and $p\in \P(a, d-a)$. In the case $d=4s+3$ there are also classes $I_Q$. 
\end{proof}

\begin{theorem} The rank series of the cohomology ring of $\mathbf{A_{2r+1}}$ for $r$ even is
\[
   \mathfrak{b}_t(d)+ t^{d+1}\mathfrak{a}_t(d',d')+   t^{2d+2}\sum_{a=0}^{d'-1/2}\mathfrak{p}_t(2a,d+1-2a) +
\]
\[
+t^{d+1}\sum_{j=0}^{s}t^{4(d'-j)}\mathfrak{p}(2j,d+1-j)
   +t^{d+2r+1}\sum_{a=0}^{d'-1}\mathfrak{p}_t(a, d-a).
\]
In the case $d=4s+3$ there is an additional term $t^{2d+2}\mathfrak{s}(d',d')]$.
\end{theorem}  
\begin{proof} The cohomology algebra of $\mathbf{A_{2r+1}}$ for $r$ odd is generated by $\ker d_1^{0,*}$; the cohomology classes $\sigma_Q$, $I_{Q,a}$,  and $\tau_{d'-s,p},...,\tau_{d',p}$; and $U^+_{a,d-a, 2r+1}\s p-U^-_{a,d-a,2r+1}\s p$ for each $a=0,..., d'-1$ and $p\in \P(a,d-a)$.  In the case $d=4s+3$ there are also classes $I_Q$.
\end{proof}  

\begin{theorem}\label{th:20.6} The rank series of the cohomology ring of $\mathbf{A_{\infty}}$ is
\[
   \mathfrak{b}_t(d)+ t^{d+1}\mathfrak{a}_t(d',d')+
\]
\[
   + t^{2d+2}\sum_{a=0}^{d'-1/2}\mathfrak{p}_t(2a,d+1-2a)+t^{d+1}\sum_{j=0}^{s}t^{4(d'-j)}\mathfrak{p}(2j,d+1-j)
\]
In the case $d=4s+3$ there is an additional term $t^{2d+2}\mathfrak{s}(d',d')$.
\end{theorem}  
\begin{proof} The cohomology algebra of $\mathbf{A_{2r}}$ for $r$ odd is generated by $\ker d_1^{0,*}$; the cohomology classes $\sigma_Q$, $I_{Q,a}$,  and $\tau_{d'-s,p},...,\tau_{d',p}$. In the case $d=4s+3$ there are also classes $I_Q$.
\end{proof}

%%%%%%%%%%%%%%%%%%%%%%%%%%%%%%%%%%%%%%%%%%%%%%%%%%%%%%%%%%%%%%%%%%%%%%%%%%%%%%%%%
%%%%%%%%%%%%%%%%%%%%%%%%%%%%%%%%%%%%%%%%%%%%%%%%%%%%%%%%%%%%%%%%%%%%%%%%%%%%%%%%%
%%%%%%%%%%%%%%%%%%%%%%%%%%%%%%%%%%%%%%%%%%%%%%%%%%%%%%%%%%%%%%%%%%%%%%%%%%%%%%%%%
%%%%%%%%%%%%%%%%%%%%%%%%%%%%%%%%%%%%%%%%%%%%%%%%%%%%%%%%%%%%%%%%%%%%%%%%%%%%%%%%%
%%%%%%%%%%%%%%%%%%%%%%%%%%%%%%%%%%%%%%%%%%%%%%%%%%%%%%%%%%%%%%%%%%%%%%%%%%%%%%%%%
\section{The cohomology group of $\mathbf{A_r}$ for $d=4s+2$}

\begin{theorem} The rank series of the cohomology ring of $\mathbf{A_1}$ is
\[
   t^{d+1}\mathfrak{p}_t(0, d+1)]+t^{d+1}\sum_{a=1}^{d/2}\mathfrak{p}_t(a, d+1-a).
\]
\end{theorem}  
\begin{proof}  By Remark~\ref{r:17.3}, for $d$ even the spectrum $\mathbf{A_1}$ splits as
\[
     \mathbf{A_1}=[\mathbf{A_1(0,d+1)}]\vee [\mathbf{A_1(1,d)/A_0}] \vee \cdots \vee [\mathbf{A_1(d/2,d/2+1)/A_0}].
\] 
\end{proof}

\begin{theorem} The rank series of the cohomology ring of $\mathbf{A_{2r}}$ for $r$ odd is
\[
    \mathfrak{b}_t(d)+ t^{d+1}\sum_{j=0}^{s}t^{4(2s+1-j)}\mathfrak{p}(2j+1,d-2j) +
    t^{2d+2r}\sum_{i=0}^{s} \mathfrak{p}_t(2i,d-2i).
\]
\end{theorem}
\begin{proof} The cohomology algebra of $\mathbf{A_{2r}}$ is generated by $\ker d_1^{0,*}$, which contributes the polynomial $\mathfrak{b}_t(d)$; the classes $\tau_{2s-j+1,p}$ for $j=0, ...s$, which contributes the second term; the classes $U_{2i, d-2i, 2r}\s ep$ for each $i=0,..., s$ and $p\in \mathcal{P}(2i,d-2i)$. 
\end{proof}

\begin{theorem}  The rank series of the cohomology ring of $\mathbf{A_{2r+1}}$ for $r$ odd is
\[
    \mathfrak{b}_t(d)+ t^{d+1}\sum_{j=0}^{s}t^{4(2s+1-j)}\mathfrak{p}(2j+1,d-2j)+ t^{2d+2r+1}\sum_{i=0}^{s}\mathfrak{p}_t(2i,d-2i).
\]
\end{theorem}
\begin{proof} The cohomology algebra of $\mathbf{A_{2r+1}}$ is generated by $\ker d_1^{0,*}$; the classes $\tau_{2s-j+1,p}$ for $j=0, ...s$; and the cohomology classes 
\[
U^+_{2i, d-2i, 2r+1}\s ep- U^{-}_{2i, d-2i, 2r+1}\s ep
\] 
for each non-negative integer $i\le s$ and $p\in \mathcal{P}(2i, d-2i)$. 
\end{proof}

\begin{theorem}  The rank series of the cohomology ring of $\mathbf{A_{2r}}$ for $r$ even is
\[
    \mathfrak{b}_t(d)+ t^{d+1}\sum_{j=0}^{s}t^{4(2s+1-j)}\mathfrak{p}(2j+1,d-2j)+
    \]
    \[
    +t^{d+2r}\mathfrak{s}_t(2s+1,2s+1)  +t^{d+2r}\sum_{a=0}^{2s} \mathfrak{p}_t(a,d-a).
\]
\end{theorem}
\begin{proof} The cohomology algebra of $\mathbf{A_{2r+1}}$ is generated by $\ker d_1^{0,*}$; the classes $\tau_{2s-j+1,p}$ for $j=0, ...s$;  and the cohomology classes 
$U_{a,d-a,2r}\s p$ for each $a=0,..., 2s$ and $p\in \mathcal{P}(a,d-a)$; and for $a=2s+1$ and $p\in \mathcal{SP}(2s+1,2s+1)$.  

\end{proof}

\begin{theorem}  The rank series of the cohomology ring of $\mathbf{A_{2r+1}}$ for $r$ even is
\[
    \mathfrak{b}_t(d)+ t^{d+1}\sum_{j=0}^{s}t^{4(2s+1-j)}\mathfrak{p}(2j+1,d-2j)+
    \]
    \[
    t^{d+2r+1}\mathfrak{a}_t(2s+1,2s+1)+ t^{d+2r+1}\sum_{a=0}^{2s} \mathfrak{p}_t(a,d-a).
\]
\end{theorem}
\begin{proof} The cohomology algebra of $\mathbf{A_{2r+1}}$ is generated by $\ker d_1^{0,*}$; the classes $\tau_{2s-j+1,p}$ for $j=0, ...s$;  and the cohomology classes 
\[
U^+_{a, d-a, 4r+1}\s p- U^-_{a, d-a, 4r+1}\s p
\] 
for each $a=0,..., 2s$ and $p\in \mathcal{P}(a, d-a)$; and for $a=2s+1$ and each $p\in \mathcal{AP}(2s+1,2s+1)$. 

\end{proof}

\begin{theorem}\label{th:4s+2}  The rank series of the cohomology ring of $\mathbf{A_{\infty}}$ is 
\[
    \mathfrak{b}_t(d)+ t^{d+1}\sum_{j=0}^{s}t^{4(2s+1-j)}\mathfrak{p}(2j+1,d-2j).
    \]
\end{theorem}
\begin{proof} The cohomology algebra of $\mathbf{A_{\infty}}$ is generated by $\ker d_1^{0,*}$, and the cohomology classes $\tau_{2s-j+1,p}$ for $j=0, ...s$.
\end{proof}

\begin{table} 
\caption{Linear generators of $\mathrm{ker}(d_1^{1,*})$.}\label{t:17}
\begin{tabular}{|c|c|}
\hline
 $d=4s$ &  $\tau_{2s-j,p}=\sum_{a,b}\ (-1)^a U_{a,b,1}\s p$\\
        &  where \\
        & $p\in p'_{2s-j} \P(2j+1,4s-2j)$, \\
        & $j=0,...,s-1; \ a=0,...,2j+1;\  b=4s+1-a$ \\ \cline{2-2}
        &  $\sigma_Q=\sum_{a,b}\ (-1)^a U_{a,b,1}\s Q$ \\ 
        &  where \\
        &  $Q\in \SP(2s,2s);\ a=0,...,2s; \ b=4s+1-a$ \\
\hline
$d=4s+1$&  $I_{Q,a}= U_{a,b,1}\s Q$ \\
        &  where \\
        &  $Q\in e\P(a,b), \ a=0, 2, 4, ...,2s,\  b=4s+2-a$ \\ \cline{2-2}
        &  $\tau_{2s+1-j,p}=\sum_{a,b}\ (-1)^a U_{a,b,1}\s p$\\
        &  where \\
        & $p\in p'_{2s+1-j} \P(2j, 4s+2-j)$, \\
        &  $j=0,...,s; \ a=0,...,2j;\  b=4s+2-a$ \\ \cline{2-2}
        &  $\sigma_Q=\sum_{a,b}\ (-1)^a U_{a,b,1}\s Q$ \\ 
        &  where \\
        &  $Q\in \AP(2s+1,2s+1),\ a=0,...,2s+1; \ b=4s+2-a$ \\
\hline
 $d=4s+2$ &  $\tau_{2s+1-j,p}=\sum_{a,b}\ (-1)^a U_{a,b,1}\s p$\\
        &  where \\
        & $p\in p'_{2s+1-j} \P(2j+1,4s+2-2j)$, \\
        & $j=0,...,s; \ a=0,...,2j+1;\  b=4s+3-a$ \\ \cline{2-2}
        &  $\sigma_Q=\sum_{a,b}\ (-1)^a U_{a,b,1}\s Q$ \\ 
        &  where \\
        &  $Q\in \SP(2s+1, 2s+1),\ a=0,...,2s+1; \ b=4s+3-a$ \\
\hline
 $d=4s+3$&  $I_{Q,a}= U_{a,b,1}\s Q$ \\
        &  where \\
        &  $Q\in e\P(a,b), \ a=0, 2, 4, ...,2s,\  b=4s+4-a$ \\ \cline{2-2}
        &  $I_{Q}= U_{2s+2,2s+2,1}\s Q$ \\
        &  where \\
        &  $Q\in e\SP(2s+2,2s+2)$ \\ \cline{2-2}
        &  $\tau_{2s+2-j,p}=\sum_{a,b}\ (-1)^a U_{a,b,1}\s p$\\
        &  where \\
        & $p\in p'_{2s-j+2} \P(2j,4s+4-2j)$, \\
        & $j=0,...,s; \ a=0,...,2j;\  b=4s+4-a$ \\ \cline{2-2}
        &  $\sigma_Q=\sum_{a,b}\ (-1)^a U_{a,b,1}\s Q$ \\ 
        &  where \\
        &  $Q\in \AP(2s+2,2s+2), \ a=0,...,2s+2; \ b=4s+4-a$ \\
\hline
\end{tabular}
\end{table}

%\item ($a=b=c$) \quad We have $d_1^{1,*}: U_{c,c,1}\s e_{a,b} \mapsto %U_{c+1,c,2}\s e_{a,b} - U_{c,c+1,2}\s e_{a,b}$.
%\end{itemize}

\addcontentsline{toc}{part}{Cobordism groups of Morin maps}

\section{Rational homotopy theory}\label{s:rational}

Let $X$ be a simply connected space with $H_*(X; \Q)$ of finite type. Then, by the Hopf theorem, $H^*(\Omega X, \Q)$ is a free graded commutative algebra and, by the Cartan-Serre and Milnor-Moore theorems, the Hurewicz homomorphism 
\[
   \pi_*(\Omega X)\otimes \Q \longrightarrow P_*(\Omega X)
\]
is an isomorphism onto the primitive subspace for $H_*(\Omega X)$, e.g., see \cite[Theorem 10.7]{DNF} where the primitive elements are identified with those classes whose Kronecker pairing vanishes with any cohomology class $x$ that decomposes as a product $x=yz$ of cohomology classes in positive degrees.  Furthermore, the space $\Omega X$ has a rational homotopy type of a weak product of spaces $K(\Q, n)$. More precisely, there are rational homotopy equivalences
\[
     K(\Q, 2s+i)\simeq_{\Q} S^{2s+1} \qquad \mathrm{and} \qquad K(\Q, 2s)\simeq_{\Q} \Omega S^{2s+1}. 
\] 
Let $a_j\colon S^{2m_j+1}\to X$ and $b_i\colon S^{2n_i+2}\to X$ be maps representing the homotopy classes in a basis of $\pi_*(X)\otimes \Q$. Then the adjoints of the $b_i$'s and the loops of the $a_j$'s give rise to a continuous map 
\[
\prod K(\Q, 2n_i+1)\times \prod K(\Q, 2m_j) \simeq_\Q    \prod S^{2n_i+1} \times \prod \Omega S^{2m_j+1} \longrightarrow \Omega X, 
\]
from a weak product of infinitely many spaces. This map induces an isomorphism of rational homotopy groups. It also induces an isomorphism of rational homology groups (e.g., see the textbook \cite{FHT}).

\section{Definition of cobordism groups of Morin maps}\label{s:def}

For this section we fix a value for $r=0,1,...,\infty$.

A proper map $F: M\to N\times [0,1]$ is said to be a {\it cobordism} of two proper $\A_r$-maps $f_i: M_i\to N\times \{i\}$, with $i=0,1$, if $\partial M=M_0\sqcup M_1$, the map $F|M_i: M_i\to N\times\{i\}$ coincides with $f_i$ for $i=0,1$, and the restriction of $f$ to collar neighborhoods of $M_0$ and $M_1$ in $M$ can be identified with the disjoint union of suspensions 
\[
    f_0\times \id\colon M_0\times [0, \delta)\longrightarrow N\times [0, \delta) 
\]    
and 
\[
    f_1\times \id\colon M_1\times (1-\delta, 1]\longrightarrow N\times (1-\delta, 1], 
\]    
where $\id$ is the identity map of an appropriate space, and $\delta>0$ is a sufficiently small real number. The classes of cobordant (proper) $\A_r$-maps of dimension $d$ into a fixed manifold $N$ form a semigroup $H^{d}(N;\A_r)$ with respect to the operation which in terms of representatives is given by taking the disjoint union of maps. It turns out that the semigroup $H^{d}(N;\A_r)$ is a group. 

\begin{theorem}\label{th:9.0} For each $d\ge 0$ there exists an infinite loop space $\Omega^{\infty}\mathbf{A_r}=\Omega^{\infty}\mathbf{A_r}(d)$ such that for each manifold $N$, the cobordism group $H^{d}(N;\A_r)$ is isomorphic to the group $[N, \Omega^{\infty-d}\mathbf{A_r}]$ of homotopy classes of maps. 
\end{theorem}

In the stated form Theorem~\ref{th:9.0} is proved by the author \cite{Sa10}. With additional assumptions that $N$ is closed and $\dim N\ge 2$, it is proved in the paper \cite{Sa} and a closely related theorem is proved by Ando in \cite{An4}. In the case $d<0$, the theorem for cobordism groups of $\A_r$-maps is also true; its versions are due to Ando~\cite{An4}, Szucs~\cite{Sz} and the author~\cite{Sa}, \cite{Sa10}. The spectrum $\mathbf{A_r}$ is a counterpart of the spectra constructed by Eliashberg, who was motivated by Lagrangian and Legendrian immersions~\cite{El} (note, however, that in contrast to equivariant spectra in \cite{El} and \cite{An4}, the spectrum $\Omega^{\infty}\mathbf{A_r}$ defines a cobordism theory on the category of topological spaces).

\begin{remark} Cobordism groups of $\A_r$-maps should not be confused with bordism groups of $\A_r$-maps. These two groups are \underline{different}. The definition of bordism groups is obtained from the definition of cobordism groups by replacing ``proper $\A_r$-maps" by ``$\A_r$-maps of closed manifolds" (for details see the paper \cite{Sa}). We do not claim that the bordism group of $\mathcal R$-maps into a manifold $N$ is isomorphic to $[N, \Omega^{\infty-d}\mathbf{A_r}]$. 
\end{remark}

An $\A_r$-map $f\colon M\to N$ is said to be \emph{orientable} if the stable vector bundle $f^*TN\ominus TM$ over $M$ is orientable. A choice of orientation on $f^*TN\ominus TM$ determines an \emph{orientation} on the $\A_r$-map $f$. The \emph{oriented cobordism group of $\A_r$-maps} is defined similarly to the cobordism group of $\A_r$-maps by replacing ``$\A_r$-maps" by ``oriented $\A_r$-maps".

\section{Computation of rational cobordism groups of $\A_r$-maps} \label{s:final}

From evaluation of the homology groups of $E_1^{*,*}$, it follows that the groups $E_2^{i,*}$, with $i>1$, are trivial. Hence the Kazarian spectral sequence for $\A_r$ singularities collapses at the second term.  

\begin{corollary}\label{c:f.1}
For each dimension $n$ the vector space $H_n(\mathbf{A_r})$ is finite dimensional. 
\end{corollary}
\begin{proof} We have seen that $H^n(\mathbf{A_r})$ is of finite rank. By the Universal Coefficient Theorem, the groups $H_n(\mathbf{A_r})$ are also of finite ranks.  
\end{proof}

Let $\Omega^{\infty}\mathbf{A_r}$ denote the classifying loop space of $\A_r$-maps. Then, by considering spaces $\Omega(\Omega^{t-1}[\mathbf{A_r}]_t)$, we conclude that $H^*(\Omega^{\infty}\mathbf{A_r})$ is a free graded commutative algebra, and its generators are in bijective correspondence with classes in 
\[
    \pi_*(\Omega^{\infty}\mathbf{A_r})\otimes \Q\approx \pi_*(\mathbf{A_r})\otimes \Q.
\]
On the other hand,  the Hurewicz homomorphism 
\[
    \pi_n(\mathbf{A_r})\otimes \Q\longrightarrow H_n(\mathbf{A_r}) 
\]
is an isomorphism (e.g., see \cite[Corollary 10.11]{DNF}). Thus the map $\prod K(\Q, n_i)\to \Omega^{\infty}\mathbf{A_r}$ induces an isomorphism of rational homology groups, where the product is indexed by a set of generators of the vector space $\pi_*(\Omega^{\infty}\mathbf{A})\otimes \Q$. Thus we deduce the following theorem.

%Let us recall that for a connected space $X$ with $H_*(X)$ of finite type, there is an isomorphism 
%\[
%      \tilde H_*(X) = \pi_*(\Omega^{\infty}\Sigma^{\infty}X)\otimes \Q.
%\]
%Hence the ranks of the coefficients of oriented Morin cobordism groups can be readily computed using Table~\ref{t:17}. %(see Theorem~\ref{th:11.1}). 

%On the other hand, the infinite loop space $\Omega^{\infty}\mathbf{A}$ is a product of finitely many spaces of type %$K(\Q, n_i)$ indexed by a set of generators of the vector space $\pi_*(\Omega^{\infty}\mathbf{A})\otimes \Q$. Hence our %computations immediately yield the rational cohomology groups of $\Omega^{\infty}\mathbf{A}$ (see %Theorem~\ref{th:main}). 

\begin{theorem}\label{th:char} For each $r\ge 1$, the rational oriented cobordism class of an $\A_r$-map is completely determined by the characteristic classes described in Theorem~\ref{th:main}. 
\end{theorem}

\begin{theorem} For each $r\ge 1$, the rational oriented cobordism class of an $\A_r$-map into a Euclidean space is completely 
determined by its characteristic numbers. 
\end{theorem}
\begin{proof} On one hand side, the rational oriented cobordism class $x$ of an $\A_r$-map into a Euclidean space is determined by the corresponding homotopy class $x'$ in $\pi_*(\Omega^{\infty}\mathbf{A_r})\otimes \Q$. On the other hand, the image $x$ of $x'$ under the Hurewecz homomorphism is determined by the values $\left\langle p,x'\right\rangle$ of cohomology classes $p\in H^*(\Omega^{\infty}\mathbf{A_r})$. Finally, we observe that these values coincide with characteristic numbers of $\A_r$-maps into a Euclidean space. 
\end{proof}

It also immediately follows that for rational oriented cobordism groups ``the higher singularities are not necessary''. 

%\begin{corollary} In the case of maps of negative odd codimension every Morin %map is rationally cobordant to a fold map. 
%\end{corollary}

\begin{corollary}\label{c:main} Every Morin map is rationally cobordant through Morin maps to a fold map. 
\end{corollary}

\begin{landscape}
\begin{table}
\caption{Differential $d^{1,*}_1$ in the case $d=4s$.}
\[
\xymatrix{
  &  U_{2s, 2s+1, 1}\s \Q[p_1,\dots, p_s, p_1', ..., p'_{s}] \ar[dr]  \ar[r]  &(eU_{2s, 2s, 2}\oplus U_{2s, 2s, 2})\s \mathcal{AP}[p_1,\dots, p_s, p_1', ..., p'_{s}]            & \\     
  &  U_{2s-1, 2s+2, 1}\s \Q[p_1,\dots, p_{s-1}, p_1', ..., p'_{s+1}] \ar[dr]  \ar[r]  & U_{2s-1, 2s+1, 2}\s \Q[p_1,\dots, p_{s-1}, p_1', ..., p'_{s}]             & \\     
  &  U_{2s-2, 2s+3, 1}\s \Q[p_1, \dots, p_{s-1}, p_1', ..., p'_{s+1}] \ar[dr]  \ar[r]  &(eU_{2s-2, 2s+2, 2}\oplus U_{2s-2, 2s+2, 2})\s \Q[p_1, \dots, p_{s-1}, p_1', ..., p'_{s+1}]        & \\     
  &  U_{2s-3, 2s+4, 1}\s \Q[p_1, \dots, p_{s-2}, p_1', ..., p'_{s+2}] \ar[dr]  \ar[r] & U_{2s-3, 2s+3, 2}\s \Q[p_1, \dots, p_{s-2}, p_1', ..., p'_{s+1}]                     & \\
  &  \cdots \ar[dr]  \ar[r]  &\cdots            & \\     
  &  U_{3, 4s-2, 1}\s \Q[p_1, p_1', ..., p'_{2s-1}] \ar[dr]  \ar[r]  &U_{3, 4s-3, 2}\s \Q[p_1, p_1', ..., p'_{2s-2}]             & \\     
  &  U_{2, 4s-1, 1}\s \Q[p_1, p_1', ..., p'_{2s-1}] \ar[dr]  \ar[r]  &(eU_{2, 4s-2, 2}\oplus U_{2, 4s-2, 2})\s \Q[p_1, p_1', ..., p'_{2s-1}]             & \\     
  &  U_{1, 4s, 1}\s \Q[p_1', ..., p'_{2s}] \ar[dr]  \ar[r]  & U_{1, 4s-1, 2}\s \Q[p_1', ..., p'_{2s-1}]             & \\     
  &  U_{0, 4s+1, 1}\s \Q[p_1', ..., p'_{2s}]   \ar[r] & (eU_{0, 4s, 2} \oplus U_{0, 4s, 2})\s \Q[p_1', ..., p'_{2s}]                     & 
  }
\]
\label{table:10}
\end{table}
\end{landscape}

\begin{landscape}
\begin{table}
\caption{Differential $d^{1,*}_1$ in the case $d=4s+1$.}
\[
\xymatrix{
  &  U_{2s+1, 2s+1, 1}\s \AP[p_1,\dots, p_s, p_1', ..., p'_{s}] \ar[dr]  \ar[r]  &  0           & \\     
  &  (eU_{2s, 2s+2, 1}\oplus U_{2s, 2s+2, 1})\s \Q[p_1,\dots, p_s, p_1', ..., p'_{s+1}] \ar[dr]  \ar[r]  
        &U_{2s, 2s+1, 2}\s \Q[p_1,\dots, p_s, p_1', ..., p'_{s}]            & \\     
  &  U_{2s-1, 2s+3, 1}\s \Q[p_1,\dots, p_{s-1}, p_1', ..., p'_{s+1}] \ar[dr]  \ar[r]  
        &U_{2s-1, 2s+2, 2}\s \Q[p_1,\dots, p_{s-1}, p_1', ..., p'_{s+1}]             & \\     
  &  (eU_{2s-2, 2s+4, 1}\oplus U_{2s-2, 2s+4, 1})\s \Q[p_1, \dots, p_{s-1}, p_1', ..., p'_{s+2}] \ar[dr]  \ar[r]                 &U_{2s-2, 2s+3, 2}\s \Q[p_1, \dots, p_{s-1}, p_1', ..., p'_{s+1}]        & \\     
  &  U_{2s-3, 2s+5, 1}\s \Q[p_1, \dots, p_{s-2}, p_1', ..., p'_{s+2}] \ar[dr]  \ar[r] 
        &U_{2s-3, 2s+4, 2}\s \Q[p_1, \dots, p_{s-2}, p_1', ..., p'_{s+2}]                     & \\
  &  \cdots \ar[dr]  \ar[r]  &\cdots            & \\     
  &  U_{3, 4s-1, 1}\s \Q[p_1, p_1', ..., p'_{2s-1}] \ar[dr]  \ar[r]  
        &U_{3, 4s-2, 2}\s \Q[p_1, p_1', ..., p'_{2s-1}]             & \\     
  &  (eU_{2, 4s, 1}\oplus U_{2, 4s, 1})\s \Q[p_1, p_1', ..., p'_{2s}] \ar[dr]  \ar[r]  
        &U_{2, 4s-1, 2}\s \Q[p_1, p_1', ..., p'_{2s-1}]             & \\     
  &  U_{1, 4s+1, 1}\s \Q[p_1', ..., p'_{2s}] \ar[dr]  \ar[r]  
        &U_{1, 4s, 2}\s \Q[p_1', ..., p'_{2s}]             & \\     
  &  (eU_{0, 4s+2, 1}\oplus U_{0, 4s+2, 1})\s \Q[p_1', ..., p'_{2s+1}]   \ar[r] 
        &U_{0, 4s+1, 2}\s \Q[p_1', ..., p'_{2s}]                     & 
  }
\]
\end{table}
\end{landscape}

\begin{landscape}
\begin{table}\label{t:12}
\caption{Differential $d^{1,*}_1$ in the case $d=4s+2$.}
\[
\xymatrix{
  &  U_{2s+1, 2s+2, 1}\s \Q[p_1,\dots, p_s, p_1', ..., p'_{s+1}] \ar[dr]  \ar[r]  
        &U_{2s+1, 2s+1, 2}\s \mathcal{AP}[p_1,\dots, p_s, p_1', ..., p'_{s}]            & \\     
  &  U_{2s, 2s+3, 1}\s \Q[p_1,\dots, p_{s}, p_1', ..., p'_{s+1}] \ar[dr]  \ar[r]  
        &(eU_{2s, 2s+2, 2}\oplus U_{2s, 2s+2, 2})\s \Q[p_1,\dots, p_{s}, p_1', ..., p'_{s+1}]             & \\     
  &  U_{2s-1, 2s+4, 1}\s \Q[p_1, \dots, p_{s-1}, p_1', ..., p'_{s+2}] \ar[dr]  \ar[r]  
        &U_{2s-1, 2s+3, 2}\s \Q[p_1, \dots, p_{s-1}, p_1', ..., p'_{s+1}]        & \\     
  &  U_{2s-2, 2s+5, 1}\s \Q[p_1, \dots, p_{s-1}, p_1', ..., p'_{s+2}] \ar[dr]  \ar[r] 
        &(eU_{2s-2, 2s+4, 2}\oplus U_{2s-2, 2s+4, 2})\s \Q[p_1, \dots, p_{s-1}, p_1', ..., p'_{s+2}]     & \\
  &  \cdots \ar[dr]  \ar[r]  &\cdots            & \\     
  &  U_{3, 4s, 1}\s \Q[p_1, p_1', ..., p'_{2s}] \ar[dr]  \ar[r]  
        &U_{3, 4s-1, 2}\s \Q[p_1, p_1', ..., p'_{2s-1}]             & \\     
  &  U_{2, 4s+1, 1}\s \Q[p_1, p_1', ..., p'_{2s}] \ar[dr]  \ar[r]  
        &(eU_{2, 4s, 2}\oplus U_{2, 4s, 2})\s \Q[p_1, p_1', ..., p'_{2s}]             & \\     
  &  U_{1, 4s+2, 1}\s \Q[p_1', ..., p'_{2s+1}] \ar[dr]  \ar[r]  
        & U_{1, 4s+1, 2}\s \Q[p_1', ..., p'_{2s}]             & \\     
  &  U_{0, 4s+3, 1}\s \Q[p_1', ..., p'_{2s+1}]   \ar[r] 
        & (eU_{0, 4s+2, 2} \oplus U_{0, 4s+2, 2})\s \Q[p_1', ..., p'_{2s+1}]                     & 
  }
\]
\end{table}
\end{landscape}

\begin{landscape}
\begin{table}
\caption{Differential $d^{1,*}_1$ in the case $d=4s+3$.}
\[
\xymatrix{
  &  U_{2s+2, 2s+2, 1}\s (\AP\oplus e\SP)[p_1,\dots, p_{s+1}, p_1', ..., p'_{s+1}] \ar[dr]  \ar[r]  &  0       & \\     
  &  U_{2s+1, 2s+3, 1}\s \Q[p_1,\dots, p_s, p_1', ..., p'_{s+1}] \ar[dr]  \ar[r]  
        &U_{2s+1, 2s+2, 2}\s \Q[p_1,\dots, p_s, p_1', ..., p'_{s+1}]            & \\     
  &  (eU_{2s, 2s+4, 1}\oplus U_{2s, 2s+4, 1})\s \Q[p_1,\dots, p_{s}, p_1', ..., p'_{s+2}] \ar[dr]  \ar[r]  
        &U_{2s, 2s+3, 2}\s \Q[p_1,\dots, p_{s}, p_1', ..., p'_{s+1}]             & \\     
  &  U_{2s-1, 2s+5, 1}\s \Q[p_1, \dots, p_{s-1}, p_1', ..., p'_{s+2}] \ar[dr]  \ar[r]                 
        &U_{2s-1, 2s+4, 2}\s \Q[p_1, \dots, p_{s-1}, p_1', ..., p'_{s+2}]        & \\     
  &  (eU_{2s-2, 2s+6, 1}\oplus U_{2s-2, 2s+6, 1})\s \Q[p_1, \dots, p_{s-1}, p_1', ..., p'_{s+3}] \ar[dr]  \ar[r] 
        &U_{2s-2, 2s+5, 2}\s \Q[p_1, \dots, p_{s-1}, p_1', ..., p'_{s+2}]                     & \\
  &  \cdots \ar[dr]  \ar[r]  &\cdots            & \\     
  &  U_{3, 4s+1, 1}\s \Q[p_1, p_1', ..., p'_{2s}] \ar[dr]  \ar[r]  
        &U_{3, 4s, 2}\s \Q[p_1, p_1', ..., p'_{2s}]             & \\     
  &  (eU_{2, 4s+2, 1}\oplus U_{2, 4s+2, 1})\s \Q[p_1, p_1', ..., p'_{2s+1}] \ar[dr]  \ar[r]  
        &U_{2, 4s+1, 2}\s \Q[p_1, p_1', ..., p'_{2s}]             & \\     
  &  U_{1, 4s+3, 1}\s \Q[p_1', ..., p'_{2s+1}] \ar[dr]  \ar[r]  
        &U_{1, 4s+2, 2}\s \Q[p_1', ..., p'_{2s+1}]             & \\     
  &  (eU_{0, 4s+4, 1}\oplus U_{0, 4s+4, 1}) \s \Q[p_1', ..., p'_{2s+2}]   \ar[r] 
        &U_{0, 4s+3, 2}\s \Q[p_1', ..., p'_{2s+1}]                     & 
  }
\]
\label{t:13}
\end{table}
\end{landscape}

\end{document}